\documentclass[reqno]{amsart}
\usepackage[utf8]{inputenc}
\usepackage{amsthm,amsfonts,amstext,amssymb,mathrsfs,amsmath,latexsym} 
\usepackage{cite} 
\usepackage{enumerate}
\usepackage{mdwlist}
\usepackage[abs]{overpic}

\usepackage{hyperref} 
\usepackage{comment}
\usepackage{tikz}
\usepackage{caption}
\usepackage{subcaption}
\usepackage{color}
\usepackage{esint}

\definecolor{red}{RGB}{255,0,0}
\definecolor{green}{RGB}{0,150,0}
\definecolor{blue}{RGB}{0,0,255}  

\usepackage{cleveref}
 
\crefname{equation}{equation}{equations}
\crefname{figure}{Figure}{Figures}

\newtheorem{thm}{Theorem}[section]
\newtheorem{prop}[thm]{Proposition}
\newtheorem{lem}[thm]{Lemma}
\newtheorem{cor}[thm]{Corollary}

\theoremstyle{definition}
\newtheorem{definition}[thm]{Definition} 
\newtheorem{example}[thm]{Example}

\theoremstyle{remark}
\newtheorem{remark}[thm]{Remark} 
\numberwithin{equation}{section} 

\DeclareMathOperator{\re}{Re}
\DeclareMathOperator{\im}{Im}
\renewcommand{\Re}{\mathop{\rm Re}}
\renewcommand{\Im}{\mathop{\rm Im}}
\DeclareMathOperator{\supp}{supp} 
\DeclareMathOperator{\discr}{Discr}
\DeclareMathOperator{\const}{const}
\DeclareMathOperator{\res}{Res}

\DeclareMathOperator{\card}{card}

\newcommand{\N}{\mathbb{N}}
\newcommand{\C}{\mathbb{C}}
\newcommand{\R}{\mathbb{R}}

\newcommand{\Q}{\mathbb{Q}}
\newcommand{\boh}{\mathit{o}}
\newcommand{\Boh}{\mathcal{O}}

\newcommand{\lebesgue}{\operatorname{mes}}

\title[Critical measures: structure of trajectories]{Critical measures for vector energy: global structure of trajectories of quadratic differentials}

\author[A. Mart\'{\i}nez-Finkelshtein]{Andrei Mart\'{\i}nez-Finkelshtein}

\address[AMF]{Department of Mathematics, University of Almer\'{\i}a, Almer\'{\i}a, Spain}

\email{andrei@ual.es}

\author[G.~Silva]{Guilherme L.~F.~Silva}

\address[GS]{KU Leuven, Department of Mathematics, Celestijnenlaan 200B bus 2400, B-3001 Leuven, Belgium
\newline
*Current address: Department of Mathematics, University of Michigan, 530 Church Street, Ann Arbor, MI 48109, USA}

\email{silvag@umich.edu}

 

\begin{document}

\begin{abstract} 
Saddle points of a vector logarithmic energy with a vector polynomial external field on the plane constitute the {\it vector-valued critical measures}, a notion that finds a natural motivation in several branches of analysis. We study in depth the case of measures $\vec \mu=(\mu_1, \mu_2,\mu_3)$ when the mutual interaction comprises both attracting and 
repelling forces.

For arbitrary vector polynomial external fields we establish general structural results about critical measures, such as their characterization in terms of an algebraic equation 
solved by an appropriate combination of their Cauchy transforms, and the symmetry properties (or the $S$-properties) exhibited by such measures. In consequence, we conclude that vector-valued critical measures are supported on a finite number of analytic arcs, that are trajectories of a quadratic differential globally defined on a three-sheeted Riemann surface. 
The complete description of the so-called critical graph for such a differential is the key to the construction of the critical measures.

We illustrate these connections studying in depth a one-parameter family of critical measures under the action of a cubic external field. This choice is motivated by the 
asymptotic analysis of a family of (non-hermitian) multiple orthogonal polynomials, that is subject of a forthcoming paper. Here we compute explicitly the Riemann surface and the 
corresponding quadratic differential, and analyze the dynamics of its critical graph as a function of the parameter, 
giving a detailed description of the occurring phase transitions. When projected back to the complex plane, this construction gives us the complete family of vector-valued critical measures, that in this context turn out to be vector-valued equilibrium measures.

\end{abstract}

\maketitle


\section{Introduction and statement of main results}

\subsection{Historical background} \label{sec:historical}

Statistical systems of many particles have been object of intense analysis for a long time, both from the perspective of physics and mathematics. The study of a particular type of 
interacting particle systems, the so-called determinantal point processes (and their cousins, Pfaffian processes), has been especially fruitful in the past thirty years. This 
success can be 
explained both by the ubiquitous character and flexibility of these models (describing the eigenvalues of several random matrix ensembles, non-intersecting diffusion paths, random 
growth models, random tilings, to mention a few) and the introduction of new tools, intrinsically related with the analytic theory of orthogonal polynomials and their 
generalizations. 

A common feature of these models is that either the joint probability density, the correlation functions, the normalization constant or the generating function can be expressed as 
a determinant (resp., a Pfaffian), and the right selection of the functions appearing in these determinants unveils the integrable structure of the underlying processes. Another 
unified property of these models is the possibility to put them in the framework of the so-called log gases, where the particles behave as charges on one or two dimensional sets, 
subject to the logarithmic interaction. 

The best known case is the spectrum of some unitary matrix ensembles \cite{deift_book,Deiftetal1999}, described in terms of classical families of orthogonal polynomials on the real
line. Their zeros, all real and simple, satisfy an electrostatic model that goes all the way back to Stieltjes \cite{Stieltjes1885}, and solve a minimization problem for the 
associated (logarithmic) energy. In other words, the zeros provide an equilibrium configuration on the conducting real line. Similar situation occurs for some non-intersecting 
paths models or  for the  six-vertex model in statistical mechanics \cite{bleher_liechty_book}, to mention some more examples. 

Another classical framework for orthogonal polynomials (orthogonality on the unit circle), developed in the seminal works of Szeg\H{o} \cite{szego_book}, is connected to the 
analysis of the Ising model  \cite{deift_its_krasovsky_ising,deift_its_krasovsky_fisher_hartwig}.

Further immediate generalizations of the problems above oblige us to leave the real line and extend the notion of orthogonality and the associated log gas models to the complex 
plane. 
The so-called \emph{non-hermitian orthogonal polynomials} appear naturally in approximation theory, as denominators of rational (e.g.~Pad\'e) approximants to analytic functions 
\cite{gonchar_rakhmanov_rato_rational_approximation,nikishin_sorokin_book}, or equivalently, in the study of continued fractions. Electrostatic models on a conducting plane are 
also associated with the polynomial solutions of certain linear differential equations \cite{martinez_rakhmanov,MR2647571}.

Recently, non-hermitian orthogonality found its way to areas with a much more ``modern'' flavor, playing the crucial role in the description of the 
rational solutions to  Painlevé equations \cite{balogh_bertola_bothner_generalized_vorobev_yablonski,bertola_bothner_painleve}, in theoretical physics  
\cite{alvarez_alonso_medina_1,alvarez_alonso_medina_2,bleher_deano_painleve_I,bleher_deano_partition_function} and in numerical analysis  
\cite{deano_kuijlaars_huybrechs_complex_orthogonal_polynomials}. 

Observe that due to analyticity, there is a certain freedom in the choice of the integration contour for the non-hermitian orthogonal polynomials, which means that the location of 
their zeros is a priori not clear. As it follows from the fundamental works of Nuttall \cite{MR1036606} and Stahl  
\cite{stahl_orthogonal_polynomials_complex_measures,stahl_orthogonal_polynomials_complex_weight_function} (see also the survey \cite{MR2963451}), the limiting zero counting measure 
of these polynomials can be characterized in terms of certain max-min problems of the logarithmic potential theory, where the 
energy is minimized on each compact set, but these minima are maximized over an appropriately chosen class of contours. These questions (whose simplest example is the Chebotarev’s 
problem \cite{MR2664566,Kuzmina1980} about the selection of sets of minimal logarithmic capacity) are closely related to the free boundary problems in the theory of partial 
differential equations.

Motivated by his study of convergence of Pad\'e approximants to functions with branch points, Stahl  
\cite{stahl_orthogonal_polynomials_complex_measures,stahl_orthogonal_polynomials_complex_weight_function} not only proved that the zeros of the associated non-hermitian orthogonal 
polynomials distribute asymptotically as the equilibrium measure on the set of minimal capacity within the appropriate class (proving in passing their existence), but also 
characterized these sets in terms of a symmetry property, now known as the {\it $S$-property}. These ideas were considerably extended by Gonchar and Rakhmanov 
\cite{gonchar_rakhmanov_rato_rational_approximation} to include the case when the orthogonality weight depends on the degree of the polynomial. The electrostatic model features now 
a non-trivial external field (or background potential) that persists in the asymptotic regime; in consequence, the notion of the $S$-property had to be extended to deal with this 
external field.

The resulting Gonchar-Rakhmanov-Stahl (or GRS) theory allows to formulate conditional theorems of the form: ``\emph{if} there exists a compact set with the $S$-property 
corresponding to the associated electrostatic model, then the weak limit of the zero counting measures of the related non-hermitian orthogonal polynomials is the equilibrium 
measure of this set''. However, the existence of such compact sets is a problem from the geometric function theory, and it usually has to be settled in each concrete situation, 
basically from scratch. There are very few genuinely general rigorous results in this sense, for instance the most recent one of Kuijlaars and Silva 
\cite{kuijlaars_silva} settling down the existence for polynomial external fields on the plane, and also the contribution by Kamvissis and Rakhmanov \cite{kamvissis_rakhmanov,kamvissis} 
which deals with the Green potential instead of the logarithmic potential.

The notion of the $S$-property found also a natural interpretation in the light of the Deift-Zhou's non-linear steepest descent method for the Riemann-Hilbert problems 
\cite{deift_book}. One of the key steps in the asymptotic analysis is the deformation of the contours, which transforms oscillatory behavior into a non-oscillatory plus 
exponentially decaying one. This deformation consists in ``opening of lenses'' along the  level sets of certain functions (known as {\it $g$-functions}). It is precisely the 
$S$-property of these sets which guarantees the exponential decay on all non-relevant contours.

The available log gas models on the conducting plane suggest to extend the analysis of the equilibrium distributions (local energy minima on the prescribed sets) to the study of 
saddle points of the logarithmic energy on the plane. This yields the notion of  \emph{critical measures}, studied systematically in  \cite{martinez_rakhmanov}, following the 
original ideas in the unpublished manuscript  \cite{Perevoznikova_Rakhmanov}. The advantage of dealing with saddle points of the energy is the possibility to avoid the 
complications of the free boundaries, as well as to use variational techniques that go back to the work of L\"owner and Schiffer.

One of the main conclusions of \cite{martinez_rakhmanov} (see also \cite{kuijlaars_silva}) is the fact that the 
Cauchy transform of any critical measure satisfies (a.e.~on the plane) a quadratic equation, known as the \emph{spectral curve} in the context of random matrix theory. An immediate
consequence of this fact is that the support of such a measure is a finite union of analytic curves, all of them being trajectories of the same quadratic differential on the plane.

Quadratic differentials are known to play an essential role in the study of conformal and quasi-conformal mappings, univalent functions, etcetera. Teichm\"uller \cite{MR0003242},
based on the fundamental work of Gr\"otzsch, formulated a principle according to which quadratic differentials and their singularities are associated with the solutions of 
extremal 
problems in geometric function theory. Although he did not prove any general result supporting his principle, we should give credit to Teichm\"uller for putting spotlight on 
quadratic differentials. Schiffer \cite{MR1575335,MR0064149} introduced several variational techniques in the geometric function theory, allowing to find differential equations for 
the 
boundaries of the extremal domains, which in turn give a more formal connection with quadratic differentials. See \cite{jenkins_book,Pommerenkebook,MR867407} for further details.

As it was mentioned above, the Riemann-Hilbert asymptotic analysis for $2\times 2$ matrix-valued functions characterizing the non-hermitian orthogonality provides another natural 
connection with quadratic differentials: the ``right'' level curves for the $g$-functions, that usually  have a quite complicated structure, are trajectories of a suitable 
quadratic differential. Thus, the problem of existence of the appropriate curves can be reduced to the question about the global structure of such trajectories. This is not a 
simplification (the description of the global structure can be a formidable task), but it opens the gate to other techniques from the geometric function theory. These ideas have 
been successfully applied by Bertola \cite{bertola_boutroux}, being closely related (in fact, in some sense equivalent) to 
the max-min approach mentioned before, see also \cite{bertola_tovbis_quartic_weight}. In the same spirit and in a much more explicit form, quadratic differentials also appeared in 
similar asymptotic problems in  
\cite{baik_deift_et_al_quadratic_differential,atia_martinez_et_al_quadratic_differentials,martinez_thabet_quadratic_differentials,Kuijlaars/McLaughlin:01,Kuijlaars/Mclaughlin:04,holst_shapiro}. 
We emphasize that in all these problems, the underlying quadratic differentials are defined on the complex plane.

Returning to our motivation rooted in the interacting particle systems, recent works have showed that more general models, such as the hermitian matrix model with an external 
source 
\cite{kuijlaars_bleher_external_source_multiple_orthogonal,bleher_delvaux_kuijlaars_external_source,kuijlaars_bleher_external_source_gaussian_I}, the two-matrix model
\cite{duits_kuijlaars_two_matrix_model,duits_kuijlaars_mo,bertola_gekhtman_szmigielski_cauchy_two_matrix_model,balogh_bertola_vector_potential_problem,duits_geudens_kuijlaars}, 
some classes of non-intersecting paths
\cite{kuijlaars_multiple_orthogonal_random_matrix_theory_ICM,kuijlaars_martinez-finkelshtein_wielonsky_bessel_paths} or the normal matrix model 
\cite{bleher_kuijlaars_normal_matrix_model,kuijlaars_lopez_normal_matrix_model} require  considering log gases with more than one group of charges. In this case, standard 
orthogonal polynomials are not enough, and we must turn to the so-called \emph{multiple orthogonal polynomials} and to the associated vector equilibrium problems \cite{MR807734}. 
The connection of these two notions was put forward first by Gonchar and Rakhmanov~\cite{MR807734,gonchar_rakhmanov_1981_pade}, motivated by the analytic theory of Hermite--Pad\'e 
approximants. The reader interested in further details is advised to look into the monograph by Nikishin and Sorokin~\cite{nikishin_sorokin_book}, which features a nice 
introduction to vector equilibrium problems from the approximation theory's perspective. 

Although the analytic theory of multiple orthogonal polynomial is still in its infancy, two somewhat extreme cases are fairly well understood (see 
e.g.~\cite{MR2247778,MR2963452}). 
The first one, known as the Angelesco system, corresponds to orthogonality with respect to measures living on disjoint sets, whose asymptotics is described by a vector equilibrium 
for repelling systems of charges. The other one, called the Nikishin system, where on the contrary the orthogonality is defined on the same set for all measures, is governed by a 
system of mutually attracting charges. In both situations the size of the charged ensembles is fixed. From the recent works of Aptekarev \cite{MR2475084} and  Rakhmanov 
\cite{MR2796829} it became apparent that the ``intermediate'' cases require more complex log-gases models for their asymptotic descriptions, mixing both attracting and repelling 
charges, as well as allowing for charges ``flowing'' from one ensemble to another, see Sections \ref{sec:mainresults} and \ref{section_general_variational} below.

In the context of the non-hermitian orthogonality,  multiple orthogonal polynomials yield naturally the notion of the \emph{critical vector-valued measures}, defined again as 
saddle points (and thus, solutions of some variational problems) for the total logarithmic energy.  

In the same vein as the standard non-hermitian orthogonality can be characterized by a $2\times 2$ matrix valued Riemann-Hilbert problem \cite{deift_book,Fokas92}, the multiple 
orthogonality can be associated with larger ($k\times k$, $k\geq 3$) matrix valued Riemann-Hilbert problems \cite{Assche01}. Some heuristics (that in some situations can be made 
rigorous using the WKB approximation arguments \cite{MR2964145}) indicates that the asymptotics of multiple orthogonal polynomials (and the critical vector-valued measures) can be 
described using higher-order algebraic equations. This is indeed the case, as it follows from some previous works and from the results of this paper.

Some instances of the critical vector-valued measures have appeared in literature, basically in the study of the limit zero distribution of some  multiple orthogonal polynomials 
satisfying non-hermitian orthogonality conditions. The most studied cases are those presenting strong symmetry, so that the support of these measures can be easily derived. 
Other works, such as 
\cite{aptekarev_kuijlaars_vanassche_hermite_pade_genus_0,aptekarev_toulyakov_vanassche_hyperelliptic_uniformization}, consider non-symmetric situations. Their basic approach is to 
start from a general form 
expected for a cubic equation that should govern the system, and then make a genus ansatz, hence reducing the free parameters of the equation to a minimum. Although a 
vector equilibrium is present, due to the explicit form of the deduced algebraic equation its detailed study is not needed and usually bypassed. Moreover, in such situations there 
is no external 
field involved, which is indeed a considerable simplification, even in the scalar case. 

One of our main motivations is to understand the structure of the critical vector-valued measures with harmonic external fields. We follow the approach initiated in 
\cite{martinez_rakhmanov,kuijlaars_silva}, and show that the variational method, applied to critical measures, allows to deduce the corresponding algebraic equation (spectral 
curve) 
and the associated trajectories of a quadratic differential, now on a compact Riemann surface instead of the plane, whose analysis will be the central part of this paper.

\subsection{Main results: general polynomial external fields}\label{sec:mainresults}

Assume we are given a vector of three non-negative measures $\vec\mu=(\mu_1,\mu_2,\mu_3) $, compactly supported on the plane. For the 
interaction matrix $A$ and vector of real-valued functions (``external fields'') $\vec\phi=(\phi_1,\phi_2,\phi_3)$
\begin{equation}
\label{interactionmatrix}
A=(a_{j,k})=
\begin{pmatrix}
 1 & \frac{1}{2} & \frac{1}{2} \\
 \frac{1}{2} & 1 & -\frac{1}{2} \\
 \frac{1}{2} & -\frac{1}{2} & 1
\end{pmatrix},
\quad
\phi_j=\re \Phi_j,\quad j=1,2,3,
\end{equation}
where $\Phi_1,\Phi_2,\Phi_3$ are polynomials\footnote{Although our methodology is, generally speaking, valid also for general positive-semidefinite interaction matrices and 
when $\Phi_j'$ are rational, we restrict ourselves here to the given matrix $A$ and the
polynomial external fields; these choices avoid some technicalities related to cumbersome notation and the order or coincidence of the poles.}, and denoting by
$$
I(\mu,\nu):=\iint\log\frac{1}{|x-y|}d\mu(x)d\nu(y)
$$
the logarithmic energy of two compactly supported measures $\mu$ and $\nu$, we consider the total energy 
functional of the form \cite{MR807734}
\begin{equation}\label{energy_functional}
E(\vec \mu)=E(\vec\mu,\vec\phi)=\sum_{j,k=1}^3a_{j,k} I(\mu_j,\mu_k)+\sum_{j=1}^3\int \phi_j d\mu_j.
\end{equation}

Matrix $A$ is symmetric, singular and positive-semidefinite. Its only 
non-zero eigenvalue (of order 2) is $3/2$, and the corresponding eigenspace is given by the vectors $(v_1, v_2, v_3)^T \in \C^3$ (in what follows, $(\cdot)^T$ stands for the 
transpose) satisfying
$$
v_1-v_2-v_3=0.
$$
As it will be seen below, the external fields in \eqref{interactionmatrix} must satisfy the compatibility condition
of $\vec \Phi'=\left(\Phi'_1,  \Phi'_2, \Phi'_3 \right)^T$ being an eigenvector for the eigenvalue $3/2$, that is,
\begin{equation}\label{compatibility_condition_external_fields}
\Phi_1'(z)-\Phi_2'(z)=\Phi_3'(z),\quad z\in\C.
\end{equation}
We complete the description of the electrostatic model by introducing the constraints on vector-valued measures under consideration: given a parameter $\alpha\in [0,1]$, we require that 
$\vec\mu=(\mu_1,\mu_2,\mu_3) $ satisfies 
\begin{equation}
\label{massconstraint}
|\mu_1|+|\mu_2|=1,\quad |\mu_1|+|\mu_3|=\alpha,\quad |\mu_2|-|\mu_3|=1-\alpha. 
\end{equation}
Here $|\mu|$ denotes the total mass (total variation) of the measure $\mu$.

It should be pointed out that a particular case of this model (with all $\Phi_j\equiv 0$ and $\alpha=1/2$) appears in the work of Aptekarev \cite{MR2475084} and Rakhmanov 
\cite{MR2796829} in their study of Hermite-Padé approximants 
\cite{aptekarev_kuijlaars_vanassche_hermite_pade_genus_0,rakhmanov_hermite_pade,aptekarev_vanassche_yatsselev}. Notice that $A$ in \eqref{interactionmatrix} contains submatrices 
that are the interaction matrices both for the Angelesco and for the Nikishin cases.

Let us denote by $\lebesgue_2$ the plane Lebesgue measure on $\C$. 
We are interested in the existence of critical vector-valued measures within a family $\mathcal M_\alpha$ of measures  $\vec\mu=(\mu_1,\mu_2,\mu_3) $ 
satisfying \eqref{massconstraint} plus the following conditions:
\begin{itemize}
\item each $\mu_j$, $j=1, 2, 3$, is a non-negative Borel measure, supported on a compact set in $\C$ of zero plane Lebesgue measure,
$$
\lebesgue_2(\supp \mu_j)=0, \quad j=1, 2, 3,
$$
and such that the energy $E(\vec \mu)$, defined in \eqref{energy_functional}, is finite.
\item the set
\begin{equation} \label{defSetS}
S_\alpha:=\bigcup_{1\leq j<k\leq 3}(\supp\mu_j\cap \supp\mu_k)
\end{equation}
is finite, 
\begin{equation}\label{assumption_intersection_supports}
\card S_\alpha< \infty.
\end{equation}
\end{itemize}

The \textit{vector equilibrium problems} deal with the minimizers of the energy functional  \eqref{energy_functional} over a family of measures $\vec\mu$ living on prescribed (and 
fixed) 
subsets of $\C$, see e.g.~\cite{MR807734,beckermann_et_al_equilibrium_problems,hardy_kuijlaars_vector_equilibrium} and the references  therein. As it follows from 
\cite{beckermann_et_al_equilibrium_problems}, the equilibrium measure (global minimizer of $E(\vec \mu)$), if it exists in the class $\mathcal M_\alpha$, is unique. 

Here we consider a natural extension of the notion of vector equilibrium for the conducting plane: the \textit{critical vector-valued measures}, defined as 
follows. For $t\in \C$ and $h\in C^2(\C)$, denote 
by $\mu^t$ the pushforward measure of $\mu$ induced by the variation of the plane $z\mapsto h_t(z)=z+th(z)$, $z\in\C$. 

\begin{definition}\label{def:criticalmeasures}
We say that for $\alpha\in [0,1]$, $\vec\mu=(\mu_1,\mu_2,\mu_3)$ is an {\it $\alpha$-critical 
	measure} (or a saddle point of the energy $E(\cdot)$) if $\vec \mu\in \mathcal M_\alpha$, and 
\begin{equation}\label{variations_critical_measure}
\lim_{t\to 0}\frac{E(\vec\mu^t)-E(\vec\mu)}{t}=0,
\end{equation}
for every function $h\in C^2(\C)$.
\end{definition}
Sometimes, when the value of $\alpha$ is irrelevant or is clear from the context, we will simplify the terminology speaking simply about \emph{critical measures}.

In order to formulate our first results we introduce some notation. Given a (scalar) non-negative Borel measure $\lambda$, compactly supported on the plane, we denote by
\begin{equation*}
U^{\lambda}(z):=\int \log \frac{1}{|x-z|}d\lambda(x)
\end{equation*}
its logarithmic potential, which is harmonic in $\C\setminus \supp\lambda$ and superharmonic in $\C$. Additionally, the principal value of the Cauchy transform
\begin{equation}\label{principalvalue}
C^\lambda(z):= \lim_{\epsilon \to 0+} \int_{|z-x|> \epsilon }  \frac{1}{ x-z }\, d\lambda(x)
\end{equation}

is analytic in $\C\setminus \supp\lambda$, and
\begin{equation}\label{derivative_potential}
2 \, \frac{\partial U^\lambda}{\partial z} (z)=  C^{\lambda}(z),\quad z\in \C\setminus \supp\lambda,
\end{equation}
with
$$
\frac{\partial}{\partial   z} = \frac{1}{2}\, \left(\frac{\partial}{\partial x} - i \frac{\partial}{\partial y}\right).
$$

The interaction matrix $A$ in \eqref{interactionmatrix} satisfies the following identities:
\begin{equation}\label{identitiesforA}
2A =   B^T B = 3I_3 - b b^T, \quad B= \begin{pmatrix}
1 & 1 & 0 \\
-1 & 0 & -1 \\
0 & -1 & 1
\end{pmatrix}, \quad b= \begin{pmatrix}
1  \\
-1 \\
-1  
\end{pmatrix},
\end{equation}
where $I_3$ is the $3\times 3$ identity matrix. For a vector of measures $\vec\mu=(\mu_1,\mu_2,\mu_3)\in \mathcal M_\alpha$, define the vector of functions $\vec \xi=\left(\xi_1, 
\xi_2, \xi_3 \right)^T$ through the equality
\begin{equation}\label{definitionofXivec} 
\vec \xi = B \left( \frac{1}{3}\vec \Phi' + \vec C \right),
\end{equation}
where $\vec C = \left(C^{\mu_1}, C^{\mu_2}, C^{\mu_3} \right)^T$. Alternatively, the functions $\xi_1,\xi_2,\xi_3$ are given explicitly as

 \begin{equation}\label{definition_xi_functions}
\begin{aligned}
\xi_1(z) & =\frac{\Phi_1'(z)}{3}+\frac{\Phi_2'(z)}{3}+C^{\mu_1}(z)+C^{\mu_2}(z),\\
\xi_2(z) & =-\frac{\Phi_1'(z)}{3}-\frac{\Phi_3'(z)}{3}-C^{\mu_1}(z)-C^{\mu_3}(z),\\
\xi_3(z) & =-\frac{\Phi_2'(z)}{3}+\frac{\Phi_3'(z)}{3}-C^{\mu_2}(z)+C^{\mu_3}(z).
\end{aligned}
\end{equation}

It was proved in \cite[Section 5]{martinez_rakhmanov} (see also \cite{kuijlaars_silva}) that the Cauchy transform of a scalar critical measure is a solution of a 
quadratic equation $\lebesgue_2$-a.e.~on $\C$, which is a characterizing property of such measures. Our first result is a natural generalization of this property to the critical vector-valued measures:
\begin{thm}\label{thm_spectral_curve}
Suppose that for $\alpha\in [0,1]$, $\vec \mu=(\mu_1,\mu_2,\mu_3) \in \mathcal M_\alpha$ is an $\alpha$-critical measure in the sense of Definition~\ref{def:criticalmeasures}. 
Then there exist a polynomial $R$ and a rational function $D$ with poles of order at most $2$ such that the functions $\xi_1,\xi_2,\xi_3$ in \eqref{definitionofXivec}
satisfy $\lebesgue_2$-a.e.~on $\C$  the algebraic equation
\begin{equation}\label{spectral_curve_general}
\xi^3- R(z)\xi+D(z)=0.
\end{equation}
The polynomial $R$ is given by
\begin{equation}\label{expression_R}
R(z)=\frac{1}{9}\sum_{j,k=1}^3 a_{j,k}\Phi_j'(z)\Phi_k'(z)-\sum_{j=1}^3 \int \frac{\Phi'_j(x)-\Phi'_j(z)}{x-z}\, d\mu_j(x),\quad z\in\C.
\end{equation} 

Moreover, each of the measures $\mu_1,\mu_2,\mu_3$ is supported on a finite union of analytic arcs, and they are absolutely continuous with respect to the arclength measure of 
their supports. 

Each pole $p$ of $D$, if exists, belongs to the support of at least two of the measures $\mu_1$, $\mu_2$, $\mu_3$, and for each such a 	measure $\mu_k$ its density satisfies
\begin{equation}\label{local_behavior_blow_up}
\frac{d\mu_k}{|dz|}(z)=|z-p|^{-\nu}(c+\boh(1)),\quad \mbox{ as } z\to p \mbox{ along } \supp\mu_k,
\end{equation}
for some nonzero constant $c$ (which may depend on $p$ and $k$), where $3\nu$ is the order of $p$ as a pole of $D$.
\end{thm}

We also prove a converse to Theorem~\ref{thm_spectral_curve}.

\begin{thm}\label{thm_reciprocal_spectral_curve}
 Suppose that for $\alpha\in [0,1]$, $\vec \mu=(\mu_1,\mu_2,\mu_3) \in \mathcal M_\alpha$ is such that the measures $\mu_1,\mu_2,\mu_3$ are supported on a finite union of 
analytic arcs, and are absolutely continuous with respect to the arclength.

If for some polynomials $V_1,V_2,V_3$ with
\begin{equation}\label{compatibility_condition_pols_V}
V_1'(z)+V_2'(z)+V_3'(z)=0,\quad z\in \C,
\end{equation}
the functions
 \begin{equation}\label{shifted_resolvents_pols_V}
\begin{aligned}
\xi_1(z) & =V_1'(z)+C^{\mu_1}(z)+C^{\mu_2}(z),\quad z\in \C\setminus(\supp\mu_1\cup\supp\mu_2),\\
\xi_2(z) & =V_2'(z)-C^{\mu_1}(z)-C^{\mu_3}(z),\quad z\in \C\setminus(\supp\mu_1\cup\supp\mu_3),\\
\xi_3(z) & =V_3'(z)-C^{\mu_2}(z)+C^{\mu_3}(z),\quad z\in \C\setminus(\supp\mu_2\cup\supp\mu_3).
\end{aligned}
\end{equation}
satisfy an algebraic equation of the form \eqref{spectral_curve_general} for a polynomial $R$ and a rational function $D$, then $\vec \mu$ is a 
critical measure for the energy functional \eqref{interactionmatrix}--\eqref{energy_functional} and the external fields defined by
\begin{equation}\label{definition_new_external_fields}
\Phi_1(z):=V_1(z)-V_2(z),\quad \Phi_2(z):=V_1(z)-V_3(z),\quad \Phi_3(z):=V_3(z)-V_2(z).
\end{equation}
\end{thm}

Notice that with the definition \eqref{definition_new_external_fields} and the compatibility condition \eqref{compatibility_condition_pols_V} the equalities 
\eqref{shifted_resolvents_pols_V} take the form \eqref{definition_xi_functions}.

Few remarks are in order. 

Theorems \ref{thm_spectral_curve} and \ref{thm_reciprocal_spectral_curve} show that there is a clear connection between critical measures and cubic equations of a specific form. 
The fact that the equation is cubic has to do with the  rank of the interaction matrix $A$ in \eqref{interactionmatrix}, which is 2: from the proofs it is natural to expect that 
critical measures (with $\geq 3$ components) for an energy functional with an interaction matrix of rank $r$ should be related to algebraic equations of 
degree $r+1$.

Furthermore, the compatibility condition \eqref{compatibility_condition_pols_V} has to do with the structure of the interaction matrix 
\eqref{interactionmatrix}, and finds its expression also in the absence of the $\xi^2$  term in the algebraic equation 
\eqref{spectral_curve_general}. Equation \eqref{compatibility_condition_pols_V} and the structure of \eqref{shifted_resolvents_pols_V} should be seen as normalization conditions, 
ensuring that the underlying energy functional is of the precise form \eqref{energy_functional}. Different choices for these identities would lead to a different interaction 
matrix in \eqref{energy_functional}. 

The notion of the critical vector-valued measures can be extended to accommodate also log-rational external fields and domains with non-trivial boundaries (``hard edges'', in the 
terminology of random matrix theory). The variations of the energy in these situations must allow for fixed points, see \cite{martinez_rakhmanov} for the 
scalar case. Fixed points and singularities of the external fields will appear among possible points of blow-up of the density of the critical measures, and also as possible poles 
of the coefficients $R$ and $D$ in \eqref{spectral_curve_general}.

Finally, a reader familiar with the notion of the scalar equilibrium and scalar critical measures might be puzzled by the fact that we allow for a blow-up of the densities of such 
measures in the case of the polynomial external fields. Indeed, in the scalar case all critical measures in such circumstances have uniformly bounded densities on the whole 
complex 
plane \cite{kuijlaars_silva}. This is no longer the case in the vector case, as the following two examples illustrate: they show that the components of the critical vector-valued measures can create 
``artificial hard edges'', obstructing each other and presenting unbounded densities when it was not expected, i.e.~not created by the external field or by the boundaries of the conducting domain. 

\begin{example}[Repulsive or Angelesco interaction]\label{example_angelesco}
Consider the algebraic equation
$$
\xi^3-\xi-\frac{1}{z}=0.
$$
It has simple branch points at $\pm 3\sqrt{3}/2$, and a double branch point at $z=0$, and three solutions $\xi_j$ specified by their behavior as $z\to \infty$:
$$
\xi_1(z)=-\frac{1}{z}+\mathcal O\left(\frac{1}{z^2}\right), \quad \xi_2(z)=1+\frac{1}{2z}+\mathcal O\left(\frac{1}{z^2}\right), \quad \xi_3(z)=-1+\frac{1}{2z}+\mathcal 
O\left(\frac{1}{z^2}\right)  ,
$$
with $ \xi_1(z)+\xi_2(z)+\xi_3(z)=0$. A direct calculation shows that there are two measures, $\mu_1$, living on $[-3\sqrt{3}/2,0]$, and $\mu_2(x)=\mu_1(-x)$, such that 
$|\mu_1|=|\mu_2|=1/2$, and
\begin{align*}
C^{\mu_1}(z) & =1-\xi_2(z), \quad z\in \C\setminus [-3\sqrt{3}/2, 0], \\ C^{\mu_2}(z) &=-1-\xi_3(z),  \quad z\in \C\setminus [0, 3\sqrt{3}/2].
\end{align*}

Hence Theorem~\ref{thm_reciprocal_spectral_curve} shows that $\vec \mu =(\mu_1, \mu_2, 0)\in \mathcal M_{1/2}$ is a $1/2$-critical vector-valued measure, according to Definition 
\ref{def:criticalmeasures}, corresponding to the interaction matrix \eqref{interactionmatrix} and the polynomial external fields given by
$$
\Phi_1(z)=-z=-\Phi_2(z),  \quad \Phi_3(z)=-2z.
$$ 
Notice however that the densities of both measures $\mu_j$, $j=1, 2$, blow up at $z=0$ as $|z|^{-1/3}$.

It is worth mentioning that the same pair of measures $(\mu_1, \mu_2)$ solves the Ange\-lesco-type vector \emph{equilibrium} problem of minimizing \eqref{energy_functional} under 
the assumptions
$$
\supp\mu_1 \subset (-\infty,0],\quad \supp\mu_2 \subset [0,+\infty),\quad |\mu_1|=|\mu_2|=\frac{1}{2}, \quad \mu_3=0.
$$
This equilibrium problem does exhibit a hard edge at the origin, and the proof is based on the same type of energy variation, but leaving $z=0$ as a fixed point. It should be pointed out that a similar problem (but on finite intervals) was analyzed first by Kalyagin in \cite{kalyagin}.
\end{example}

\begin{example}[Attractive or Nikishin interaction]\label{example_bertola}
Let us consider now the following cubic equation, which (up to a certain rescaling) have appeared already in \cite{balogh_bertola_vector_potential_problem}:
\begin{equation}\label{secondCubiceq}
\xi^3-\frac{1}{3}\xi-\frac{1}{z^2}+\frac{2}{27}=0.
\end{equation}
As in the Example \ref{example_angelesco}, it has simple branch points at $\pm 3\sqrt{3}/2$, and a double branch point at $z=0$.

Let $(81-12z^2)^{1/2}$ denote the branch of the square root in $\C\setminus [-3\sqrt{3}/2,3\sqrt{3}/2]$ with positive boundary values on  $(-3\sqrt{3}/2,3\sqrt{3}/2)$ from the upper half plane, and let 
$$
\Upsilon(z):= \left( -1+ \frac{3(81-12z^2)^{1/2}}{2z^2}\right)^{1/3}
$$
also have positive boundary values on  the same interval from the upper half plane. Then $\Upsilon$ is holomorphic in $\C\setminus [-3\sqrt{3}/2,3\sqrt{3}/2]$, and 
$$
\xi_1(z)=\frac{1}{3}\left(e^{\pi i /3 }\Upsilon(z)+ \frac{1}{e^{\pi i /3 } \Upsilon(z)}  \right)=\frac{1}{3}-\frac{1}{z}+\mathcal O\left(\frac{1}{z^2}\right) \text{ as } z\to \infty
$$
is holomorphic in $\C\setminus [0,3\sqrt{3}/2]$. Moreover, Cardano's formula shows that $\xi_1$ is a solution of \eqref{secondCubiceq}; the other two solutions are given by  
$$
\xi_2(z)=\xi_1(-z) \quad \text{and} \quad \xi_3(z)=-\xi_1(z)-\xi_2(z).
$$

For the selected branch of $\Upsilon$, the $\xi_{1,\pm}$ boundary values of $\xi_1$ on $(0,3\sqrt{3}/2)$ from the upper and lower half plane, respectively, satisfy
$$
w(x):=\frac{1}{2\pi i} \left( \xi_{1,+}-\xi_{1,-} \right)(x)>0, \quad x \in (0,3\sqrt{3}/2).
$$
Thus, $d\mu_2(x):=w(x)dx$ is a positive unit Borel measure on $[0,3\sqrt{3}/2]$, and its Cauchy transform is
$$
C^{\mu_2}(z) =\xi_1(z)-\frac{1}{3}, \quad  z\in \C \setminus [0,3\sqrt{3}/2)].
$$
Analogously, $\mu_3(z):=\mu_2(-z)$ is a positive unit Borel measure on $[-3\sqrt{3}/2, 0]$, and
$$
C^{\mu_3}(z) =-\xi_2(z)+\frac{1}{3}, \quad  z\in \C \setminus [-3\sqrt{3}/2),0].
$$
Theorem~\ref{thm_reciprocal_spectral_curve} shows that $\vec \mu =(0, \mu_2/2, \mu_3/2)\in \mathcal M_{1/2}$ is a $1/2$-critical vector-valued measure, according to Definition 
\ref{def:criticalmeasures}, corresponding to the interaction matrix \eqref{interactionmatrix} and the polynomial external fields given by
$$
\Phi_1(z)=0, \quad \Phi_2(z)=z=-\Phi_3(z) .
$$ 
Notice that now the densities of both measures $\mu_j$, $j=2, 3$, blow up at $z=0$ as $|z|^{-2/3}$.
\end{example}

The next example illustrates that in Theorem~\ref{thm_reciprocal_spectral_curve} we cannot discard  the degenerate cases when the coefficient $D$ in 
\eqref{spectral_curve_general} is identically zero, or other situations when \eqref{spectral_curve_general} becomes reducible over the field of rational functions:
\begin{example}\label{example_scalar_critical_measure}
If $\mu$ is a scalar critical measure for the external field $\re V$, where $V$ is a polynomial, then there exists a polynomial $Q$ for 
which \cite{martinez_rakhmanov,kuijlaars_silva}
\begin{equation}\label{alg_equation_scalar_critical_measure}
\left(C^{\mu}(z)+\frac{1}{2}V'(z)\right)^2=Q(z),\quad \lebesgue_2\mbox{-- a.e. } z\in \C. 
\end{equation}

If we set
$$
\mu_1=\mu_3=0,\quad \mu_2=\mu,\quad \Phi_2=V, \quad \Phi_1=- \Phi_3=\frac{V}{2},
$$
then
\begin{equation}\label{energy_scalar}
E(\vec\mu^t)=I(\mu^t,\mu^t)+\int \re V d\mu^t,
\end{equation}
which shows that the vector-valued measure $\vec \mu=(\mu_1,\mu_2,\mu_3)$ is $\alpha$-critical for $\alpha=0$. The function $\xi_2$ in \eqref{definition_xi_functions} becomes zero, 
whereas $\xi_1$, $\xi_3$ are analytic continuations of each other across $\supp\mu$. From the proof of Theorem~\ref{thm_spectral_curve} in 
Section~\ref{section_general_variational} we conclude that the corresponding algebraic equation \eqref{spectral_curve} is then reducible and has the form
\begin{equation}\label{reduced_algebraic_equation}
\xi(\xi^2-Q(z))=0,
\end{equation}
where $Q$ is the polynomial in \eqref{alg_equation_scalar_critical_measure}.

Reciprocally, if the Cauchy transform of a probability measure $\mu$ satisfies an algebraic equation of the form \eqref{alg_equation_scalar_critical_measure} for some polynomials 
$V$ and $Q$, then the construction just carried out combined with Theorem~\ref{thm_reciprocal_spectral_curve} gives us that $\mu$ is a scalar critical measure for the external 
field $\re V$.

In very much the same spirit, if we now set 
\begin{equation}\label{example_scalar_critical_measure_2}
 \mu_2=\mu_3=\mu,\quad \mu_1=0,\quad \Phi_1=2\Phi_2=2\Phi_3=V,
\end{equation}
 then the energy of $\vec\mu^t$ is also given by \eqref{energy_scalar}, showing again that $\vec\mu^t$ is $\alpha$-critical, now for $\alpha=1$. The respective functions $\xi_1,\xi_2$ are 
analytic continuation of each other across $\supp\mu$, whereas $\xi_3\equiv 0$ on $\C$, and the corresponding algebraic equation is reducible to the form 
\eqref{reduced_algebraic_equation}.
\end{example}

The critical measure in \eqref{example_scalar_critical_measure_2} violates the assumption \eqref{assumption_intersection_supports}, but the conclusions of 
Theorem~\ref{thm_spectral_curve} still hold. However, condition \eqref{assumption_intersection_supports} cannot be simply dropped, as the following example shows:
	
\begin{example}\label{example_arno}
Let $\Phi_1$, $\Phi_2$ be two arbitrary polynomials. For $\alpha=1/2$ and the polynomial vector external field $\vec\phi=(\phi_1, \phi_2, \phi_1-\phi_2)^T$, 
$\phi_j=\Re \Phi_j$, satisfying \eqref{compatibility_condition_external_fields}, consider the vector-valued measure $\vec \mu=(\mu_1, \mu_2, \mu_3)$, with
$$
\mu_1=\mu_2=\frac{1}{2}\mu,\quad \mu_3=0.
$$
The total energy of $\vec \mu$ is
$$
E(\vec\mu )= \frac{3}{4} \left( I(\mu ,\mu )   +  \frac{2}{3}  \int( \phi_1 + \phi_2) d\mu  \right).
$$
Thus, if $\mu$ is a \textit{scalar} probability critical measure in the external field
\begin{equation*}
\phi =\frac{2}{3}(\phi_1+\phi_2).
\end{equation*}
then \eqref{variations_critical_measure} for $\vec\mu$ takes place. In other words, $\vec \mu$ is $1/2$-critical in the external field $\vec\phi$. However, the statement of Theorem~\ref{thm_spectral_curve} does not hold: each of the corresponding functions $\xi_1,\xi_2,\xi_3$ in 
\eqref{definition_xi_functions} satisfies an algebraic equation of degree $2$, but if $\phi_1\neq  \phi_2$, these 
equations cannot be combined into a single algebraic equation of degree $3$. To be more precise, calculations in the spirit of the proof of Theorem~\ref{thm_spectral_curve} in 
Section~\ref{section_general_variational} below show that if $Q$ is the polynomial in the right hand side of the equation \eqref{alg_equation_scalar_critical_measure} for $\mu$ 
(with $V'=(\Phi_1'+\Phi_2')/3$), then
$$
\xi_1^2=Q, \quad \xi_2= -\frac{1}{2}(\xi_1+\Phi_1'- \Phi_2'), \quad \xi_3= -\frac{1}{2}(\xi_1-\Phi_1'+ \Phi_2').
$$

Clearly, the only assumption not satisfied by the measure $\vec \mu$ in this example is   \eqref{assumption_intersection_supports}. 
\end{example}

A consequence of Theorem~\ref{thm_spectral_curve} is that the set 
\begin{equation}\label{definitionXialpha}
\Xi_\alpha = \left(\supp \mu_1 \cup \supp \mu_2 \cup \supp \mu_3\right)\setminus S_\alpha
\end{equation}
is a finite union of disjoint analytic arcs, each of them contained in exactly one of the supports $\supp\mu_1,\supp\mu_2,\supp\mu_3$.

Any orientation of an arc $\Sigma\in \Xi_\alpha$ defines its left and right sides, as well as the left (or $+$) and right (or $-$) continuous boundary values of each $\xi_j$ on 
$\Sigma$, that we denote by $\xi_{j\pm}$. In what follows, we use extensively (and without further warning) the notation mod $3$, so that $\xi_0=\xi_3$, $\xi_4=\xi_1$ and so 
forth. 

Since $\xi_j$'s are the three branches of the cubic equation \eqref{spectral_curve_general}, 
assumption \eqref{assumption_intersection_supports} implies that for any curve $\Sigma\subset \Xi_\alpha$ there exists an index $j\in \{1, 2, 3\}$ such that
\begin{equation}\label{standard_orientation1}
\xi_{j+}(z)=\xi_{j-}(z),\quad \xi_{(j+1)\pm}(z)=\xi_{(j-1)\mp}(z),\quad  z\in \Sigma.
\end{equation}

Critical measures are intimately connected with vector equilibrium problems and the vector symmetry ($S$-property), as it is evidenced by the next theorem.

\begin{thm}\label{thm_s_property} 
Given $j\in \{1, 2, 3\}$, let $\Sigma$ be an open analytic arc in $  \supp\mu_j\setminus S_\alpha$, not containing any branch point of \eqref{spectral_curve_general}. 
There exists a constant $l=l(\Sigma)\in 
\R$ for which both the Euler-Lagrange variational equation (\emph{equilibrium 
condition})
\begin{equation}\label{variational_equations_critical_measures}
\sum_{k=1}^3a_{j,k}U^{\mu_k}(z)+\frac{\phi_{j}(z)}{2}=l,
\end{equation}
and the $S$-property
\begin{equation*}
\frac{\partial}{\partial n_+}\left(\sum_{k=1}^3a_{jk}U^{\mu_k}(z)+\frac{\phi_{j}(z)}{2}\right)=\frac{\partial}{\partial 
n_-}\left(\sum_{k=1}^3a_{jk}U^{\mu_k}(z)+\frac{\phi_{j}(z)}{2}\right)
\end{equation*}
hold true for $z\in \Sigma$, where $n_\pm$ are the unit normal vectors to $\Sigma$, pointing in the opposite directions.

\end{thm}

If, for $j=1, 2, 3$, the constants $l(\Sigma)=:l_j$ above are independent of the connected
component $\Sigma$ of $\operatorname{supp}\mu_j$, and additionally satisfy $l_1-l_2=l_3$, then the critical vector-valued measure is the
equilibrium measure for the vector of contours $(\operatorname
{supp}\mu_1,\operatorname{supp}\mu_2,\operatorname{supp}\mu_3)$,
see \cite[Theorem~1.8]{beckermann_et_al_equilibrium_problems}. Examples of critical but non-equilibrium measures in the scalar case are well-known (see e.g.~\cite{martinez_rakhmanov}). Using the construction of Example~\ref{example_scalar_critical_measure}, in particular \eqref{example_scalar_critical_measure_2}, we can easily build vector critical non-equilibrium measures. 

The $S$-property is the natural generalization of the scalar $S$-property, taking into account that in those arcs the potentials of the measures 
$\mu_{j-1}$ and $\mu_{j+1}$ are harmonic. 

Critical vector-valued measures $\vec \mu$ are also closely connected with trajectories of a quadratic differential on the Riemann surface $\mathcal R=\mathcal R_1 \cup \mathcal 
R_2 \cup \mathcal R_3$ of the algebraic equation \eqref{spectral_curve_general}, with 
\begin{equation}\label{construction_riemann_surface_general}
\begin{aligned}
\mathcal R_1 & = \overline \C\setminus (\supp\mu_1\cup \supp\mu_2), \\
\mathcal R_2 & =\overline \C\setminus (\supp\mu_1\cup \supp\mu_3), \\
\mathcal R_3 & =\overline\C\setminus (\supp\mu_2\cup \supp\mu_3).
\end{aligned}
 \end{equation}
 
The sheets are connected across arcs in $\Xi_\alpha$ according to the following rule: if for $\Sigma\subset \Xi_\alpha$ the condition  \eqref{standard_orientation1} holds, then we connect 
the sheets $\mathcal R_{j+1}$ and $\mathcal R_{j-1}$ through $\Sigma$ crosswise, identifying the $\pm$-side of $\Sigma$ on $\mathcal R_{j+1}$ with its $\mp$-side on $\mathcal 
R_{j-1}$.

As usual, we regard the solutions $\xi_1,\xi_2,\xi_3$ in \eqref{definition_xi_functions} as branches of the same meromorphic function $\xi:\mathcal R\to \overline \C$, defined by 
\eqref{spectral_curve_general}, so that function $\xi_j$ is the restriction of $\xi$ to the sheet $\mathcal R_j$ of the Riemann surface $\mathcal R$. 

\begin{thm}\label{thm_quadratic_differential}
Let
 \begin{equation*}
 Q(z) =
 \begin{cases}
 \xi_2(z)-\xi_3(z)  , & \text{ on } \mathcal R_1, \\
 \xi_1(z)-\xi_3(z) , & \text{ on } \mathcal R_2, \\
 \xi_1(z)-\xi_2(z) , & \text{ on } \mathcal R_3.
 \end{cases}
 \end{equation*}
Then $Q^2$ extends to a meromorphic function on $\mathcal R$. Moreover, 
 \begin{equation}
\label{defQqd}
\varpi =-Q^2 (z)dz^2
\end{equation}
is a meromorphic quadratic differential on $\mathcal R$ with possible poles only at the points at $\infty$. Finally, for $j=1,2,3$, each arc of $\supp\mu_j$ is an arc 
of trajectory of $\varpi$.
\end{thm}

As it was mentioned in Section \ref{sec:historical}, the connection between quadratic differentials and critical (or equilibrium) measures is not new. For scalar critical measures, 
the quadratic 
differential can be globally projected on the complex plane, so it is not necessary to consider it on a non-trivial Riemann surface. Even for vector-valued measures, this connection has 
been 
exploited before, see for instance \cite{tovbis_kuijlaars_supercritical,kuijlaars_vanassche_wielonsky_hermite_pade}, but only locally on the plane. To our knowledge, 
Theorem~\ref{thm_quadratic_differential} is the first time this connection is stated on the whole surface $\mathcal R$, which provides a powerful tool for construction of critical 
measures. 

\begin{remark} \label{remark:Angelesco-Nikishin} 
Assume that at $z=p$ we have \eqref{local_behavior_blow_up} with $\nu \in \{1/3, 2/3\}$ (see Theorem~\ref{thm_spectral_curve}). Then $p$ 
is a triple branch point of $\mathcal R$, so that  the local coordinate for $\mathcal R$ at $z=p$ is $z=p+w^3$. 
Theorem~\ref{thm_quadratic_differential}  asserts 
that $\supp\mu_j$ lives on trajectories of the quadratic differential $\varpi$ defined in \eqref{defQqd}, which in this local coordinate takes the form  $\varpi=w^{4-6\nu}(\const+o(1)) dw^2$, $\const\neq 0$.

If $\nu=1/3$ (or equivalently if the polynomial $D$ in \eqref{spectral_curve_general} has a simple pole at $p$), then by the local structure of the trajectories, their canonical 
projection from $\mathcal R$ onto $\C$ consists of two analytic arcs intersecting  at $p$ perpendicularly (a ``cross''). As Example~\ref{example_angelesco} shows, this 
configuration can correspond to the interaction of two repulsive measures, and $p$ can potentially belong to the intersection of the three components $\mu_j$.

However, for $\nu=2/3$ (which means that $D$ has a double pole at $p$) the canonical projection of the trajectories from $\mathcal R$ onto $\C$ at $p$ consists of a single 
analytic arc through this point. It implies, due to \eqref{assumption_intersection_supports}, that in this case $p$ belongs to the intersection of at most two (and then, exactly 
two) components of $\vec \mu$; moreover, it will be the endpoint of an arc in each of the supports. A quick analysis of the cubic branching shows that these two positive measures 
cannot interact repulsively (``Angelesco'' interaction), and hence, must attract each other (``Nikishin'' interaction). Thus, Example~\ref{example_bertola} represents (at least, 
locally) the most generic case of the $2/3$-blowup of the density of a critical vector-valued measure.  
\end{remark}

\begin{remark}
 We should point out that \eqref{assumption_intersection_supports} plays a substantial role for Theorem~\ref{thm_spectral_curve}, but the conclusion of 
Theorem~\ref{thm_reciprocal_spectral_curve} still holds true even if we drop \eqref{assumption_intersection_supports} completely. In this situation, if $\Sigma$ is an analytic arc 
contained in the support of at least two of the measures $\mu_1,\mu_2,\mu_3$, and $\mu_2\neq \mu_3$ on $\Sigma$, then we can choose an orientation on $\Sigma$ for which
\begin{equation}\label{standard_orientation}
\xi_{1+}(z)=\xi_{3-}(z),\quad \xi_{2+}(z)=\xi_{1-}(z),\quad \xi_{3+}(z)=\xi_{2-}(z),\quad z\in \Sigma.
\end{equation}

The Euler-Lagrange equation \eqref{variational_equations_critical_measures} holds true for the three measures $\mu_1,\mu_2,\mu_3$ (and possibly different constants 
$l$), and with respect to the orientation \eqref{standard_orientation} the $S$-property now becomes
\begin{equation*}
\frac{\partial}{\partial n_+}\left(\sum_{k=1}^3a_{jk}U^{\mu_k}(z)+\frac{\phi_j(z)}{2}\right)=(-1)^{j+1}\frac{\partial}{\partial 
n_-}\left(\sum_{k=1}^3a_{j+1,k}U^{\mu_k}(z)+\frac{\phi_{j+1}(z)}{2}\right).
\end{equation*}

The conclusion of Theorem~\ref{thm_quadratic_differential} is also valid in this situation (with appropriate gluing of the sheets of $\mathcal R$ across the overlaps of the supports). However, we do not know whether there are positive measures with non-trivial overlapping supports for which the functions \eqref{definition_xi_functions} are the three solutions to the 
same cubic equation as in Theorem~\ref{thm_reciprocal_spectral_curve}. 

\end{remark}

\subsection{Main results: the cubic case}\label{sec:resultscubic}

In order to illustrate the results stated above we analyze in depth the interesting and highly non-trivial cubic case, corresponding to the energy functional 
\eqref{interactionmatrix}--\eqref{energy_functional}, with
\begin{equation}\label{cubic_choice_potentials}
\Phi_1(z)=\Phi_2(z)=z^3,\quad \Phi_3(z)=0,\quad z\in\C,
\end{equation}
so that the external fields are 
\begin{equation}\label{external_fields_cubic_potential}
\phi(z)=\phi_1(z)=\phi_2(z)=\re z^3,\quad \phi_3(z)=0,\quad z\in\C;
\end{equation}
in what follows, we consider only the values of the parameter $\alpha\in [0,1/2)$; see the discussion in the Remark~\ref{restrictiononAlpha} and also in the 
introduction to Section~\ref{section_cubic_main}.

We build a continuous one-parameter family of 
critical vector-valued measures $\vec \mu$ by choosing in  \eqref{spectral_curve_general} the  coefficients 
\begin{equation}\label{coefficients_cubic}
R(z)= 3z^4-3z-c,\quad D(z)=-2z^6+3z^3+cz^2-3\alpha(1-\alpha),
\end{equation}
where $c=c(\alpha)$ is the real parameter given by
\begin{equation}\label{definition_c_intro}
c=-\left(\frac{243}{64}(1-4\alpha(1-\alpha))^2\right)^{\frac{1}{3}}.
\end{equation}

\begin{thm}\label{thm_genus_zero}
For $\alpha\in [0,1/2)$, there exists a one-parameter family $\vec\mu=\vec\mu_\alpha\in \mathcal M_\alpha$ of critical measures for the potentials 
\eqref{external_fields_cubic_potential} for which the 
corresponding algebraic equation \eqref{spectral_curve_general} has coefficients \eqref{coefficients_cubic}.

Moreover, for any such a critical measure, if the associated algebraic equation \eqref{spectral_curve_general} defines a Riemann surface of 
genus $0$, then its coefficients are given by \eqref{coefficients_cubic} for some choice of the cubic root in \eqref{definition_c_intro}.
\end{thm}
\begin{remark}\label{remarkaboutthegenus}
As it follows from the proof of this theorem in Section~\ref{section_cubic_spectral_curve} (see also Remark~\ref{rem:qdforDifferentC}), there are other choices for the 
coefficient $c$ that yield critical vector-valued measures, but for which the associated Riemann surface is of genus 1. We do not consider this case here.

Theorem~\ref{thm_equilibrium_conditions_cubic} below (see also Remark~\ref{rem_equilibrium_conditions_cubic}) shows that the measures $\vec \mu_\alpha$ from Theorem~\ref{thm_genus_zero} are in fact equilibrium measures on an appropriate family of contours in $\mathbb C$.
\end{remark}

For the critical measure $\vec\mu_\alpha$ given by Theorem~\ref{thm_genus_zero}, the algebraic function $\xi$ defined by \eqref{spectral_curve_general} has two real branch points 
$a_1<b_1$ and two nonreal branch 
points $b_2$, $a_2=\overline{b_2}$ ($\im b_2>0$). The support of the components $\mu_j$ of the critical measure  $\vec\mu_\alpha =(\mu_1,\mu_2,\mu_3)$ can be easily described; we 
show that it exhibits a phase transition for a certain value of $\alpha$:
\begin{thm}\label{thm_existence_critical_measures_cubic}
There exists a critical value $\alpha_c\in (0,1/2)$, determined by
\begin{equation}\label{equation_definition_alphac}
\im\int_{b_2}^{a_1}(\xi_1(s)-\xi_3(s))ds=0,
\end{equation}
where $\xi_1,\xi_3$ are the functions in \eqref{definition_xi_functions}, such that
\begin{itemize}
 \item If $\alpha<\alpha_c$ then $\mu_3=0$, $\supp\mu_1=[a_1,b_1]$, and $\supp\mu_2$ is an analytic arc, disjoint from $\supp\mu_1$ and connecting the branch points $a_2,b_2$.
 \item If $\alpha_c<\alpha<1/2$ then none of the measures is zero, and there exists a value $a_*\in (a_1,b_1)$, determined by the condition 
\begin{equation}\label{definition_a_star_supercritical}
\im\int_{b_2}^{a_*}(\xi_1(s)-\xi_3(s))ds=0,
\end{equation}
for which $\supp\mu_1=[a_*,b_1]$, $\supp\mu_3=[a_1,a_*]$ and $\supp\mu_*=\gamma_1\cup\gamma_2$, where $\gamma_1$ is an analytic arc on the upper half plane connecting $b_2$ to 
$a_*$, and $\gamma_2$ is its complex conjugate arc.
\end{itemize}
\end{thm}

\begin{figure}
\begin{subfigure}{.5\textwidth}
\centering
\begin{overpic}[scale=1]{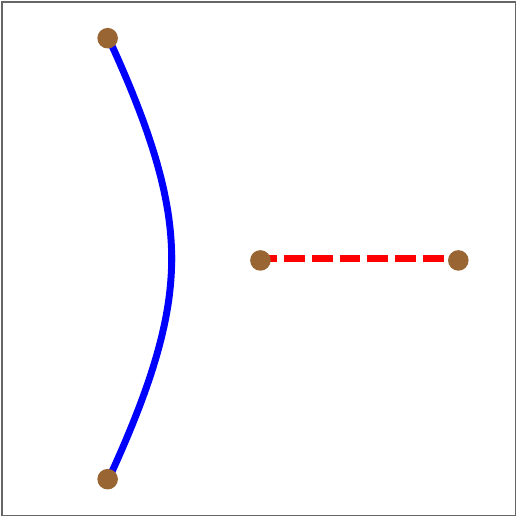}
\end{overpic}
\end{subfigure}%
\begin{subfigure}{.5\textwidth}
\centering
\begin{overpic}[scale=1]{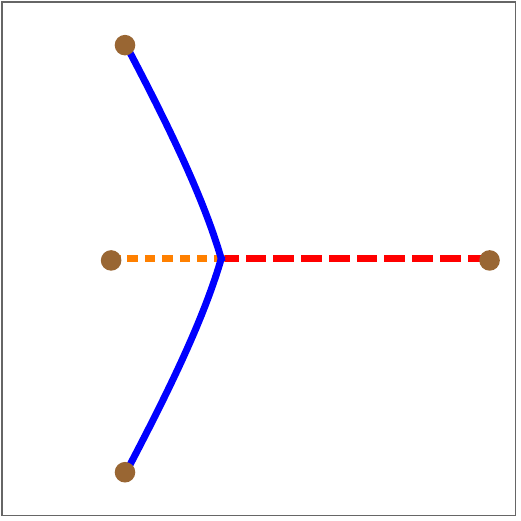}
\end{overpic}
\end{subfigure}
\caption{For $\tau\approx 0.126<\tau_c$ (left panel) and for $\tau\approx 0.227>\tau_c$ (right panel), numerical evaluation of the supports of the measures $\mu_1$ (long dashed 
line), $\mu_2$ (continuous line) and $\mu_3$ (short dashed line - only on the right panel). }\label{figure_supports}
\end{figure}

We refer to Figure~\ref{figure_supports} for an actual numerical evaluation of the supports of the measures $\mu_1,\mu_2,\mu_3$ in 
Theorem~\ref{thm_existence_critical_measures_cubic}.

\begin{remark}\label{restrictiononAlpha}
The restriction $\alpha<1/2$ arises somewhat naturally in our analysis. For $\alpha=1/2$, the algebraic equation \eqref{spectral_curve_general} with coefficients
\eqref{coefficients_cubic}--\eqref{definition_c_intro} still defines a vector-valued measure $\vec \mu_{1/2}$ in the sense of Theorem~\ref{thm_reciprocal_spectral_curve}, but the support 
of its third component is unbounded; hence, strictly speaking the measure $\vec \mu_{1/2}$ does not belong to $\mathcal M_{1/2}$. 
\end{remark}

In rough terms, Theorems~\ref{thm_spectral_curve} and \ref{thm_reciprocal_spectral_curve} induce a map 
$$
\{\mbox{critical vector-valued measures}\} \mapsto \{ \mbox{algebraic equations} \}
$$
and Theorems \ref{thm_genus_zero} and \ref{thm_existence_critical_measures_cubic} characterize this map completely for a subclass of algebraic equations and the cubic potential 
\eqref{cubic_choice_potentials}.

The existence of the critical measure $\vec \mu$ assured by Theorem~\ref{thm_genus_zero} and the topological description of its support given by 
Theorem~\ref{thm_existence_critical_measures_cubic} follow after a careful analysis of the associated quadratic differential $\varpi$ in \eqref{defQqd}. We describe the 
critical graph of $\varpi$ for $\alpha=0$ and indicate its subsets corresponding to the supports $\supp\mu_1,\supp\mu_2, \supp\mu_3$. In other words, we lift the supports to the 
associated Riemann 
surface, embedding them into the critical graph of $\varpi$. We then deform the parameter $\alpha$ in the interval $(0,1/2)$, observing the dynamics of the critical graph and 
keeping track of the supports of the measures.

During this deformation procedure, the critical graph of $\varpi$ displays several transitions, described in 
Sections~\ref{section_degenerate}--\ref{section_trajectories_tau2_1-4}, but only one of such transitions, governed by equation \eqref{equation_definition_alphac}, has direct 
impact 
on the topology of $\supp\mu_1,\supp\mu_2,\supp\mu_3$.

It should be stressed that there are not many tools for such an analysis (we have summarized the basic facts about quadratic differentials in 
the Appendix \ref{appendix_quadratic_differentials}), and that it can be extremely complicated, even on the Riemann sphere. The situation becomes even more involved when we 
consider trajectories on a more general compact Riemann surface. We believe that the methodology we have developed is of independent interest, and can be applied 
in other situations as well.

\subsection{Critical measures and max-min problems}\label{sec:max_min}

The critical measure given by Theorem~\ref{thm_genus_zero} is crucial in the study of multiple orthogonal polynomial
$P_{n,m}(z)=z^{N}+\hdots$ defined through
\begin{equation}\label{conditions_multiple_orthogonality}
\begin{aligned}
\int_{\Sigma_1}z^j P_{n,m}(z)e^{-Nz^3}dz=0, & \quad j=0,\hdots,n-1, \\
\int_{\Sigma_2}z^j P_{n,m}(z)e^{-Nz^3}dz=0, & \quad j=0,\hdots,m-1, \\
\end{aligned}
\end{equation}
where $N=n+m$ and $\Sigma_1,\Sigma_2$ are contours extending to $\infty$ with angles $-\frac{2\pi}{3}$, $0$ and $-\frac{2\pi}{3},\frac{2\pi}{3}$, respectively. It turns out that  
the asymptotic behavior of $P_{n,m}$ as $m,n\to \infty$ in such a way that $n/N\to \alpha\in (0,1/2)$, is governed by this critical vector-valued measures $\vec \mu\in \mathcal 
M_\alpha$: under this regime, the sequence of zero counting measures associated with $(P_{n,m})$ converges to the sum of the first two components of $\vec \mu$. 
The two extremal cases, $\alpha=0$ and $\alpha=1/2$, have been proved in \cite{deano_kuijlaars_huybrechs_complex_orthogonal_polynomials} and 
\cite{filipuk_vanassche_zhang}, respectively, and the general case $\alpha\in (0,1/2)$ will be addressed in a follow-up publication \cite{martinez_silva2}. 

As we already mentioned, zeros of orthogonal polynomials on the real line are characterized as solutions to logarithmic equilibrium problems. The measure $\vec\mu$ given by 
Theorem~\ref{thm_genus_zero}, and hence the zeros of the multiple orthogonal polynomials \eqref{conditions_multiple_orthogonality}, also admits a vector equilibrium problem 
characterization, as explained next.

Let $\mathcal T$ be the class of contours $\Sigma$ extending to $\infty$ with angles $-\frac{2\pi}{3},\frac{2\pi}{3}$ and intersecting the real axis exactly once, say 
at $a_*=a_*(\Sigma)$. Denote by $\mathcal M_\alpha(\Sigma)$ the subset of measures $\vec\nu=(\nu_1,\nu_2,\nu_3)\in \mathcal M_\alpha$ that further satisfy
$$
\supp\nu_1\subset [a_*,+\infty),\quad \supp\nu_2 \subset \Sigma,\quad \supp\nu_3\subset (-\infty,a_*].
$$

The $\alpha$-{\it equilibrium measure} $\vec\nu_\alpha=\vec\nu_{\alpha,\Sigma}$ of $\Sigma \in \mathcal T$ (for the energy functional $E(\cdot)$ given in \eqref{energy_functional} 
with \eqref{cubic_choice_potentials}--\eqref{external_fields_cubic_potential}) is the unique measure $\vec\nu_\alpha \in \mathcal M_\alpha(\Sigma)$ for which
$$
E(\vec\nu_\alpha)=\inf_{\vec\nu \in \mathcal M_\alpha(\Sigma)} E(\nu)=:E_\alpha(\Sigma).
$$

\begin{thm}\label{thm_equilibrium_measure_cubic}
For $\alpha\in (0,1/2)$, there exists a contour $\Gamma\in \mathcal T$ for which its $\alpha$-equilibrium measure is the $\alpha$-critical measure $\vec \mu$ given by 
Theorem~\ref{thm_genus_zero}.
\end{thm}

We conjecture that the contour $\Gamma$ is characterized by the max-min property
$$
E_\alpha(\Gamma)=\sup_{\Sigma\in \mathcal T}E_\alpha(\Sigma).
$$

This characterization would be analogous to the max-min property for non-hermitian orthogonal polynomials 
\cite{rakhmanov_orthogonal_s_curves,kuijlaars_silva,stahl_orthogonal_polynomials_complex_measures,stahl_orthogonal_polynomials_complex_weight_function}. Furthermore, an adaptation 
of the methods in \cite{martinez_rakhmanov,kuijlaars_silva} shows that the $\alpha$-equilibrium measure of any contour solving this max-min problem is in fact an $\alpha$-critical 
measure.

\subsection{Structure of the rest of the paper}

Theorems~\ref{thm_spectral_curve}, \ref{thm_s_property}, 
\ref{thm_quadratic_differential}, and a simplified version of Theorem~\ref{thm_reciprocal_spectral_curve} are proven in Section~\ref{section_general_variational} (the complete proof of Theorem~\ref{thm_reciprocal_spectral_curve} is rather technical and thus deferred to Appendix~\ref{appendix1}). In Section~\ref{section_cubic_spectral_curve} we derive the coefficients 
\eqref{coefficients_cubic}--\eqref{definition_c_intro} and prove the second part of Theorem~\ref{thm_genus_zero}. Assuming the existence of certain short trajectories (fact proved later, in Section~\ref{section:dynamics}), we derive the first part of Theorem~\ref{thm_genus_zero} and 
Theorem~\ref{thm_equilibrium_conditions_cubic} in Section~\ref{section_spectral_curve_to_variational_equations}, which yields Theorem~\ref{thm_equilibrium_measure_cubic}. Section~\ref{section:dynamics} is devoted to the 
analysis of the global structure of the quadratic differential \eqref{defQqd} for the cubic potentials \eqref{cubic_choice_potentials}, which also provides a proof of 
Theorem~\ref{thm_existence_critical_measures_cubic} (or equivalently, of Theorem~\ref{thm:criticaltraj}, where the structure of 
the support of the critical vector-valued  measure is rephrased in the terminology of quadratic differentials). Finally, two more appendices contain the minimal background on quadratic differentials (Appendix~\ref{appendix_quadratic_differentials}) and the description of the numerical experiments and procedures used in this paper (Appendix~\ref{appendix_numerics}).

\section{Critical vector-valued measures} \label{section_general_variational}

Critical measures were defined in the previous section in terms of the vanishing of the total energy \eqref{variations_critical_measure}. Here we use an alternative 
characterization, more convenient for calculations. When $\vec\mu=\mu$ is a scalar measure, the following 
Proposition~\ref{proposition_local_variation_energy} coincides with \cite[Lemma~3.1]{martinez_rakhmanov}. The proof extends trivially to vectors of measures $\vec\mu$ and 
polynomial external fields as considered here. We skip the details.

\begin{prop}\label{proposition_local_variation_energy}
 For any vector of measures $\vec\mu\in \mathcal M_\alpha$, it is valid
 \begin{equation*}
 E(\vec\mu^t)=E(\vec\mu)+t\re \mathcal D_h(\vec\mu)+\Boh(t^2), \quad t\to 0,
 \end{equation*}
 where
 \begin{equation}\label{variations_energy}
 \mathcal D_h(\vec\mu)=-\sum_{j,k=1}^3 a_{j,k}\iint \frac{h(x)-h(y)}{x-y}d\mu_j(x)d\mu_k(y)+\sum_{j=1}^3\int \Phi_j'(x) h(x)d\mu_j(x).
 \end{equation}
\end{prop}

Although the definition of critical measures only requires to evaluate $\mathcal D_h(\vec \mu)$ for test functions $h\in C^2(\C)$, the quantity $\mathcal D_h(\vec \mu)$ as 
defined above will be used for any function defined on $\supp\mu_1\cup\supp\mu_2\cup\supp\mu_3$, as long as the integrals on the right-hand side of 
\eqref{variations_energy} exist.

In particular, considering $h$ and $-ih$, we get from Proposition~\ref{proposition_local_variation_energy} 

\begin{cor}\label{corollary_critical_measure}
A measure $\vec\mu\in \mathcal M_\alpha$ is critical if, and only if, 
$$
\mathcal  D_h(\vec\mu)=0,
$$
for every function $h\in C^2(\C)$.
\end{cor}

The following lemma is inspired by \cite{martinez_rakhmanov}, where a similar situation (but involving only one measure) was analyzed, and borrowed in the following 
form from \cite[Lemma~3.5]{kuijlaars_silva}, where a simplified proof is given. The proof from \cite{kuijlaars_silva} is easily extended to our situation by simply choosing 
$g\equiv 1$ therein and mimicking the arguments.

\begin{lem}\label{lemma_extension_test_functions}
 Suppose $\vec\mu$ is a critical vector-valued measure. If $z\in \C$ satisfies
 \begin{equation}\label{absolute_convergence_cauchy_transform}
 \int\frac{d\mu_j(x)}{|x-z|}<\infty,\quad j=1,2,3,
 \end{equation}
 then $\mathcal D_{h_z}(\vec \mu)=0$ for the meromorphic function
 \begin{equation}\label{shiffer_variation}
 h_z(x)=\frac{1}{x-z},\quad x\in \C.
 \end{equation}
\end{lem}

The next result is straightforward, and we skip its proof:
\begin{lem}\label{lemma_integrability_cauchy_transform}
 Let $f$ be holomorphic in a punctured neighborhood of $p\in \C$. If $|f(z)|^r$ is $\lebesgue_2$-locally integrable at $p$ for some $r>0$, then $p$ is either a removable 
singularity or a 
pole of $f$ of order $<2/r$.
\end{lem}

For any finite Borel measure $\mu$, it follows from Tonelli's theorem that the function
$$
z\mapsto \int \frac{d\mu(x)}{|x-z|}
$$
is locally $\lebesgue_2$-integrable and in particular finite $\lebesgue_2$-a.e.; that is, \eqref{absolute_convergence_cauchy_transform} is satisfied for $\lebesgue_2$-a.e. $z\in 
\C$. Hence the Cauchy transform $C^{\mu}$ of the measure $\mu$ is finite $\lebesgue_2$-a.e. 
In particular, the functions $\xi_1,\xi_2,\xi_3$ in \eqref{definition_xi_functions} are finite $\lebesgue_2$-a.e. and locally integrable with respect to $\lebesgue_2$, so that
\begin{equation}\label{first_coefficient_polynomial}
\sum_{j=1}^3 \xi_j^2 (z)
\end{equation}
is well-defined $\lebesgue_2$-a.e. and is analytic on open subsets of $\C\setminus(\supp\mu_1\cup\supp\mu_2\cup\supp\mu_3)$.

Recall the quantities $R$, $\mathcal D_h(\vec \mu)$ and $h_z$ defined in \eqref{expression_R}, \eqref{variations_energy} and \eqref{shiffer_variation}, respectively. The next 
lemma is crucial for what comes later:
\begin{lem}\label{lemma_polynomial_coefficient}
If $\vec \mu\in \mathcal M_\alpha$, then the functions $\xi_j$ defined in \eqref{definition_xi_functions} satisfy the identity
\begin{equation}\label{identity_energy_polynomials}
	\frac{1}{2}\left(\xi_1^2 + \xi_2^2 +\xi_3^2 \right)(z) = R(z) + \mathcal D_{h_z}(\vec \mu) \quad \text{$\lebesgue_2$-a.e.}
	\end{equation}
Moreover, the statements
\begin{enumerate}
	\item[(i)] $\sum_{j=1}^3 \xi_j^2 (z)$ is a polynomial,
	\item[(ii)] $\sum_{j=1}^3 \xi_j^2 (z) = 2 R(z)$ $\lebesgue_2$-a.e.
	\item[(iii)] $\mathcal  D_{h_z}(\vec \mu)=0$ for every $z$ satisfying \eqref{absolute_convergence_cauchy_transform},
\end{enumerate}
are equivalent. 
\end{lem}

\begin{proof}
For $h=h_z$,
\begin{equation}\label{h_variation_step_1}
\begin{aligned}
\iint \frac{h_z(x)-h_z(y)}{x-y}d\mu_j(x)d\mu_k(y)& =-\iint \frac{1}{(x-z)(y-z)}d\mu_j(x)d\mu_k(y) \\
					     & =-C^{\mu_j}(z)C^{\mu_k}(z),
\end{aligned}
\end{equation}
and
\begin{equation}\label{identity_single_integral}
	\int h_z(x)\Phi_j'(x)d\mu_j(x)=\Phi'_j(z)C^{\mu_j}(z)-Q_j(z),
\end{equation}
where
\begin{equation}\label{h_variation_step_2}
	Q_j(z)=-\int \frac{\Phi'_j(x)-\Phi'_j(z)}{x-z}d\mu_j(x).
\end{equation}
Notice that the integrand in the right-hand side of \eqref{h_variation_step_2} is a polynomial both in $x$ and in $z$, and thus (since $\mu_j$ is finite and compactly supported) it 
is also $d\mu_j(x)$-integrable for all $z$. Hence, $Q_j$ is a bona fide polynomial. Additionally, 
\eqref{h_variation_step_1} and \eqref{identity_single_integral} also show that $D_{h_z}(\vec\mu)$ is finite whenever \eqref{absolute_convergence_cauchy_transform} holds true.

Using \eqref{h_variation_step_1} and \eqref{identity_single_integral} in \eqref{variations_energy}, we get that
\begin{equation*}
\mathcal D_{h_z}(\vec \mu)=\sum_{j,k=1}^3 a_{j,k} C^{\mu_j}(z)C^{\mu_k}(z)+\sum_{j=1}^3 \Phi_j'(z)C^{\mu_j}(z)-\sum_{j=1}^3 Q_j(z),
\end{equation*}
or equivalently,
\begin{equation}\label{aux_equation_1}
\mathcal D_{h_z}(\vec \mu)=\vec C(z)^T A  \vec C(z) +  (\vec\Phi'(z))^T \vec C(z)-\sum_{j=1}^3 Q_j(z),
\end{equation}
with the previously used notation $\vec C = \left(C^{\mu_1}, C^{\mu_2}, C^{\mu_3} \right)^T$ and $\vec \Phi'=\left(\Phi'_1,  \Phi'_2, \Phi'_3 \right)^T$. 

On the other hand, by \eqref{definitionofXivec} (and taking also into account both identities from \eqref{identitiesforA}) we get
\begin{align*}
	\frac{1}{2}\left(\xi_1^2+\xi_2^2+\xi_3^2 \right) &= \frac{1}{2} \vec \xi^T \vec \xi = \left( \frac{1}{3}\vec \Phi' + \vec C \right)^T A \left( \frac{1}{3}\vec \Phi' + \vec 
C \right) \\
	&= \frac{1}{9}(\vec \Phi')^T A \vec \Phi' +    \vec C^T A  \vec C  +\frac{2}{3} \left(  \vec \Phi'   \right)^T A \vec C \\
	&= \frac{1}{9}(\vec \Phi')^T A \vec \Phi' +    \vec C^T A  \vec C  +\frac{1}{3} \left(  \vec \Phi'   \right)^T \left(3I_3 - b b^T \right) \vec C\\
	&= \left\{\vec C^T A  \vec C  + \left(  \vec \Phi'   \right)^T \vec C \right\} + \frac{1}{9}(\vec \Phi')^T A \vec \Phi' -     \frac{1}{3} \left( b^T \vec \Phi'   \right)^T 
  \left( b^T   \vec C\right).
\end{align*}
Notice that the compatibility condition \eqref{compatibility_condition_external_fields} means precisely that $b^T \vec \Phi' =\vec 0$, so that 
$$
\frac{1}{2}\left(\xi_1^2+\xi_2^2+\xi_3^2 \right) =   \frac{1}{9}(\vec \Phi')^T A \vec \Phi' + \left\{\vec C^T A  \vec C  + \left(  \vec \Phi'   \right)^T \vec C \right\}, 
$$
and using \eqref{aux_equation_1} we can simplify the expression between brackets above to conclude that 
$$
\frac{1}{2}\left(\xi_1^2+\xi_2^2+\xi_3^2 \right)  = \mathcal D_{h_z}(\vec \mu) + \frac{1}{9}(\vec \Phi')^T A \vec \Phi' + \sum_{j=1}^3 Q_j=\mathcal D_{h_z}(\vec \mu) 
+R,
$$
where $R$ was defined in \eqref{expression_R}. This identity  is valid for values 
of $z$ satisfying \eqref{absolute_convergence_cauchy_transform}, which proves \eqref{identity_energy_polynomials}, as well as the equivalence of \emph{(ii)} and \emph{(iii)}.

Clearly, \emph{(ii)} implies \emph{(i)}. On the other hand, by \eqref{h_variation_step_1},
\begin{equation*}
\lim_{z\to \infty} \iint \frac{h_z(x)-h_z(y)}{x-y}d\mu_j(x)d\mu_k(y)=0,
\end{equation*}
and since function $h_z(\cdot)$ converges uniformly to $0$ as $z\to \infty$, by the Dominated Convergence Theorem we conclude that 
$$
\lim_{z\to \infty}\int h_{z}(x)\Phi'_j(x)d\mu_j(x)=0.
$$
Thus,
\begin{equation}\label{limit_variational_energy}
\lim_{z\to\infty} \mathcal D_{h_z}(\vec\mu)=0.
\end{equation}
Measures $\mu_1,\mu_2$ and $\mu_3$ are compactly 
supported, so that \eqref{identity_energy_polynomials} holds for all $z$ sufficiently large, and with \eqref{limit_variational_energy},
$$
\lim_{z\to \infty}	\left(\xi_1^2 + \xi_2^2 +\xi_3^2 - 2R \right)(z) = 0.
$$
This establishes that \emph{(i)} implies \emph{(ii)}.

\end{proof}


\begin{proof}[Proof of Theorem~\ref{thm_spectral_curve}]

A combination of Corollary~\ref{corollary_critical_measure} and Lemma~\ref{lemma_polynomial_coefficient} shows that if $\vec\mu$ is critical, then the polynomial $R$ in 
\eqref{expression_R} is alternatively expressed as
$$
R(z)=\frac{1}{2}(\xi_1(z)^2+\xi_2(z)^2+\xi_3(z)^2).
$$

From the expressions \eqref{definition_xi_functions} it follows that
\begin{equation}\label{sum0}
\xi_1+\xi_2+\xi_3=0,
\end{equation}
and as a consequence
\begin{equation}\label{aux_equation_3}
R=\frac{1}{2}\left(\xi_1^2+\xi_2^2+\xi_3^2 \right) = \xi_1^2+\xi_1\xi_2+\xi_2^2=\xi_2^2+\xi_2\xi_3+\xi_3^2=\xi_3^2+\xi_3\xi_1+\xi_1^2.
\end{equation}

Multiplying the identities in \eqref{aux_equation_3} by appropriate factors $(\xi_j-\xi_k)$, it follows that
\begin{equation}\label{aux_equation_4}
\xi_1(z)^3-\xi_1(z)R(z)=\xi_2(z)^3-\xi_2(z)R(z)=\xi_3(z)^3-\xi_3(z)R(z).
\end{equation}

Define
\begin{equation}\label{definition_D_xi_j}
D(z)=-\xi_j(z)^3+\xi_j(z)R(z).
\end{equation}

Due to \eqref{aux_equation_4}, $D$ does not depend on the choice of $j\in \{1,2,3\}$. When we choose $j=1$, we see that 
$D$ is analytic on $\C\setminus (\supp\mu_1\cup\supp\mu_2)$. Similarly, when we choose $j=2,3$, we see that $D$ should be analytic on $\C\setminus (\supp\mu_1\cup\supp\mu_3)$ and 
$\C\setminus(\supp\mu_2\cup\supp\mu_3)$, respectively. Consequently, $D$ can only have singularities at $S_\alpha$, defined 
in \eqref{defSetS}, so from 
\eqref{assumption_intersection_supports} these singularities are all isolated and in a finite number.

Moreover, 
\begin{align*}
|D(z)|^{1/3} & \leq (|\xi_j(z)|^3+|\xi_j(z)||R(z)|)^{1/3} \\
	     & \leq |\xi_j(z)| +|\xi_j(z)|^{1/3}|R(z)|^{1/3},
\end{align*}
where for the last inequality we used $(x+y)^{1/3}\leq x^{1/3}+y^{1/3}$, which is valid for any non negative numbers $x,y$.

The function $\xi_j$ is locally $\lebesgue_2$-integrable (see the comments after Lemma~\ref{lemma_integrability_cauchy_transform}). Hence the function $D$ satisfies the 
conditions of 
Lemma~\ref{lemma_integrability_cauchy_transform} for $r=1/3$ and any choice of $p\in S_\alpha$, so the points in $S_\alpha$ are not essential singularities of $D$. Moreover, from 
the behavior 
of the $\xi_j$'s and $R$ when $z\to\infty$, it follows that $D$ has polynomial growth at $\infty$. We conclude that $D$ is rational (although it could be identically zero).

In summary, we have shown that under the assumptions of Theorem~\ref{thm_spectral_curve}, there exist a polynomial $R$, defined in \eqref{expression_R}, and a rational function $D$ with possible poles in $S_\alpha$, such that the functions $\xi_1$, $\xi_2$, $\xi_3$ satisfy
\begin{equation*}
\xi_j(z)^3-R(z)\xi_j(z)+D(z)=0,\quad \lebesgue_2-\mbox{a.e.}
\end{equation*}
The fact that all the components $\mu_j$ are supported on a finite union of analytic arcs, and that they are absolutely continuous with respect to the arclength measure of their 
supports is then a direct consequence of \cite[Theorem 2]{borcea_et_al_algebraic_cauchy_transform}. 

Let $p$ be a pole of $D$. The functions $\xi_1,\xi_2,\xi_3$ satisfy the cubic equation 
\eqref{spectral_curve_general}, so their local behavior near $p$ is of the form
\begin{equation}\label{local_behavior_xi}
\xi_j(z)=(z-p)^{-\nu_j}(\kappa_j+\boh(1)),\quad \mbox{ as } z\to p, \quad j=1,2,3,
\end{equation}
where $\nu_1,\nu_2,\nu_3\in \Q$ and $\kappa_1,\kappa_2,\kappa_3$ are nonzero constants. Since $p$ is a pole of $D$ and $R$ is a polynomial, from \eqref{definition_D_xi_j}
we learn the $\nu_j$'s are all equal, say to $\nu$, and $3\nu$ should be equal to the order of $p$ as a pole of $D$. Moreover, we further get that none of the functions 
$\xi_1,\xi_2$ and $\xi_3$ is analytic near $p$, so $p$ has to belong to the support of at least two of the measures $\mu_1,\mu_2$ and $\mu_3$.
On the other hand, since $\xi_j$ is a linear combination of polynomials and Cauchy transforms of finite measures, we must also have 
$\nu<1$, so we conclude
$$
\nu\in \left\{\frac{1}{3}, \frac{2}{3}\right\}.
$$
The behavior \eqref{local_behavior_blow_up} then follows from Plemelj's formula and \eqref{local_behavior_xi}, concluding the proof.

\end{proof}

\begin{remark}
The fact that (scalar) finite measures, whose Cauchy transform satisfies a quadratic equation $\lebesgue_2$-q.e., live on a finite set of analytic curves was established in 
\cite{martinez_rakhmanov} without any additional constraint on the measure (the proof was inspired by \cite{MR1909635}). This was extended to an arbitrary algebraic equation in 
\cite{borcea_et_al_algebraic_cauchy_transform}, but with a crucial assumption on the measure that its support has a zero plane Lebesgue measure (this assumption is embedded in the 
definition of the class $\mathcal M_\alpha$ in Section~\ref{sec:mainresults}). Nevertheless, we believe that the assertion about the support is still true even if we drop this a 
priori restriction on $\vec \mu$.
\end{remark}

Our next goal is to prove Theorem~\ref{thm_reciprocal_spectral_curve}. In other words, for a vector of measures $\vec\mu\in\mathcal M_\alpha$ we assume that the components of $\vec\mu$ are supported on a finite union of analytic arcs, and the associated functions $\xi_j$ in \eqref{shifted_resolvents_pols_V} satisfy an algebraic equation of the form \eqref{spectral_curve_general} for some polynomial $R$ and some rational function $D$. As 
we already pointed out after the statement of Theorem~\ref{thm_reciprocal_spectral_curve}, these assumptions allow us to write $\vec\xi$ as in \eqref{definition_xi_functions} for the vector of 
external fields $\vec\Phi=(\Phi_1,\Phi_2,\Phi_3)$ given by \eqref{definition_new_external_fields}.
In particular, it follows that $\xi_j$'s satisfy \eqref{sum0} and consequently 
$$
R=\xi_1\xi_2+\xi_1\xi_3+\xi_2\xi_3=\frac{1}{2}(\xi_1^2+\xi_2^2+\xi_3^2).
$$
From Lemma~\ref{lemma_polynomial_coefficient} we get immediately the following result:
\begin{cor}\label{corollary_converse_critical_measure}
 Under the assumptions of Theorem~\ref{thm_reciprocal_spectral_curve}, the vector of measures $\vec \mu$ satisfies
 $$
 \mathcal D_{h_z}(\vec \mu)=0
 $$
 for $z\in \C\setminus(\supp\mu_1\cup\supp\mu_2\cup\supp\mu_3)$, where $h_z$ is as in \eqref{shiffer_variation}.
\end{cor}

Hence, the proof of Theorem~\ref{thm_reciprocal_spectral_curve} is reduced basically to extending the conclusion of Corollary~\ref{corollary_converse_critical_measure} to the whole 
class of functions $h\in C^2(\C)$. Although this is relatively straightforward for $h$ analytic in a neighborhood of the union of $\supp \mu_j$, it takes more effort beyond the 
analyticity due to the integrand $(h(x)-h(y))/(x-y)$ in \eqref{variations_energy}. In this case we need to appeal to a ``refined'' version of the celebrated Mergelyan's theorem 
that would allow for an approximation of a given sufficiently smooth function by a sequence of analytic ones, and in such a way that the sequence of derivatives is uniformly 
bounded. The complete proof we were able to find is rather technical, so for the sake of readability we postpone it to Appendix~\ref{appendix1}, and present here the arguments 
valid for the case when $h$ vanishes in a neighborhood of the set of double poles of $D$ (or when this set is empty). We formulate it as an independent proposition:

\begin{prop}\label{lemma_variations_without_double_poles}
	Let $\vec\mu\in \mathcal M_\alpha$ be a vector of measures satisfying the conditions of Theorem~\ref{thm_reciprocal_spectral_curve}. Suppose that $h\in C^2(\C)$ satisfies
	$$
	\supp h\cap \widehat S_\alpha =\emptyset,
	$$
	where $\widehat S_\alpha$ is the set of double poles of $D$. Then $\mathcal D_h(\vec\mu)=0$.
\end{prop}
\begin{proof}
	Recall that  the union of the support of the components of  $\vec\mu$  is a finite union of analytic 
	arcs, whose points of intersection are denoted by $S_\alpha$ (see \eqref{defSetS}), and  the poles of the 
	coefficient $D$ in \eqref{spectral_curve_general} belong to $S_\alpha$, as it follows from the arguments in \eqref{local_behavior_xi} {\it et seq}. Furthermore, if $p\in S_\alpha\setminus \widehat S_\alpha$, then
	\begin{equation}\label{local_behavior_density_outside_double_poles}
	\frac{d\mu_j}{ds}(z)=\Boh\left(|z-p|^{-\nu}\right),\quad \mbox{ as } z\to p, \quad j=1,2,3,
	\end{equation}
	for some $\nu\leq 1/3$ (possibly $\nu<0$).	
	
	Given $h$ as above, let us define 
	$$
	\vec g(z)=(g_1(z), g_2(z), g_3(z))^T \quad \text{and} \quad \vec H(z)=(H_1(z), H_2(z), H_3(z))^T,
	$$ 
	with
	$$
	g_j(z)   =\int \frac{h(x)}{x-z}d\mu_j(x),\quad H_j(z)   =\int \frac{h(x)-h(z)}{x-z}d\mu_j(x),\quad  j=1,2,3,
	$$
	where as usual we understand $g_j(z)$ in terms of its principal value (as in \eqref{principalvalue}), for $z\in\supp \mu_j$. With this convention and recalling 
	\eqref{definitionXialpha}, on any open analytic arc 
	$\Sigma$ from $\Xi_\alpha$, not containing any branch point of \eqref{spectral_curve_general}, by the Sokhotsky--Plemelj's formula we have
	\begin{equation}\label{boundaryvaluesg}
	g_{j\pm}(z)=\pm \pi i h(z) \mu_j'(z) + g_j(z),\quad z\in \Sigma,
	\end{equation}
	where $ \mu_j'=d\mu_j/ds$ is the Radon-Nikodym derivative of $\mu_j$ with respect to the line element on $\Sigma$, so that if $\Sigma\cap   \supp\mu_j=\emptyset$, we 
	have $\mu_j'=0$.
	
	Let $\Sigma$ be an open analytic arc from $\supp \mu_1 \cup \supp \mu_2 \cup \supp \mu_3$, not containing any branch point of \eqref{spectral_curve_general}. Observe that the 
	matrix $B$ in \eqref{identitiesforA} satisfies the identity $B^T=M_1-M_2$, where
	$$
	M_1=\begin{pmatrix}
	1 & 0 & 0 \\
	0 & 0 & -1 \\
	0 & -1 & 0 
	\end{pmatrix}, \quad M_2=\begin{pmatrix}
	0 & 1 & 0 \\
	-1 & 0 & 0 \\
	0 & 0 & -1
	\end{pmatrix}.
	$$
	Since $\vec H$ is continuous across $\Sigma$,  by \eqref{identitiesforA},
	\begin{equation}\label{identityforvecH}
	2A\vec H(z)= \left(M_1 -M_2 \right)B \vec H(z)=M_1 B \vec H_-(z)- M_2 B \vec H_+(z), \quad z\in \Sigma.
	\end{equation}
	Furthermore, the identity 
	$$
	\vec H(z)=\vec g(z) -h(z) \vec C(z)
	$$
	holds for all $z\in \C\setminus (\supp \mu_1 \cup \supp \mu_2 \cup \supp \mu_3)$, and we obtain that
	\begin{equation}\label{identityAH}
	2A\vec H(z) = \left(M_1 B \vec g_-- M_2 B \vec g_+ \right)(z)- h(z)\left(M_1 B \vec C_-- M_2 B \vec C_+ \right)(z).
	\end{equation}
	
	Recall that Equation~\eqref{compatibility_condition_pols_V} allows us to rewrite \eqref{shifted_resolvents_pols_V} as \eqref{definition_xi_functions} with the external fields 
	given in \eqref{definition_new_external_fields}. From  \eqref{definitionofXivec},
	\begin{equation*}
	B \vec C  = \vec \xi - \frac{1}{3} B   \vec \Phi'    ,
	\end{equation*}
	so that
	\begin{align}
	M_1 B \vec C_-(z)- M_2 B \vec C_+  (z) &=  \left(M_1 ( \vec \xi_- - \frac{1}{3} B   \vec \Phi')- M_2   ( \vec \xi_+ - \frac{1}{3} B   \vec \Phi') \right)(z) \nonumber\\
	& =    M_1   \vec \xi_- (z) - M_2  \vec \xi_+ (z)     - \frac{2}{3} A \vec \Phi'(z). \label{aux_equation_9}
	\end{align}
	
	Because the functions $\xi_j$'s satisfy a cubic equation, the boundary conditions \eqref{standard_orientation1} are valid and imply the equality 
	\begin{equation}\label{aux_equation_11}
	X M_1   \vec \xi_- (z) = X M_2  \vec \xi_+(z), \quad z\in \Sigma,
	\end{equation}
	where
	$$
	X=X(z)=
	\begin{pmatrix}
	\chi_1(z) & 0 & 0 \\
	0 & \chi_2(z) & 0 \\
	0 & 0 & \chi_3(z)
	\end{pmatrix},
	$$
	and $\chi_j$ is the characteristic function of $\supp\mu_j$, $j=1,2,3$. The matrix $X$ is piecewise constant; for ease of notation we suppress its $z$-dependence in the 
	following 
	computations.
	
	Since $\vec \Phi'$ is an eigenvector of $A$ corresponding to the eigenvalue $3/2$ (see Section~\ref{sec:mainresults}), we learn from \eqref{aux_equation_9} and 
	\eqref{aux_equation_11}
	\begin{equation}\label{identityforvecC}
	X M_1 B \vec C_-(z)- X M_2 B \vec C_+  (z) =         -  X\vec \Phi'(z), \quad z\in \Sigma.
	\end{equation}
	An immediate consequence of this identity, combined with Sokhotsky--Plemelj's formula $\vec C_+=2\pi i \vec \mu' + \vec C_-$, is
	\begin{equation}\label{identityforvecCswap}
	X M_1 B \vec C_+(z)-X M_2 B \vec C_-  (z) =    -  X \vec \Phi'(z)+2\pi i X(M_1+M_2) B \vec \mu'(z), \quad z\in \Sigma,
	\end{equation}
	where as usual, $ \vec \mu = (\mu_1, \mu_2, \mu_3)^T$, etcetera. 
	
	Multiplying \eqref{identityAH} by $X$ and using \eqref{identityforvecC},  we conclude that
	\begin{equation*}
	2XA\vec H(z) = X\left(M_1 B \vec g_-- M_2 B \vec g_+ \right)(z) +  h(z) X\vec \Phi'(z), \quad z\in \Sigma.
	\end{equation*}
	Since
	$$
	\left(M_1 B\right)^T = - M_2 B,\quad X^T=X,
	$$
	we can write it equivalently as
	\begin{equation}\label{identityAH2}
	2\left(\left(A\vec H\right)^T X\right)(z) = \left( \vec g_+ ^T  M_1 BX - \vec g_-^T M_2 BX  \right)(z) +  h(z)  \left( \vec  \Phi'(z)\right)^TX , \quad z\in \Sigma.
	\end{equation}
	
	Recall that  $\Sigma\subset \Xi_\alpha$ was arbitrary, so we can integrate this formula along $\supp \mu_1 \cup \supp \mu_2 \cup \supp \mu_3$.
	
	By \eqref{boundaryvaluesg}, for $j\in \{ 1, 2, 3\}$ we get
	\begin{equation}\label{thm_reciprocal_spectral_curve_eq_1}
	\int g_{j\pm}(y)d\mu_k(y)=\pm \pi i \int h(y) \mu_j'(y) d\mu_k(y)+   \int g_j(y)d\mu_k(y),
	\end{equation}
	where we recall that if $g_j$ is discontinuous across $\supp \mu_k$, we understand the 
	second integral in the right hand side above in terms of its principal value. By assumption, $\supp h$ does not contain double 
	poles of the coefficient $D$, and we use the condition \eqref{local_behavior_density_outside_double_poles} to get the behavior
	\begin{equation}\label{behavior_h_mu}
	h(y)\mu_j'(y)=\Boh(|y-p|^{-\nu}),\quad \mbox{ as } y\to p
	\end{equation}
	for some $\nu=\nu(p)\leq 1/3$ and any endpoint $p$ of the support of $\mu_k$. In particular, this implies that $h\mu_j'$ and $g_j$ are $\mu_k$-integrable, so the 
	integrals in the right-hand side of \eqref{thm_reciprocal_spectral_curve_eq_1} are convergent. Moreover, the behavior
	\eqref{behavior_h_mu} allows us to interchange the order of integration for the second integral in the right-hand side of \eqref{thm_reciprocal_spectral_curve_eq_1} (see  
	\cite[Equation~(20), page~25]{monakhov_book}) in order to get
	\begin{equation*}
	\int g_j(y)d\mu_k(y) =  -\int h(y)  C^{\mu_k}(y)d\mu_j(y).
	\end{equation*}
	Again by Sokhotsky--Plemelj's formula, 
	\begin{align*}
	\int h(y)  C^{\mu_k}(y)d\mu_j(y)& = \mp \pi i \int h(y) \mu'_k(y) d\mu_j(x)+\int h(y)C^{\mu_k}_\pm (y)d\mu_j(y),
	\end{align*}
	and using it in  \eqref{thm_reciprocal_spectral_curve_eq_1}, we finally conclude that
	\begin{equation*}
	\int g_{j\pm}(y)d\mu_k(y) =\pm 2\pi i \int \mu_j'(y)\mu_k'(y) h(y) dy-\int h(y)\mu_j'(y)C_\pm^{\mu_k}(y)dy.
	\end{equation*}
	
	A direct consequence of this formula is that for any $3\times 3$ constant matrix $\mathcal M$, 
	\begin{equation*}
	\int \vec g_{\pm}^T  \mathcal M d\vec \mu  =\pm 2\pi i \int h(y) (\vec \mu' (y))^T\mathcal M \vec \mu'  (y)  dy-\int h(y)(\vec \mu' (y))^T\mathcal M \vec C_\pm(y) dy.
	\end{equation*}
	In particular, we have
	\begin{align*}
	\int   \left( \vec g_+ ^T  M_1 B - \vec g_-^T M_2 B  \right)& d\vec \mu  =    2\pi i \int h(y) (\vec \mu' (y))^T \left( M_1   + M_2  \right)B\vec \mu'  (y)  dy \\
	& -  \int h(y)(\vec \mu' (y))^T \left( M_1 B\vec C_+(y)- M_2 B\vec C_-(y) \right) dy \\
	& = 2\pi i \int h(y) (\vec \mu' (y))^TX(y) \left( M_1   + M_2  \right)B\vec \mu'  (y)  dy\\
	& \quad -  \int h(y)(\vec \mu' (y))^T X(y)\left( M_1 B\vec C_+(y)- M_2 B\vec C_-(y) \right) dy \\
	& =  \int h(y)(\vec \mu' (y))^T    \vec \Phi'(y)    dy,
	\end{align*} 
	where we have used the trivial identity $\vec \mu'^TX=\vec \mu'^T$ for the last two equalities, and also \eqref{identityforvecCswap} for the last equality. Thus, integrating 
	\eqref{identityAH2} with respect to $d\vec \mu$, observing that $Xd\vec \mu=d\vec \mu$ and applying this last identity, we arrive at 
	$$
	\int \left(A\vec H\right)^T(z) d\vec \mu(z)=\int h(z)  \left( \vec  \Phi'(z)\right)^Td\vec \mu(z),
	$$
	which is a compact form of writing the condition $\mathcal D_h(\vec \mu)=0$. 
\end{proof}


\begin{proof}[Proof of Theorem~\ref{thm_s_property}]

For a vector-valued function $\vec f$  we understand by
$$
\int \vec f(z)dz
$$ 
the term by term integration, and we denote by $\vec \mu' = (\mu_1',\mu_2',\mu_3')$ the vector whose components are the densities of $\mu_1,\mu_2,\mu_3$ (with respect to the 
complex line element) along the arcs of $\Xi_\alpha$. Because the supports of the components of $\vec \mu$ are made of analytic arcs, and $\vec \mu'$ has analytic components, we 
conclude 
that $\vec U=\left(U^{\mu_1}, U^{\mu_2}, U^{\mu_3}\right)^T$ is continuous across the arcs of $\Xi_\alpha$.

Integrating \eqref{definitionofXivec} and using \eqref{derivative_potential} we get 
\begin{equation*}
\Re	\int^z \vec \xi(y)dy  =\Re \left( \frac{1}{3} B \vec \Phi (z)\right)+ B \vec U  (z)- \vec d,
\end{equation*}
where $\vec d$ is a real constant vector which only depends on the connected component of 
$\C\setminus\supp\mu_1\cup\supp\mu_2\cup\supp\mu_3$. 

Let $\Sigma$ be an open analytic arc from $\Xi_\alpha$, not containing any branch point of \eqref{spectral_curve_general}. 
%
Reasoning as for \eqref{identityforvecH}, we have
\begin{equation*}
2A\vec U(z)= \left(M_1 -M_2 \right)B \vec U(z)=M_1 B \vec U_-(z)- M_2 B \vec U_+(z), \quad z\in \Sigma.
\end{equation*}
Thus, for $z\in \Sigma$ and $\vec d_\pm$ the vector of constants for the component on the $\pm$-side of $\Sigma$,
\begin{align*}
2A\vec U(z) &=M_1 \left(\Re	\int^z \vec \xi_-(y)dy  - \Re \left( \frac{1}{3} B \vec \Phi (z)\right) + \vec d_- \right) \\
& - M_2 \left(\Re	\int^z \vec \xi_+(y)dy  - \Re \left( \frac{1}{3} B \vec \Phi (z)\right) + \vec d_+ \right)\\
& = \Re	\int^z \left(M_1 \vec \xi_-(y)- M_2 \vec \xi_+(y)\right) dy  - \frac{2}{3} \Re A \vec \Phi(z) + M_1\vec d_--M_2\vec d_+.
\end{align*}
Recalling that $A\vec \Phi' =(3/2)\vec \Phi'$ this last equation implies
$$
2A\vec U(z)=\Re\int^z \left(M_1 \vec \xi_-(y)- M_2 \vec \xi_+(y)\right) dy  - \Re\vec \Phi(z) + M_1\vec d_--M_2\vec d_+,\quad z\in \Sigma.
$$

If $\Sigma\subset \supp\mu_{4-j}$, then $\xi_j$ continuous across $\Sigma$, and
\eqref{standard_orientation1} tells us that the $(4-j)$-th entry of $M_1\xi_--M_2\xi_+$ vanishes, which in turn yields \eqref{variational_equations_critical_measures} (with 
$j$ replaced by $4-j$).

Note that due to \eqref{variational_equations_critical_measures} the $S$-property on $\supp\mu_{4-j}$ is equivalent to
\begin{equation}\label{s_property_critical_measures_disjoint_supports}
\left[\frac{\partial}{\partial z}\left(\sum_{k=1}^3a_{4-j,k}U^{\mu_k}+\frac{\phi_{4-j}}{2}\right)\right]_+ \\ =-\left[\frac{\partial}{\partial 
z}\left(\sum_{k=1}^3a_{4-j,k}U^{\mu_k}+\frac{\phi_{4-j}}{2}\right)\right]_-, 
\end{equation}

Equation \eqref{definitionofXivec} means that for $z\notin \Xi_\alpha$, 
$$
2\frac{\partial}{\partial   z} \left(A\vec U(z) +\frac{1}{2}\vec \phi(z)\right) =A \vec C(z)+ \frac{1}{2}\vec \Phi'(z)=\frac{1}{2}B^T \vec \xi(z),
$$
where we also used the identity involving $A$ and $B$ in \eqref{identitiesforA} and $A\vec \Phi' =(3/2)\vec \Phi'$.
As before, the boundary conditions \eqref{standard_orientation1} 
imply that the entry $(4-j)$ of  $B^T (\vec \xi_-  + \vec \xi_+)$ vanishes, and the equation above then implies \eqref{s_property_critical_measures_disjoint_supports}.

\end{proof}

\begin{proof}[Proof of Theorem~\ref{thm_quadratic_differential}]

We denote the restriction of $Q$ to the sheet $\mathcal R_j$ by $Q_j$. In cyclic notation mod $3$, $Q_j^2$ can be expressed as
$$
Q_j^2(z)=(\xi_{j-1}(z)-\xi_{j+1}(z))^2,\quad z\in \mathcal R_j,\quad j=1,2,3.
$$

Clearly $Q_j^2$ is meromorphic in $\overline\C\setminus (\supp\mu_1\cup\supp\mu_2\cup\supp\mu_3)$. We first prove that each $Q_j^2$ is meromorphic on the sheet $\mathcal R_j$.  If 
$\Gamma\subset \supp\mu_3$, then due to \eqref{assumption_intersection_supports} the function $\xi_1$ is analytic across $\Gamma$, hence 
$$
\xi_{2\pm}(s)=\xi_{3\mp}(s),\quad s\in \Gamma,
$$
implying that
$$
Q_{1+}^2(s)=(\xi_{2+}(s)-\xi_{3+}(s))^2=(\xi_{3-}(s)-\xi_{2-}(s))^2=Q_{1-}^2(s),\quad s\in \Gamma,
$$
so $Q_1^2$ is analytic across $\Gamma$ and, as a consequence, it follows that $Q_1^2$ is meromorphic on the whole sheet $\mathcal R_1$. Similarly we prove $Q_j^2$ is meromorphic 
on $\mathcal R_j$, $j=2,3$.

We now show that $Q^2$ is globally defined, that is, the function $Q_{j+1}^2$ is the analytic continuation of $Q_j^2$ from $\mathcal R_j$ to $\mathcal R_{j+1}$, $j=1,2,3$.

Consider an arc $\Gamma\subset \Xi_\alpha$. For the construction of $\mathcal R$ we know that $\Gamma$ connects exactly two sheets, which we assume to be $\mathcal R_1$ and 
$\mathcal R_2$, the remaining cases are analogous. Then $\Gamma\subset \supp\mu_1\setminus (\supp\mu_2\cup\supp\mu_3)$, so the function $\xi_3$ is analytic across $\Gamma$, hence 
$\xi_{1\pm}=\xi_{2\mp}$ on $\Gamma$ and
$$
Q_{1\pm}^2(s)=(\xi_{2\pm}(s)-\xi_{3}(s))^2=(\xi_{1\mp}(s)-\xi_{3}(s))^2=Q_{2\mp}^2(s),
$$
so $Q_2^2$ is the analytic continuation of $Q_1^2$ to $\mathcal R_2$ across $\Gamma$.

Thus, $Q^2$ is meromorphic along any arc connecting the sheets. 
It is clear that the functions $\xi_1,\xi_2,\xi_3$, and hence $Q^2$, can only be unbounded at the points at $\infty$ and also at the poles of the coefficient $D$ in 
\eqref{spectral_curve_general}. If $p$ is such a pole, then the three functions $\xi_1,\xi_2,\xi_3$ are branched at $p$, and the local 
coordinate of $\mathcal R$ at $z=p$ is of the form $z=p+u^3$, $u\in \C$, and hence
\begin{equation}\label{local_coordinate}
dz^2= 9u^4du^2. 
\end{equation}

Since $p$ is at most a double pole of $D$, it follows from the identities
$$
D(z)=-\xi_j(z)^3+R(z)\xi_j(z),\quad j=1,2,3,
$$
that the functions $\xi_1,\xi_2,\xi_3$ all blow up with the same order $\nu/3$, $\nu\in \{1,2\}$, that is
$$
\xi_j(z)=\frac{\kappa_j}{z^{\nu/3}}(1+\boh(1)),\quad \mbox{ as } z\to p,\quad j=1,2,3,
$$
for some nonzero constants $\kappa_1,\kappa_2,\kappa_3$, so
$$
Q^2(z)=\frac{c}{z^{2\nu/3}}(1+\boh(1))=\frac{c}{u^{2\nu}}(1+\boh(1)),\quad \mbox{ as } z\to p,
$$
for some constant $c$. Combining with \eqref{local_coordinate}, we get
$$
-Q(z)^2dz^2= \tilde c u^{4-2\nu}(1+\boh(1))du^2,
$$
for some constant $\tilde c$. Since $\nu\leq 2$, we get that $\varpi$ has to remain bounded near $p$, and hence $\varpi$ can only have poles at the points at $\infty$.
%

Let $\Sigma\subset \Xi_\alpha\cap  \supp \mu_1$ be an open arc. Due to assumption \eqref{assumption_intersection_supports}, function $\xi_3$ is analytic across $\Sigma$. Now the definition of  function $\xi_1$ in \eqref{definition_xi_functions}, the Sokhotsky-Plemelj's formula and \eqref{standard_orientation1} yield that for $s\in \Sigma$,  
$$
d\mu_1(s)  = d\mu_1(s)+d\mu_2(s) =  \frac{1}{2\pi i}\left(\xi_{1+}(s)-\xi_{1-}(s)\right)ds=  \frac{1}{2\pi i}\left(\xi_{1+}(s)-\xi_{2+}(s)\right)ds.
$$
Analogously, if $\Sigma\subset \Xi_\alpha\cap  \supp \mu_2$, then $\xi_2$ is analytic across $\Sigma$, and for $s\in \Sigma$, 
$$
d\mu_2(s)  = d\mu_1(s)+d\mu_2(s) =  \frac{1}{2\pi i}\left(\xi_{1+}(s)-\xi_{1-}(s)\right)ds=  \frac{1}{2\pi i}\left(\xi_{1+}(s)-\xi_{3+}(s)\right)ds.
$$
Finally, if $\Sigma\subset \Xi_\alpha\cap  \supp \mu_3$, then $\xi_1$ is analytic across $\Sigma$, and using again the definition of  $\xi_3$ in \eqref{definition_xi_functions} and the Sokhotsky-Plemelj's formula, we get that for $s\in \Sigma$, 
$$
d\mu_3(s)  = d\mu_3(s)- d\mu_1(s) =  \frac{1}{2\pi i}\left(\xi_{3+}(s)-\xi_{3-}(s)\right)ds=  \frac{1}{2\pi i}\left(\xi_{3+}(s)-\xi_{2+}(s)\right)ds.
$$
This shows that the measures $\mu_j$ must be supported on arcs of trajectories of $\varpi$.
\end{proof}

%

\section{The cubic case}\label{section_cubic_main}

In the following two sections we deal with the cubic case, and describe a one-parameter family of critical vector-valued measures for the energy functional 
\eqref{interactionmatrix}--\eqref{energy_functional} and for the choice \eqref{cubic_choice_potentials}, so that the external fields are given by 
\eqref{external_fields_cubic_potential}. Although the $\alpha$-critical measures were defined for $\alpha\in [0,1]$, here we restrict our attention to  $\alpha\in 
[0,1/2]$. As it will follow from our analysis below, as $\alpha \nearrow 1/2$, the support of the component $\mu_3$ of the $\alpha$-critical measure becomes unbounded. Furthermore, 
our original motivation was the asymptotic analysis of the multiple orthogonal polynomials \eqref{conditions_multiple_orthogonality}, for which the case of $\alpha\in [1/2,1]$ can 
be easily reduced to $\alpha\in [0,1/2]$ by an appropriate rotation of the plane. Thus, the selection of this interval for $\alpha$ is natural in the present situation, although 
our method carries over without any special difficulty to the whole range of $\alpha$.

\subsection{The spectral curve}\label{section_cubic_spectral_curve}

Recall that by Theorem~\ref{thm_spectral_curve} the shifted resolvents, defined in \eqref{definition_xi_functions}, of any critical vector-valued measure $\vec\mu=(\mu_1,\mu_2,\mu_3)$ 
satisfy the algebraic equation (``spectral curve'') \eqref{spectral_curve_general}. As a first step, we deduce the expressions 
\eqref{coefficients_cubic}--\eqref{definition_c_intro} for its coefficients.

For the potentials as in \eqref{external_fields_cubic_potential}, the coefficient $R$ in \eqref{expression_R} reduces to
\begin{align}\label{R_for_cubic}
 R(z) = 3z^4-3z-c,
\end{align}
where at this moment the constant $c$ is given in the form
$$
c=\int x(d\mu_1(x)+d\mu_2(x)).
$$

Since $D(z)= -\xi_1^3 + R(z)\xi_1=\xi_1 \xi_2 (\xi_1 + \xi_2)$, comparing the expansions of both expressions at $\infty$ and using \eqref{massconstraint}, we further get 
that
$$
C^{\mu_1}(z)+C^{\mu_2}(z)=-\frac{1}{z}-\frac{c}{3z^2}+\frac{1-\alpha (1-\alpha)}{3z^4}+\mathcal O\left( \frac{1}{z^5}\right), \quad z\to \infty,
$$
and
\begin{equation}\label{D_for_cubic}
D(z)=-2z^6+3z^3+cz^2-3\tau,\quad \tau :=\alpha(1-\alpha).
\end{equation}
Observe that $\tau=\alpha(1-\alpha)$ is an equivalent parametrization that gives a bijection between the interval $\alpha \in [0,\frac{1}{2}]$ and $\tau\in [0,\frac{1}{4}]$. As it 
was mentioned in Section~\ref{sec:max_min}, the extremal cases, $\tau=0$ and $\tau =\frac{1}{4}$, were studied in \cite{deano_kuijlaars_huybrechs_complex_orthogonal_polynomials} 
and \cite{filipuk_vanassche_zhang}, respectively, in their connection to the multiple orthogonal polynomials   \eqref{conditions_multiple_orthogonality}, for the choices $n=0$ and 
$n=m$, respectively, and the family of critical measures depending on $\alpha\in (0,1/2)$ are also relevant to the asymptotic analysis of these polynomials for general $n,m$. From 
a different perspective, we are studying a continuous deformation of the critical measures, interpolating the extremal cases of $\tau=0$ and $\tau =\frac{1}{4}$.

The discriminant of \eqref{spectral_curve_general} with respect to the variable $\xi$ is
\begin{align*}
\discr(z) & = 4R^3(z)-27D^2(z) \\
	  & = \begin{aligned}[t]
	  (81-324\tau)z^6+54c z^5 +9c^2 z^4 +54(9\tau-2)z^3 \\ 
	    + \, 54c(3\tau-2)z^2-36c^2z-4c^3-243\tau^2.
	   \end{aligned}
\end{align*}

Since $\discr(z)$ is a polynomial in $z$ of degree $6$, if the Riemann surface of \eqref{spectral_curve_general} has genus $0$ then, by the Riemann-Hurwitz theorem, $\discr(z)$ 
must have  a multiple root. In particular, the discriminant $\discr_1$ of $\discr(z)$ must vanish. A cumbersome but straightforward calculation (that can 
be carried out with the aid of a symbolic algebra software such as Mathematica) shows that
$$
\discr_1=a \tau p_1(c)p_2(c)^3
$$
for a nonzero real constant $a$ and
\begin{equation}\label{polynomials_genus}
\begin{aligned}
p_1(c) & = 64c^3+243(1-4\tau)^2,\\
p_2(c) & = c^6-486c^3\tau (1+\tau)+2187\tau (3\tau-1)^3,
\end{aligned}
\end{equation}
so we expect $c$ to be a root of either $p_1$ or $p_2$. For the case $\tau=\frac{1}{4}$ studied in \cite{filipuk_vanassche_zhang}, the algebraic equation
\eqref{spectral_curve_general}  reduces to
$$
\xi^3-(3z^4-3z)\xi-2z^6+3z^3-\frac{3}{4}=0,
$$
showing that $c=0$, which is a root of $p_1$. By continuity, we expect $c$ to be a root of $p_1$ for every choice of $\tau\in (0,\frac{1}{4})$,
concluding that
$$
c^3=-\frac{243}{64}(1-4\tau)^2,
$$
or
\begin{equation}\label{definition_c}
c=-\left(\frac{243}{64}(1-4\tau)^2\right)^{\frac{1}{3}}.
\end{equation}
which is the same as \eqref{definition_c_intro}.

Obviously, $c$ defined by \eqref{definition_c} can take three possible values; for the rest of the paper we choose $c$ to be real (and thus, negative), 
so that the algebraic equation \eqref{spectral_curve_general} is also real. It should be pointed out that a different choice of $c$ would lead
to an algebraic equation corresponding to another triplet of critical measures $\vec \mu$, that can be obtained from the original one by rotation by $\pm 2\pi/3$.

For the choice of $c$ in \eqref{definition_c}, we can rewrite $\discr(z)$ as
\begin{equation}
\label{defDiscriminant}
\discr(z)=\frac{243}{256}q_1(z)q_2(z)^2,
\end{equation}
where $q_1,q_2$ are given by
\begin{align}
\nonumber
q_1(z) = & \frac{256}{3^{5/3}} (1-4 \tau )^{{1/3}} z^4  +\frac{128}{9}z^3+\frac{16}{3^{1/3}} (1-4 \tau )^{2/3} z^2 \\
	      & - 3^{1/3}32(1-4 \tau )^{1/3} z + 16 (1-8 \tau )
	    , \label{definition_q_1} \\
 q_2(z)  = & 3^{1/3} (1-4 \tau )^{1/3}z-1\label{definition_q_2}.
\end{align}
For $\tau<\frac{1}{4}$ the discriminant of $q_1$ is  
$$
\const \times \tau(1-\tau)^2(27-100\tau)^3,
$$
which never vanishes, and we conclude that $q_1$ has always four distinct roots. For the choice $\tau=\frac{1}{8}$, $q_1$ has 
two complex conjugate roots and two real roots, so this also holds for any $\tau$.

The resultant of $ q_1$ and  $q_2$ is  
$$
\const \times (1-12 \tau )^4 (1-4 \tau)^{2/3},
$$
so that in the interval $0<\tau<1/4$, the polynomials $q_1$, $q_2$ share a root only when  $\tau=1/12$. It then follows that for $\tau\neq 1/12$ the roots of $q_1$ are  
branch points of multiplicity $2$ of \eqref{spectral_curve_general}.

If the double zero $b_*$ of $\discr$ (that is, the zero of $q_2$) is a branch point of \eqref{spectral_curve_general}, then its multiplicity (as a branch point) has to be 
three. This means that the three solutions to \eqref{spectral_curve_general} should coincide for $z=b_*$, so \eqref{spectral_curve_general} has to share a root with its second 
$\xi$-derivative, and hence $\xi_j(b_*)=0$, $j=,1,2,3$. Plugging this back into \eqref{spectral_curve_general} we see that $D(b_*)=0$. But
$$
D(b_*)=-\frac{(12 \tau -1)^3}{36 (1-4 \tau )^2},
$$
so $D(b_*)=0$ only for $\tau=1/12$. Hence for $\tau\neq 1/12$ the point $b_*$ is a regular point of \eqref{spectral_curve}, that is, the algebraic equation 
\eqref{spectral_curve_general} is not branched at $z=b_*$. We already observed that the simple zeros of $\discr$ are always branch points of multiplicity $2$, so the 
Riemann-Hurwitz formula says that for $\tau\neq 1/12$ the associated algebraic equation has genus $0$. Continuity with respect to $\tau$ assures that the 
genus cannot increase for $\tau=1/12$, and hence the genus is also zero for this value.

The discussion above can be summarized in the following proposition:
\begin{prop} \label{prop:algebraicCubic}
The algebraic equation \eqref{spectral_curve_general}  with coefficients given by \eqref{R_for_cubic}, \eqref{D_for_cubic} and \eqref{definition_c}, has four branch points (the 
zeros of the polynomial $q_1$ in \eqref{definition_q_1}) and a double point (the zero of $q_2$ in \eqref{definition_q_2}). Two of the branch points are 
real and the other two form a complex conjugate pair. For $\tau \neq 1/12$ all these points are distinct, while when
$\tau= 1/12$, the double point and one of the real branch points of \eqref{spectral_curve_general} coalesce. The associated Riemann surface has always genus $0$.
\end{prop}

Although the choice of $c$ in \eqref{definition_c}, as a root of $p_1$ instead of $p_2$, was mostly motivated by the construction of a continuous one-parameter family 
of critical measures $\vec \mu$ interpolating the extremal cases $\alpha=0$ and $\alpha=1/2$ studied in 
\cite{filipuk_vanassche_zhang,deano_kuijlaars_huybrechs_complex_orthogonal_polynomials}, formulas 
\eqref{R_for_cubic}, \eqref{D_for_cubic} and \eqref{definition_c} can also be explained in terms of the genus $0$ ansatz stated in Theorem~\ref{thm_genus_zero} that we prove next.

\begin{proof}[Proof of the second part of Theorem~\ref{thm_genus_zero}]
We already noted that if the Riemann surface has genus $0$, then $c$ must be a root of at least one of the polynomials $p_1$ and $p_2$ in \eqref{polynomials_genus}, and 
Proposition~\ref{prop:algebraicCubic} assures that the roots of $p_1$ give rise to genus $0$ Riemann surfaces. It thus suffices to show that zeros of $p_2$ give rise to a surface 
of genus $1$.

Consider $c$ to be a root of the polynomial $p_2$. As before, $c$ can be assumed to be real, the remaining nonreal choices of $c$ as a root of $p_2$ can be reduced 
to the real ones with suitable change of variables. Hence,
\begin{equation}\label{choice_c_genus_1}
c=3(9\tau +9\tau^2+\sqrt{3}(1+9\tau)\tau^{1/2})^{1/3} \quad \mbox{or} \quad c=3(9\tau +9\tau^2-\sqrt{3}(1+9\tau)\tau^{1/2})^{1/3}.
\end{equation}

Making the change of variables $\tau=u^6/3$, this can be written as
$$
c=c(u)=3u(u+1)(u^2-u+1), \quad 0<|u| \leq \left(\frac{3}{4}\right)^{1/6},
$$
where $u>0$ corresponds to the first value of $c$ in \eqref{choice_c_genus_1} and $u<0$ corresponds to the second choice of $c$ in \eqref{choice_c_genus_1}. 

With this identification, the discriminant $\discr(z)$ simplifies to
$$
\discr(z)=-27(u+z)^2q(z),
$$
where
\begin{multline*}
q(z) =  z^4(4u^6-3)-z^3(8u^7+6u^4)+z^2(9u^8+6u^5)\\
	  -z(10u^9+12u^6-4)+5u^{10}+12u^7+12u^4+4u.
\end{multline*}

The discriminant of $q$ is given by
$$
-6912 (u^3+1)^6 \left(10 u^3+9\right) \left(3 u^6+1\right)^3\neq 0 \quad \text{for } |u|\in (0,(3/4)^{1/6}).
$$
Moreover,
$$
q(-u)=9 u^4 \left(2 u^3+1\right)^2,
$$
hence $q(-u)\neq 0$, unless $u=-1/2^{1/3}$, which corresponds to $\tau=1/12$. But for this latter choice, the corresponding value of $c$ in \eqref{choice_c_genus_1} coincides with 
the value of $c$ in \eqref{definition_c}, hence the coefficients $R$ and $D$ are given by \eqref{coefficients_cubic}.

Thus, let $u\neq-1/2^{1/3}$. Then $\discr(z)$ has four simple roots, which have to be branch points of the equation, and one double root $z=-u$. At this double root, the 
algebraic equation \eqref{spectral_curve_general} simply reduces to $\xi^3=0$, so its three solutions coincide. This is only compatible with the 
fact that $\discr(z)$ has a double root at $z=-u$ if the three solutions are branched at this point. Hence we have four branch points of multiplicity $2$ and one branch point of 
multiplicity $3$, and the Riemann-Hurwitz formula gives us that the Riemann surface has genus $1$.
\end{proof}

The first part of Theorem~\ref{thm_genus_zero}, claiming existence of $\alpha$-critical measures, will be given by Corollary~\ref{corol_existence_critical_measure}.

\begin{remark} \label{rem:qdforDifferentC}
Obviously the quadratic differential \eqref{defQqd} still makes sense if $c$ is a root of $p_2$ as in the proof above. 
 Numerical experiments performed to compute its critical 
graph indicate that the corresponding algebraic equation should also give rise to $\alpha$-critical measures as in 
Theorem~\ref{thm_reciprocal_spectral_curve}.
\end{remark}

\subsection{Equilibrium problem from the spectral curve}\label{section_spectral_curve_to_variational_equations}

In this section, starting from the algebraic equation
\begin{equation}\label{spectral_curve}
\xi^3-R(z)\xi+D(z)=0,
\end{equation}
where it is assumed that the coefficients $R$ and $D$ are given by \eqref{R_for_cubic}--\eqref{definition_c}, we find a vector of measures $\vec{\mu}=(\mu_1,\mu_2,\mu_3)\in 
\mathcal M_\alpha$ for which the respective $\xi$-functions in \eqref{definition_xi_functions} satisfy \eqref{spectral_curve}, and consequently we prove 
Theorems~\ref{thm_existence_critical_measures_cubic} and \ref{thm_equilibrium_measure_cubic}. A central object for this analysis is the associated 
quadratic differential \eqref{defQqd} (see also \eqref{defQqd_cubic} {\it et seq.} below) and its critical graph, whose description is postponed to Section~\ref{section:dynamics}.

According to Proposition \ref{prop:algebraicCubic}, the spectral curve \eqref{spectral_curve} has two real branch points $a_1<b_1$, two non real branch points 
$b_2=\overline a_2$, 
with $\re b_2>0$, and a double point 
$$
b_*=\frac{1}{(3(1-4
   \tau) )^{1/3}}>0  \quad \text{for } \tau\in (0,1/4).
$$
Since
$$
 q_1(b_*)=\frac{32 (1-12 \tau )^2}{9(1-4 \tau)
   }\geq 0 \quad \text{for } \tau\in (0,1/4),
$$
we easily conclude that $b_*\geq b_1$ for $\tau\in (0,1/4)$, with $b_*= b_1$ if and only if $\tau=1/12$.

Moreover, equation \eqref{spectral_curve} defines a three-sheeted Riemann surface $\mathcal{R}$ of genus $0$ for every value of the parameter $\tau\in (0, 
1/4)$. In contrast to \eqref{construction_riemann_surface_general}, where the cuts for the Riemann surface are defined in terms of the supports of the critical measures, here 
the cuts that split $\mathcal{R}$ into the sheets $\mathcal R_1,\mathcal R_2,\mathcal R_3$ can be chosen in a somewhat arbitrary 
way, as long as they connect the branch points. In our case, the sheet structure will be given by two cuts: the interval of the real line connecting the real branch points $a_1, 
b_1$ of $\mathcal R$ and a Jordan arc (to be defined precisely below) that joins the complex-conjugate branch points $a_2,b_2$. This arc intersects $\R$ in a unique point $a_*$, 
whose location depends on the value of the parameter $\tau$. Namely, there is a certain critical value $\tau_c\approx 0.19$, defined later, such that $a_*\in(a_1,b_1)$ if and only 
if $\tau_c<\tau<1/4$ (what we call the \emph{supercritical regime}), see Figure~\ref{figure_sheet_structure}. It is important to stress here that $\tau_c>1/12$. 

Obviously, it ultimately follows from our arguments that this cut structure coincides with \eqref{construction_riemann_surface_general} for the critical measure given by 
Theorem~\ref{thm_existence_critical_measures_cubic}.

Using these cuts we define three oriented sets, $\Delta_1$, $\Delta_2$ and $\Delta_3$ on $\C$. Namely,
$\Delta_2$ is always the projection onto $\C$ of the cut joining the complex-conjugate branch points $a_2,b_2$ and oriented from $a_2$ to $b_2$. In the subcritical regime ($0<\tau 
<\tau_c$) we denote by $\Delta_1=[a_1, b_1]\subset \R$ oriented from $a_1$ to $b_1$, and set $\Delta_3=\emptyset$. In the supercritical regime ($\tau_c<\tau<1/4$), 
$\Delta_1=[a_*,b_1]\subset \R$ and $\Delta_3=[a_1,a_*]\subset \R$, both with the natural orientation. 
The orientation induces the left (denoted by the subscript ``$+$'') and right (with the subscript ``$-$'') boundary values of functions defined on $\C$ or $\mathcal{R}$.

We define
\begin{equation*}
\mathcal R_1=\overline \C\setminus (\Delta_1\cup \Delta_2),\quad \mathcal R_2=\overline \C\setminus (\Delta_1\cup \Delta_3),\quad \mathcal R_3=\overline\C\setminus (\Delta_2\cup 
\Delta_3),
 \end{equation*}
and build the Riemann surface $\mathcal R=\mathcal R_1 \cup \mathcal R_2 \cup \mathcal R_3$, associated to the algebraic equation \eqref{spectral_curve}, connecting each pair of 
sheets $\mathcal R_j$ crosswise across the cuts as follows: $\Delta_1$ connects $\mathcal R_1$ with $\mathcal R_{2}$, $\Delta_2$ connects $\mathcal R_1$ with $\mathcal R_3$, and 
$\Delta_3$ connects $\mathcal R_2$ with $\mathcal R_3$ (this last condition is clearly non-trivial only in the supercritical regime, when $\tau_c<\tau<1/4$), see again 
Figure~\ref{figure_sheet_structure}. 

As we  will see soon, the cuts $\Delta_1,\Delta_2,\Delta_3$ can be specified in such a way that we will be able to  define positive measures living on $\Delta_j$'s, and the 
construction of the sheet 
structure will coincide with \eqref{construction_riemann_surface_general}.

The three solutions $\xi_1,\xi_2,\xi_3$ of \eqref{spectral_curve} are enumerated according to their asymptotic expansion at infinity:
\begin{equation}\label{asymptotics_xi_functions}
\begin{split}
 & \xi_1(z)=2z^2-\frac{1}{z}+\Boh(z^{-2}),\\
 & \xi_2(z)=-z^2+\frac{\alpha}{z}+\Boh(z^{-2}),\\
 & \xi_3(z)=-z^2+\frac{1-\alpha}{z}+\Boh(z^{-2}).
\end{split}
\end{equation}

As before, the functions $\xi_1,\xi_2,\xi_3$ are regarded as branches of the same meromorphic function $\xi:\mathcal R\to \overline \C$, defined by \eqref{spectral_curve}. As it 
is rigorously given by Proposition~\ref{proposition_distribution_branchpoints} below, the function $\xi_j$ is the restriction of $\xi$ to the sheet $\mathcal R_j$ of the Riemann 
surface $\mathcal R$. 

 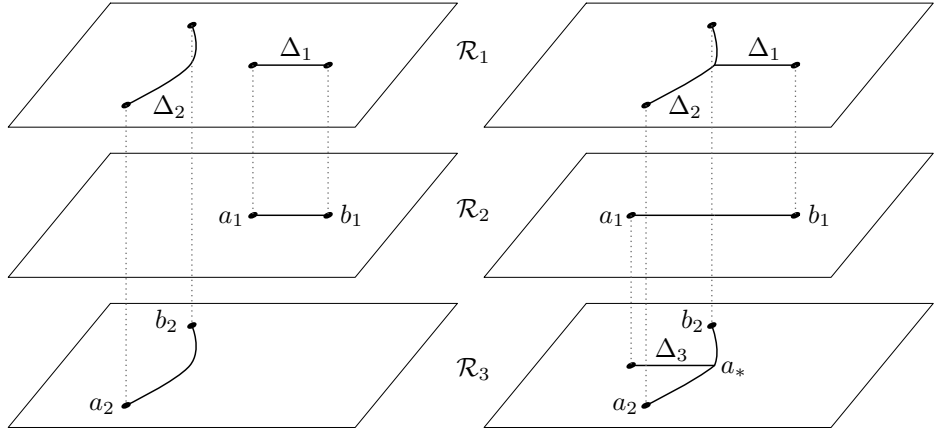
\begin{figure}[t]
\begin{center}
\begin{minipage}[c]{0.5\textwidth}
\begin{overpic}[scale=0.4,grid=false]{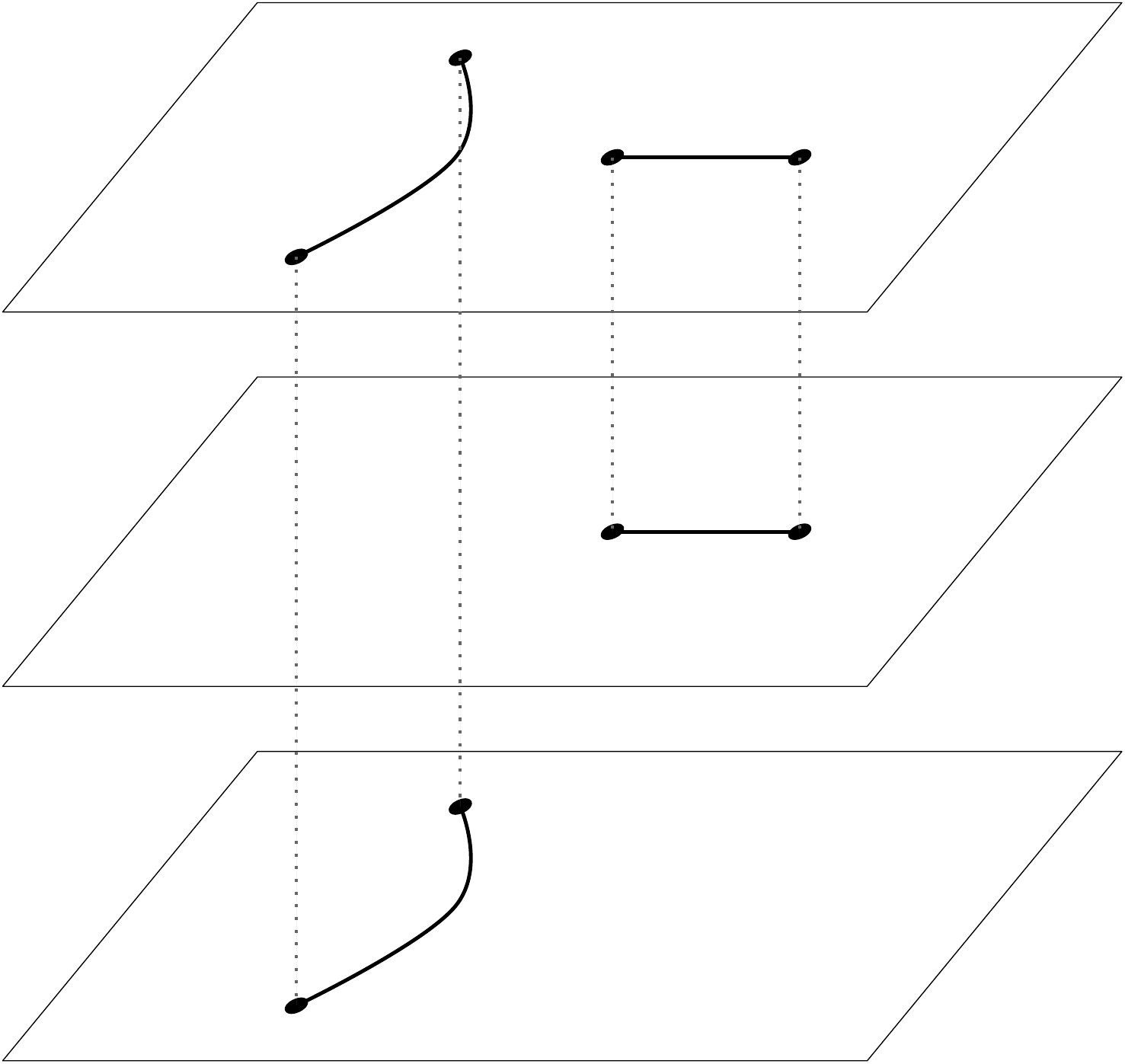}
\put(170,140){$ \mathcal R_1$}
\put(170,80){$ \mathcal R_2$}
\put(170,20){$ \mathcal R_3$}
\put(103,141){$ \Delta_1$}
\put(55,119){$ \Delta_2$}
\put(80,78){$ a_1$}
\put(126,78){$ b_1$}
\put(56,38){$ b_2$}
\put(31,7){$ a_2$}
\end{overpic}
\end{minipage}%
\begin{minipage}[c]{0.5\textwidth}
\begin{overpic}[scale=0.4]{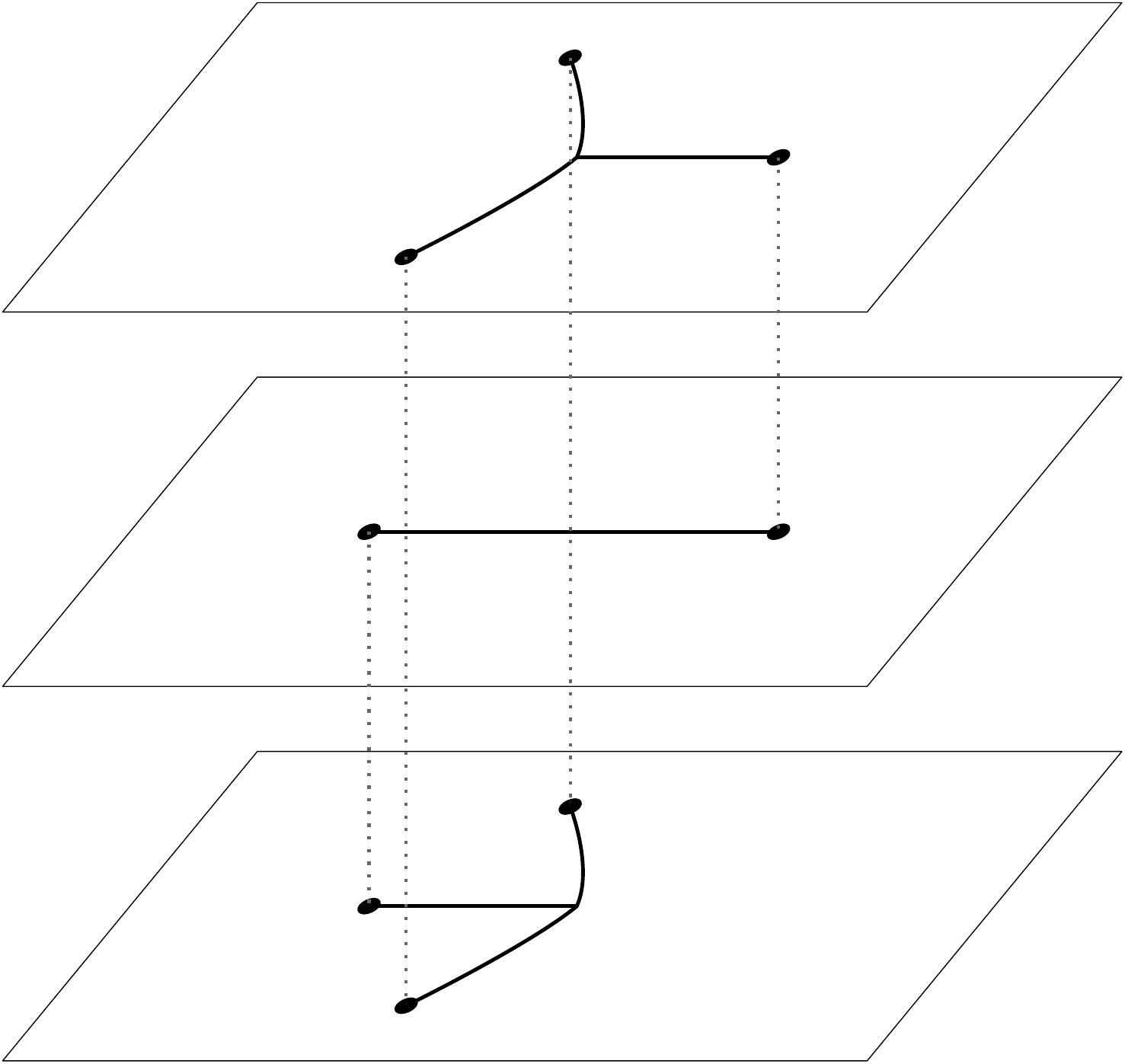}
\put(100,141){$ \Delta_1$}
\put(70,119){$ \Delta_2$}
\put(65,27){$ \Delta_3$}
\put(44,78){$ a_1$}
\put(123,78){$ b_1$}
\put(75,38){$ b_2$}
\put(49,7){$ a_2$}
\put(90,21){$ a_*$}
\end{overpic}
\end{minipage}%
\caption{Sheet structure in the \emph{precritical} $0<\tau<\tau_c$ (left) and \emph{supercritical} $\tau_c<\tau<1/4$ (right) regimes.}\label{figure_sheet_structure}
\end{center}
\end{figure}

Due to the explicit geometric description of the sets $\Delta_j$'s and sheets $\mathcal R_j$'s, it is straightforward to 
check that with
\begin{equation}
\label{defQqd_cubic}
Q(z)=
\begin{cases}
\xi_2(z)-\xi_3(z)  , & \text{ on } \mathcal R_1, \\
\xi_1(z)-\xi_3(z) , & \text{ on } \mathcal R_2, \\
\xi_1(z)-\xi_2(z) , & \text{ on } \mathcal R_3,
\end{cases}
\end{equation}
$Q^2$ extends as a single-valued meromorphic function on $\mathcal R$, and 
\begin{equation}
\label{qd}
\varpi=
-Q^2(z)\, dz^2
\end{equation}
is the corresponding rational quadratic differential on $\mathcal R$. The details are also carried out in the general setting of Theorem~\ref{thm_quadratic_differential} in its 
proof.

One of the main outcomes of the discussion in Section \ref{section:dynamics} (see Theorem~\ref{thm:criticaltraj} and Remark~\ref{remark_orthogonal_trajectories})
is that we can choose the cut $\Delta_2$ connecting $a_2,b_2$ to coincide with the trajectory along which
\begin{equation}\label{condition_trajectory_1}
\re \int^z(\xi_{1+}(s)-\xi_{3+}(s))ds = \const , \quad z\in \Delta_2.
\end{equation}
Furthermore, $\Delta_2$ can be extended to an analytic arc $\Gamma$ from $e^{-\frac{2\pi i}{3}}\infty$ to $e^{\frac{2\pi i}{3}}\infty$ in such a way that
\begin{equation}\label{extension_support_mu_2_property_1}
 \im \int^z(\xi_{1+}(s)-\xi_{3+}(s))ds= \const , \quad z\in\Gamma\setminus\Delta_2,
\end{equation}
and
\begin{equation}\label{extension_support_mu_2_property_2}
\xi_1(z)-\xi_3(z)\neq 0,\quad z\in \Gamma\setminus \{a_2,b_2\},
\end{equation}
see Figure~\ref{figure_geometry_support}. 

\begin{figure}[t]
 \centering
\begin{overpic}[scale=0.6]
{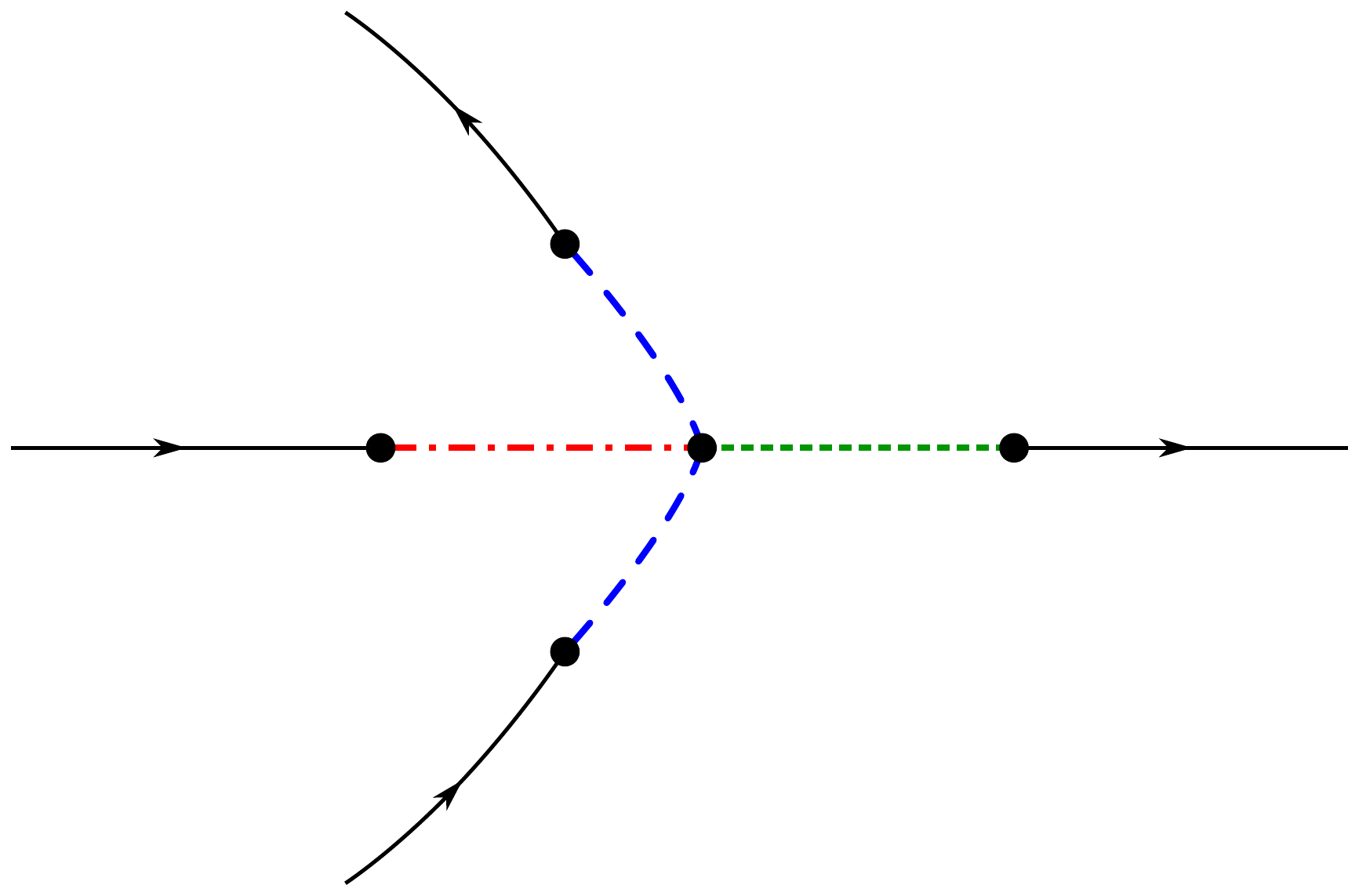}
\put(160,104){$a_*$}
\put(85,104){$a_1$}
\put(225,104){$b_1$}
\put(130,150){$b_2$}
\put(130,45){$a_2$}
\put(145,125){\textcolor{blue}{$\Delta_2$}}
\put(192.5,104){\textcolor{green}{$\Delta_1$}}
\put(122.5,104){\textcolor{red}{$\Delta_3$}}
\put(110,170){$\Gamma$}
\end{overpic}
\caption{Pictorial representation of the sets $\Delta_1$, $\Delta_2$, $\Delta_3$ and $\Gamma$ in the supercritical regime.}\label{figure_geometry_support}
\end{figure}
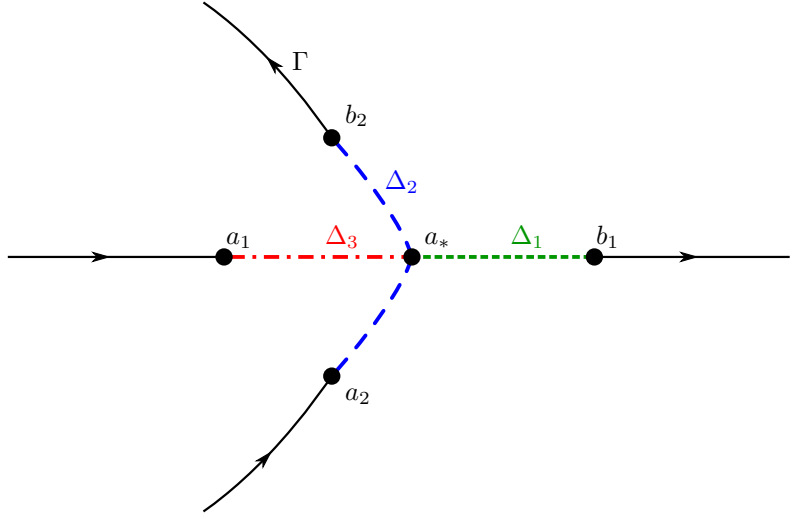

Further,
\begin{equation}\label{condition_trajectory_2}
 \re \int^z(\xi_{1+}(s)-\xi_{2+}(s))ds= \const,\quad z\in  \Delta_1,
\end{equation}
and
\begin{equation}\label{condition_trajectory_2_nonvanishing}
 \xi_{1+}(z)-\xi_{2+}(z)\neq 0,\quad z\in \Delta_1\setminus \{a_1,b_1\},
\end{equation}
and in the supercritical regime $\tau_c<\tau<1/4$, also
\begin{equation}\label{condition_trajectory_2bis}
 \re \int^z(\xi_{2+}(s)-\xi_{3+}(s))ds= \const,\quad z\in  \Delta_3,
\end{equation}
and 
$$
\xi_{2+}(s)-\xi_{3+}(s)\neq 0,\quad z\in \Delta_3\setminus \{a_1,a_*\},
$$
for which we refer to Proposition~\ref{proposition_distribution_branchpoints}.

Finally, it also holds
\begin{equation}\label{conditions_positive_solutions}
\begin{aligned}
\xi_{1}(z)-\xi_{2}(z)>0, & \quad z>b_1,\\ 
\xi_{2}(z)-\xi_{3}(z)>0, & \quad z<\min\{a_1,a_*\},
\end{aligned}
\end{equation}
and in the precritical regime $\tau<\tau_c$,
\begin{equation}\label{condition_positive_solution_2}
\xi_{1}(z)-\xi_{2}(z)<0, \quad z\in (a_*,a_1),
\end{equation}
for which again we refer to Proposition~\ref{proposition_distribution_branchpoints}.

We define measures $\mu_1$, $\mu_2$, $\mu_3$ on $\Delta_1$, $\Delta_2$, $\Delta_3$, respectively, through the formulas
\begin{equation}\label{definition_measures_mu}
\begin{split}
 d\mu_1(s) &=\frac{1}{2\pi i}(\xi_{1+}(s)-\xi_{2+}(s))ds, \quad s\in \Delta_1,\\
 d\mu_2(s) & =\frac{1}{2\pi i}(\xi_{1+}(s)-\xi_{3+}(s))ds, \quad s\in \Delta_2,\\[2mm]
\text{and } & \\
 d\mu_3(s) & =\begin{cases}
0, & \text{if } 0<\tau<\tau_c, \\[1mm]
\dfrac{1}{2\pi i}(\xi_{3+}(s)-\xi_{2+}(s))ds, \quad s\in \Delta_3, & \text{if } \tau_c<\tau<1/4, \\
 \end{cases}
\end{split}
\end{equation}
where $ds$ denotes the complex line element on the respective arc. Due to their construction, the supports of the measures $\mu_1,\mu_2,\mu_3$ satisfy the claims of 
Theorem~\ref{thm_existence_critical_measures_cubic} (see also Theorem~\ref{thm:criticaltraj}).

\begin{remark}
	Using the definition in \eqref{defQqd_cubic} we can describe formulas \eqref{definition_measures_mu} saying that we build $\mu_j$ from the values of $Q$ on the sheet 
$\mathcal R_{4-j}$, $j=1, 2, 3$.
\end{remark}

\begin{prop}
Expressions \eqref{definition_measures_mu} define positive measures $\mu_1$, $\mu_2$, $\mu_3$.
\end{prop}
\begin{proof}
We prove the statement for the supercritical regime $\tau_c<\tau<1/4$; the precritical regime, somewhat simpler, can be analyzed similarly.

 From \eqref{condition_trajectory_1} and \eqref{condition_trajectory_2} it follows that the measures $\mu_1$, $\mu_2$, $\mu_3$ are real. Moreover, the densities of $\mu_1$, 
$\mu_3$ with respect to the complex line element $ds$ are continuous and non vanishing in the interior of $\Delta_1$, $\Delta_3$, so the respective measures do not change sign. As 
for $\mu_2$, its density with respect to $ds$ is continuous when restricted to either the upper or the lower half plane, but not across $\R$. 

We start now with $\mu_1$. For $x\in\Delta_1$, we can write
\begin{equation}\label{aux_equation_5}
\mu_1([x,b_1])=-\frac{1}{2\pi i}\psi_{1+}(x),
\end{equation}
where
$$
\psi_1(z)=\int_{b_1}^z(\xi_1(s)-\xi_2(s))ds,\quad z\in \C\setminus\left((-\infty,b_1]\cup \Delta_2\right).
$$

By the asymptotic expansion  \eqref{asymptotics_xi_functions}, we have that
$$
\psi_1(z)=z^3+\Boh(z),\quad z\to \infty,
$$
and since $\xi_1-\xi_2$ does not change sign on $[b_1,+\infty)$, see \eqref{condition_trajectory_2_nonvanishing}, we conclude that
$$
\psi_1(x)>0,\quad x>b_1.
$$

The function $\psi_1$ vanishes at the point $z=b_1$ with order $3/2$. Since we already know that $\psi_1$ should map $\Delta_1$ to $i\R$, from its order of vanishing we further 
get that the imaginary part of $\psi_{1+}(x)$ is negative when $x\in \Delta_1$ is sufficiently close to the endpoint $b_1$. From 
\eqref{condition_trajectory_2_nonvanishing}, we conclude
$$
\psi_{1+}(x)\in i\R_-,\quad x\in\Delta_1.
$$
By \eqref{aux_equation_5}, this establishes the positivity of $\mu_1$.

Furthermore, the density $d\mu_3/ds$ is the analytic continuation of $ d\mu_1/ds$ across $\Delta_2$, hence $ d\mu_3 $ is also positive.

Finally, in order to establish the sign of $\mu_2$, consider first $x\in \Delta_2\cap\C_-$, $\C_-=\{z\in \C:\, \Im(z)<0 \}$. Denote by $\Delta_2[a_2,x]$ the subarc of $\Delta_2$ 
from $a_2$ to $x$. Then
\begin{equation}\label{aux_equation_6}
\mu_2(\Delta_2[a_2,x])=\frac{1}{2\pi i}\psi_{2+}(x),
\end{equation}
where
$$
\psi_2(z)=\int_{a_2}^{z}(\xi_{1+}(s)-\xi_{3+}(s))ds,\quad z\in \C\setminus((-\infty,b_1]\cup \Delta_2).
$$

From \eqref{extension_support_mu_2_property_1} and \eqref{extension_support_mu_2_property_2} we learn that $\Re( \psi_2)$ is monotone on $\Gamma(e^{-\frac{3\pi i}{2}}\infty,a_2)$.
The asymptotics
$$
\psi_2(z)=z^3+\Boh(z),\quad z\to \infty,
$$
shows that
$$
\re\psi_2(z)\to +\infty,\quad z\to \infty \mbox{ along } \Gamma,
$$
hence $\Re (\psi_2)$ is strictly decreasing on $\Gamma(^{-\frac{3\pi i}{2}}\infty,a_2)$, $\psi_2(a_2)=0$. Since $\psi_2$ vanishes with order $3/2$ on $a_2$, we conclude that
$$
\psi_{2+}(x)\in i\R_+,\quad x\in \Delta_2\cap \C_-.
$$
Thus, by \eqref{aux_equation_6} measure $\mu_2$ is positive on $\Delta_2\cap \C_-$. 

The positivity of $\mu_2$ on $\Delta_2\cap \C_+$ can be obtained by similar arguments or using the real symmetry of the density of $\mu_2$; we leave the details to the interested 
reader.

The proposition is proved.
\end{proof}

Our next goal is an expression for the Cauchy transforms of combinations of measures $\mu_1$, $\mu_2$, $\mu_3$:
\begin{prop}\label{prop_cauchy_transforms_cubic}
 The Cauchy transform of the measures $\mu_1$, $\mu_2$, $\mu_3$ defined in \eqref{definition_measures_mu} are related to the $\xi$-functions in 
\eqref{asymptotics_xi_functions} through
 \begin{align}
  & C^{\mu_1}(z)+C^{\mu_2}(z)+2z^2  =\xi_1(z),\quad z\in \C\setminus (\Delta_1\cup \Delta_2),  \label{equation_cauchy_transform_delta_1} \\
  & C^{\mu_1}(z)+C^{\mu_3}(z)+z^2   =-\xi_2(z),\quad z\in \C\setminus (\Delta_1\cup \Delta_3), \label{equation_cauchy_transform_delta_2} \\
  & C^{\mu_2}(z)-C^{\mu_3}(z)+z^2   =-\xi_3(z),\quad z\in \C\setminus (\Delta_2\cup \Delta_3). \label{equation_cauchy_transform_delta_3}
 \end{align}
In particular, the total masses of the measures $\mu_1$, $\mu_2$, $\mu_3$ satisfy \eqref{massconstraint}, namely
$$
|\mu_1|+|\mu_2|=1,\quad |\mu_1|+|\mu_3|=\alpha,\quad |\mu_2|-|\mu_3|=1-\alpha.
$$
\end{prop}
\begin{proof}
This is a straightforward consequence of residues calculations. For instance,
\begin{align*}
C^{\mu_1}(z)+C^{\mu_2}(z) & = \frac{1}{2\pi i} \int_{\Delta_1}\frac{\xi_{1+}(s)-\xi_{1-}(s)}{s-z}ds + \frac{1}{2\pi i} \int_{\Delta_2}\frac{\xi_{1+}(s)-\xi_{1-}(s)}{s-z}ds \\
			  & = \frac{1}{2\pi i} \varointclockwise \frac{\xi_1(s)}{s-z}ds \\
			  & =\res\left(\frac{\xi_1(s)}{s-z},s=z\right)+\res\left(\frac{\xi_1(s)}{s-z},s=\infty\right)\\
			  & = -2z^2+\xi_1(z).
\end{align*}
where the closed integral is computed along a contour oriented clockwise which encircles $\Delta_1\cup\Delta_2$ and does not encircle $z$.

Analogously,
\begin{align*}
C^{\mu_1}(z)+C^{\mu_3}(z) & =-\res\left(\frac{\xi_2(s)}{s-z},s=z\right)-\res\left(\frac{\xi_2(s)}{s-z},s=\infty\right)\\
			  & =-z^2-\xi_2(z),
\end{align*}
and taking the difference of these last two equations,
$$
C^{\mu_2}(z)-C^{\mu_3}(z)=-z^2+\xi_1(z)+\xi_2(z)=-z^2-\xi_3(z).
$$

Finally, \eqref{massconstraint} is a direct consequence of the asymptotic expansion \eqref{asymptotics_xi_functions}.
\end{proof}

A combination of Proposition~\ref{prop_cauchy_transforms_cubic} with Theorem~\ref{thm_reciprocal_spectral_curve} yields the first part of 
Theorem~\ref{thm_genus_zero}, namely
\begin{cor}\label{corol_existence_critical_measure}
 The vector of measures $\vec\mu =(\mu_1,\mu_2,\mu_3)\in \mathcal M_\alpha$ defined through \eqref{definition_measures_mu} is $\alpha$-critical for the potentials 
\eqref{external_fields_cubic_potential}.
\end{cor}

Recalling \eqref{derivative_potential}, the potential $U^\mu$ of a compactly supported signed measure $\mu$ for which $\C\setminus \supp\mu$ is connected is the real part of a 
primitive 
of $C^{\mu}$, that is
\begin{equation}\label{equation_recovery_potential}
U^{\mu}(z)=\re\int^z C^{\mu}(s)ds+c,\quad z\in \C\setminus \supp\mu,
\end{equation}
where the constant is chosen so as to have
$$
\lim_{z\to\infty} \left(\re\int^z C^{\mu}(s)ds+c\right)=0.
$$

Let us apply this to the measures $\mu_1$, $\mu_2$, $\mu_3$ given in \eqref{definition_measures_mu}. From \eqref{equation_cauchy_transform_delta_1},
$$
\int_{b_1}^z (C^{\mu_1}(s)+C^{\mu_2}(s))ds + \frac{2}{3} z^3-\frac{2}{3}b_1^3=\int_{b_1}^z \xi_1(s)ds,\quad z\in \C\setminus (\Delta_1\cup\Delta_2),
$$
and from \eqref{equation_cauchy_transform_delta_2},
$$
\int_{b_1}^z (C^{\mu_1}(s)+C^{\mu_3}(s))ds + \frac{1}{3} z^3-\frac{1}{3}b_1^3=-\int_{b_1}^z \xi_2(s)ds,\quad z\in \C\setminus (\Delta_1\cup\Delta_3).
$$

Summarizing,
\begin{multline*}
\int_{b_1}^z (2C^{\mu_1}(s)+C^{\mu_2}(s)+C^{\mu_3}(s))ds + z^3-b_1^3 \\ =\int_{b_1}^z (\xi_1(s)-\xi_2(s))ds,\quad z\in \C\setminus (\Delta_1\cup\Delta_2\cup \Delta_3),
\end{multline*}
where we must use the same paths of integration in the left and in the right hand sides. By  \eqref{equation_recovery_potential} we then conclude
\begin{equation}\label{aux_equation_7}
2U^{\mu_1}(z)+U^{\mu_2}(z)+U^{\mu_3}(z)+\phi(z)-l_1=\re \int_{b_1}^z(\xi_1(s)-\xi_2(s))ds, \quad z\in\C\setminus \Delta,
\end{equation}
for some constant $l_1$, the external field $\phi$ given in \eqref{external_fields_cubic_potential} and $\Delta=\Delta_1\cup\Delta_2\cup\Delta_3$. Since the $\xi_j$'s have purely 
imaginary periods, the right hand side above is well defined regardless the path chosen.

In a completely analogous way, we get
\begin{equation}\label{aux_equation_8}
\begin{aligned}
 & 2U^{\mu_2}(z)+U^{\mu_1}(z)-U^{\mu_3}(z)+\phi(z)-l_2=\re\int_{a_2}^z(\xi_1(s)-\xi_3(s))ds,\quad z\in \C\setminus\Delta,\\
 & 2U^{\mu_3}(z)+U^{\mu_1}(z)-U^{\mu_2}(z)-l_3=\re\int_{\min\{a_1,a_*\}}^z(\xi_3(s)-\xi_2(s))ds,\quad z\in \C\setminus\Delta. 
\end{aligned}
\end{equation}

Combining \eqref{aux_equation_7} with \eqref{condition_trajectory_2}, we conclude that
\begin{equation}\label{variational_equality_1}
2U^{\mu_1}(z)+U^{\mu_2}(z)+U^{\mu_3}(z)+\phi(z)-l_1=0,\quad z\in \Delta_1.
\end{equation}
Moreover, from the first equation in \eqref{conditions_positive_solutions} we also get 
\begin{equation}\label{variational_inequality_1}
2U^{\mu_1}(z)+U^{\mu_2}(z)+U^{\mu_3}(z)+\phi(z)-l_1>0,\quad z\in (b_1,+\infty).
\end{equation}
Furthermore, in the precritical case $\tau<\tau_c$ and $z\in [a_*,a_1)$,
\begin{align}
2U^{\mu_1}(z)+U^{\mu_2}(z)+U^{\mu_3}(z)+\phi(z)-l_1 & = \re \int_{b_1}^{a_1}(\xi_{1+}(s)-\xi_{2+}(s))ds \nonumber \\ 
						    & \qquad + \re \int_{a_1}^z(\xi_1(s)-\xi_2(s))ds \nonumber \\
						    & = -\int^{a_1}_z(\xi_1(s)-\xi_2(s))ds >0, \label{variational_inequality_precritical_2}
\end{align}
where for the second equality we used \eqref{condition_trajectory_2} and for the final inequality we used \eqref{condition_positive_solution_2}.

Analogously,
\begin{align}
 & 2U^{\mu_2}(z)+U^{\mu_1}(z)-U^{\mu_3}(z)+\phi(z)-l_2=0,\quad z\in \Delta_2, \label{variational_equality_2} \\
 & 2U^{\mu_3}(z)+U^{\mu_1}(z)-U^{\mu_2}(z)-\tilde l_3=0,\quad z\in \Delta_3, \label{variational_equality_3} \\
 & 2U^{\mu_3}(z)+U^{\mu_1}(z)-U^{\mu_2}(z)-\tilde l_3>0,\quad z <\min\{a_*,a_1\}, \label{variational_inequality_3} \\
 & 2U^{\mu_3}(a_*)+U^{\mu_1}(a_*)-U^{\mu_2}(a_*)=\tilde l_3 \label{variational_inequality_precritical}.
\end{align}

If $\tau\geq \tau_c$, we evaluate \eqref{variational_equality_1},\eqref{variational_equality_2}, \eqref{variational_equality_3} at the common point 
$a_*\in\Delta_1\cap\Delta_2\cap\Delta_3$ and take differences in order to get
$$
l_1-l_2-\tilde l_3=0.
$$

On the other hand, if $\tau< \tau_c$, we use \eqref{variational_inequality_precritical_2} to get
$$
2U^{\mu_1}(a_*)+U^{\mu_2}(a_*)+U^{\mu_3}(a_*)+\phi(a_*)>l_1.
$$
We now combine this inequality with \eqref{variational_equality_2} evaluated at $z=a_*\in \Delta_2$ and \eqref{variational_inequality_precritical}, and conclude
$$
\tilde l_3>l_1-l_2.
$$

Hence, for both the precritical and supercritical cases we can define
$$
l_3:=l_1-l_2,
$$
and with this definition, it follows that
\begin{align*}
 2U^{\mu_3}(z)+U^{\mu_1}(z)-U^{\mu_2}(z)- l_3>0, & \quad z <\min\{a_*,a_1\},  \\
 2U^{\mu_3}(z)+U^{\mu_1}(z)-U^{\mu_2}(z)- l_3=0, & \quad z \in \Delta_3. 
\end{align*}

One more variational inequality is based on the fact that we deal with a critical trajectory of our quadratic differential. Recall that by \eqref{aux_equation_8},
$$
2U^{\mu_2}(z)+U^{\mu_1}(z)-U^{\mu_3}(z)+\phi(z)-l_2=\re\psi_2(z),\quad z\in \C\setminus\Delta,
$$
where
$$
\psi_2(z)=\int^z_{a_2}(\xi_1(s)-\xi_3(s))ds.
$$
From \eqref{extension_support_mu_2_property_1}, we know that $\psi_2$ is real-valued on $\Gamma\setminus\Delta_2$. Hence,
$$
2U^{\mu_2}(z)+U^{\mu_1}(z)-U^{\mu_3}(z)+\phi(z)-l_2=\psi_2(z),\quad z\in \Gamma\setminus\Delta_2.
$$
Moreover, by \eqref{extension_support_mu_2_property_2}, the real-valued function $\psi_2$ is monotone on each connected component of $\Gamma\setminus\Delta_2$. Analyzing at 
$\infty$, we 
see that
$$
\psi_2(z)=z^3+\Boh(z),
$$
hence $\psi_2$ tends to $+\infty$ along $\Gamma$. Since it is zero at the endpoints $a_2$, $b_2$ of the connected components of $\Gamma\setminus\Delta_2$, we conclude that $\Psi_2$ 
is 
always positive on $\Gamma\setminus\Delta_2$, and hence
$$
2U^{\mu_2}(z)+U^{\mu_1}(z)-U^{\mu_3}(z)+\phi(z)-l_2>0,\quad z\in \Gamma\setminus \Delta_2.
$$

We summarize our findings in the following theorem:
\begin{thm}\label{thm_equilibrium_conditions_cubic}
Let measures $\mu_1$, $\mu_2$, $\mu_3$ be defined in \eqref{definition_measures_mu}. Then there exist real constants $l_1$, $l_2$ and 
$$
l_3:=l_1-l_2,
$$
such that the following variational conditions are satisfied:
 \begin{equation}\label{eq_equilibrium_conditions_cubic}
 \begin{aligned}
 & 2U^{\mu_1}(z)+U^{\mu_2}(z)+U^{\mu_3}(z)+\phi(z)-l_1=0,\quad z\in \Delta_1,\\
 & 2U^{\mu_1}(z)+U^{\mu_2}(z)+U^{\mu_3}(z)+\phi(z)-l_1>0,\quad z\in \Gamma_1\setminus\Delta_1, \\
 & 2U^{\mu_2}(z)+U^{\mu_1}(z)-U^{\mu_3}(z)+\phi(z)-l_2=0,\quad z\in \Delta_2, \\
 & 2U^{\mu_2}(z)+U^{\mu_1}(z)-U^{\mu_3}(z)+\phi(z)-l_2>0,\quad z\in \Gamma_2\setminus \Delta_2 \\
 & 2U^{\mu_3}(z)+U^{\mu_1}(z)-U^{\mu_2}(z)-l_3=0,\quad z\in \Delta_3, \\
 & 2U^{\mu_3}(z)+U^{\mu_1}(z)-U^{\mu_2}(z)-l_3>0,\quad z\in \Gamma_3\setminus \Delta_3.
\end{aligned}
\end{equation}
where $\Gamma_1=[a_*,+\infty)$, $\Gamma_2=\Gamma$, $\Gamma_3=(-\infty,a_*]$.
\end{thm}

\begin{remark} \label{rem_equilibrium_conditions_cubic}
In the pre-critical case ($\tau<\tau_c$), $\Delta_3=\emptyset$, thus the equality on $\Delta_3$ in \eqref{eq_equilibrium_conditions_cubic}  is void.
	
 The equalities in \eqref{eq_equilibrium_conditions_cubic} are the same as those in \eqref{variational_equations_critical_measures}. The extra information in 
Theorem~\ref{thm_equilibrium_conditions_cubic} is coming from the remaining equations, which assure that the triplet $(\mu_1,\mu_2,\mu_3)$ is the (unique) minimizer of the energy 
functional \eqref{energy_functional} over measures $(\nu_1,\nu_2,\nu_3)$ satisfying \eqref{massconstraint} (with $\mu_j$ replaced by $\nu_j$) and $\supp\mu_j\subset \Gamma_j$, 
$j=1,2,3$, see \cite[Theorem~1.8]{beckermann_et_al_equilibrium_problems}. This is equivalent to Theorem~\ref{thm_equilibrium_measure_cubic} with $\Gamma=\Gamma_2$.
\end{remark}

\section{Global structure of the trajectories in the cubic case} \label{section:dynamics}

\subsection{Dynamics of the singularities}

A natural first step in the study of the structure of the trajectories of a quadratic differential is to clarify the position and the character of its singular points. In the case 
of the quadratic differential \eqref{qd} we have to analyze the location and the dynamics of the branch points and the double point of the Riemann surface $\mathcal{R}$ 
(corresponding to the spectral curve \eqref{spectral_curve}) as functions of the parameter $\tau$.

Recall that in Section \ref{section_spectral_curve_to_variational_equations} we have denoted the branch points of \eqref{spectral_curve}  by $a_1,a_2,b_1$ and $b_2$, with the 
conventions
$$
a_1,b_1\in \R,\quad a_1<b_1, \quad a_2=\overline b_2,\quad \im a_2<0, 
$$
and the double point of \eqref{spectral_curve} by $b_*$.

\begin{prop}\label{proposition_description_branchpoints}
The main parameters of the Riemann surface $\mathcal R$ corresponding to the algebraic equation \eqref{spectral_curve} exhibit the following behavior: as $\tau$ grows from $0$ to 
$1/4$,
 \begin{enumerate}[\rm (i)]
  \item the coefficient $c$ in \eqref{definition_c}  increases monotonically from $-3^{5/3}/4$ to $0$;
  \item \begin{itemize}
  \item the branch point $a_1$  decreases monotonically from $3^{2/3}/4$ to $-\infty$;
  \item the branch point $b_1$  increases monotonically from $3^{2/3}/4$ to $3^{2/3}/2$;
  \item the double point $b_*$ increases monotonically from $3^{-1/3}$ to $+\infty$.
  \end{itemize}
  \item Always
  $$
  b_1\leq b_*,
  $$
  and the equality is attained only for $\tau=1/12$.
 \end{enumerate}
\end{prop}
\begin{proof}

Recall that we chose $c$ in \eqref{definition_c} to be real, thus (i) follows directly from the definition of $c$.
 
Any branch point $z=\zeta$ of $\mathcal R$ is a zero of the polynomial $q_1$ of degree $4$ and real coefficients, defined in \eqref{definition_q_1}, so that 
\begin{equation}
\label{derivativea1}
 \partial_\tau \zeta=-\frac{\partial_\tau q_1(\zeta)}{\partial_z  q_1(\zeta)}.
\end{equation}
It is easy to check that the resultant of $ q_1$ and $\partial_\tau q_1$ (with respect to the variable $z$) is 
 $$
 \const \times (1-4 \tau )^{-\frac{5}{3}} (27-100 \tau)^3 \neq 0 \quad \text{for } \tau \in (0,1/4),
 $$ 
which implies that for any root $\zeta$ of $q_1$, $\partial_\tau  q_1(\zeta)\neq 0$, and thus preserves its sign in the whole range of values of $\tau \in (0,1/4)$. 

On the other hand, since  $a_1$ (resp., $b_1$) is the smallest (resp., largest) real root of $q_1$, and the leading coefficient of $q_1$ is positive, we know that 
 $$
 \partial_z  q_1( a_1)<0, \quad  \partial_z  q_1( b_1)>0.
 $$

Straightforward calculations show that for $\tau=1/8$, $a_1=0<b_1$ and
 $$
 \partial_\tau  q_1(z)= 
 -256\, u^4  -64\left( (u-1)^2+1 \right), \quad u=\left(\frac{2}{3}\right)^{2/3} z.
 $$
In particular, $\partial_\tau  q_1(z)\leq -64$, and we conclude that in this case,
 $$
\partial_\tau  q_1(a_1)=-128,   \quad \partial_\tau  q_1(b_1)<0.
 $$
Hence, $\partial_\tau  q_1(a_1), \partial_\tau  q_1(b_1)<0$ for all $\tau \in (0,1/4)$, and by \eqref{derivativea1}, $a_1$ is a decreasing and $b_1$ is an increasing function of 
$\tau$.

From the expression \eqref{definition_q_2} it follows that
 $$
 b_*=\frac{1}{3^{1/3}(1-4\tau )^{1/3}}.
 $$
Replacing it in \eqref{definition_q_1} we get
 $$
q_1(b_*)= \frac{32 (1-12 t)^2}{9 (1-4 t)}>0 \text{ for } \tau\in (0,1/4), \quad \tau \neq 1/12.
 $$
Since for $\tau=0$,
\begin{equation}
\label{abfortau0}
a_1=b_1 = \frac{3^{2/3}}{4}< b_*=3^{-1/3},
\end{equation}
this concludes the proof of (iii). Finally, it remains to observe that for $\tau=1/4$,
$$
q_1(z)=\frac{128}{9}\left(z^3-\frac{9}{8} \right),
$$
which has one real positive root ($3^{2/3}/2$) and two complex conjugate ones. This shows that 
$$
\lim_{\tau\to 1/4-} a_1=- \infty, \quad \lim_{\tau\to 1/4-} b_1=\frac{3^{2/3}}{2}. 
$$
\end{proof}

We finish this section with a technical lemma that will be used later.

\begin{prop}\label{prop_zeros_D}
For $0<\tau<1/12$, the polynomial $D$ in \eqref{D_for_cubic} does not have zeros on $(a_1,b_1)$, whereas for $1/12<\tau <1/4$, $D$ has exactly one zero on $(a_1,b_1)$. Moreover, 
for $0<\tau<1/12$, $D(b_1)=0$ only for $\tau=1/12$, and $D(a_1)$ is never zero. 
\end{prop}
\begin{proof}
On one hand, the discriminant of the polynomial $D$ with respect to $z$ is
$$
f(\tau)=\const \tau (1-12 \tau)^2 \left(128 \tau^2-32 \tau-1\right)
$$
which shows that for $\tau\in (0,1/4)$, polynomial $D$ has no multiple roots as long as $\tau\neq 1/12$. 

On the other hand, we have seen in Section~\ref{section_cubic_spectral_curve} that the branch points of $\mathcal R$ are the zeros of the polynomial $q_1$ defined in 
\eqref{definition_q_1}. The resultant (also w.r.t. $z$) of $D$ and $q_1$ is
$$
g(\tau)=\const (1-12\tau)^3 (27-100\tau)^3,
$$
so again, for $\tau\in (0,1/4)$, polynomials $D$ and $q_1$ have no common roots as long as $\tau\neq 1/12$, in particular implying that $D$ does not vanish at $a_1,b_1$ for 
$\tau\neq 1/12$. For $\tau=1/12$, we compute $b_1=b_*=2^{-1/3}$ and 
factor
$$
D(z)=-(z-b_*)^2(2 z^4+ 2^{5/3} z^3+3 \ 2^{1/3} z^2+z+2^{-4/3}),
$$
so for $\tau=1/12$ we have $D(b_1)=0$ and $D(a_1)<0$.

It is worth pointing out that both $f$ and $g$ can be easily found by means of any computer algebra software. 

Having in mind that $D$ and $q_1$ do not share roots for $\tau\neq 1/12$, it is sufficient to establish the assertion concerning the zeros of $D$ for a 
single value of $\tau$ in $(0, 1/12)$, and for a single value of $\tau$ in $( 1/12, 1/4)$.
	
For $\tau=1/4$, $D$ is a quadratic polynomial in $z^3$ so its roots can be explicitly computed; we get that its only real roots are 
$$
z_1=\left(\frac{3-\sqrt{3}}{4}\right)^{1/3},\quad z_2=\left(\frac{3+\sqrt{3}}{4}\right)^{1/3}.
$$
Since
$$
\left(\frac{3-\sqrt{3}}{4}\right)^{1/3}< \frac{3^{2/3}}{2} <\left(\frac{3+\sqrt{3}}{4}\right)^{1/3},
$$
from Proposition~\ref{proposition_description_branchpoints} {\rm (ii)} we see that for $\tau=1/4-\epsilon$, for a certain small value of $\epsilon>0$,   
$$
a_1<z_1 <b_1<z_2.
$$

On the other hand, it is easy to check that for $\tau=0$,
$$
D(a_1)=D(b_1)=D(3^{2/3}/4)<0,
$$
so that for $\tau= \epsilon$, for a certain small value of $\epsilon>0$, $D$ does not vanish on $(a_1, b_1)$.

\end{proof}

\subsection{The Riemann surface associated to the algebraic equation} \label{sec:RSassociatedalgcurve}

In Section~\ref{section_spectral_curve_to_variational_equations} we described the construction of the three-sheeted Riemann surface $\mathcal R=\mathcal R_1 \cup \mathcal R_2 \cup 
\mathcal R_3$ and of the branch cuts in such a way that the three solutions $\xi_j$ of \eqref{spectral_curve}, specified by the asymptotic conditions 
\eqref{asymptotics_xi_functions}, become meromorphic on the respective sheet $\mathcal R_j$, with poles only at $z=\infty$. They are also pairwise distinct as long as $\discr(z)$, 
defined in \eqref{defDiscriminant}, does not vanish, i.e.~for $z\notin \{a_1, b_1, a_2, b_2, b_*\}$. 

The arc $\Delta_2$ intersects $\R$ in a unique point $a_*$; there is a critical value $\tau_c\in (1/12,1/4)$, to be specified later, such that $a_*<a_1$ for $0<\tau<\tau_c$ (what 
we called the precritical regime), and
$a_*\in(a_1,b_1)$ for $\tau_c<\tau<1/4$ (the supercritical regime).

As a first result we establish some relations between the solutions $\xi_j(z)$. 
\begin{prop}\label{proposition_distribution_branchpoints} 
Let $\tau \in (0,1/4)$, $\tau \neq \tau_c$. Then
\begin{enumerate}
\item[(i)] for $x\in \R\setminus (\Delta_1\cup \Delta_2\cup \Delta_3)$,
\begin{align}
& \xi_3(x)<\xi_2(x)<\xi_1(x), && x<\min(a_*,a_1),  \label{inequality_xi_1}\\
& \xi_2(x)<\xi_3(x)<\xi_1(x), && x>b_*,  \label{inequality_xi_2} \\
& \xi_1(x)<\xi_2(x)<\xi_3(x), && a_* < x< a_1, \quad \text{if } \tau <\tau_c,  \label{inequality_xi_3} \\
& \xi_2(x)  <\xi_1(x)<\xi_3(x), && b_1<x<b_*, \quad \text{if} \quad  0<\tau<1/12, \label{inequality_xi_4} \\
& \xi_3(x)  <\xi_2(x)<\xi_1(x), &&  b_1<x<b_*, \quad \text{if} \quad  1/12<\tau<1/4. \label{inequality_xi_5}
\end{align}
Additionally,
\begin{equation}\label{coinciding_nodes}
\begin{split}
\xi_2(b_*)<\xi_3(b_*)& =\xi_1(b_*), \quad \text{for} \quad  0<\tau<1/12,\\
\xi_2(b_*)=\xi_3(b_*) & <\xi_1(b_*), \quad \text{for} \quad  1/12<\tau<1/4.
\end{split}
\end{equation}
\item[(ii)] on $  \Delta_1\cup \Delta_2\cup \Delta_3$,
\begin{itemize}
\item  for  $x \in \overset{\circ}{\Delta}_1:=\Delta_1\setminus \{\max(a_1,a_*), b_1 \}$, 
\begin{equation}
  \xi_2(x)=\overline{\xi_1(x)}\in \C\setminus\R, \quad \xi_3(x)\in \R,   \label{equality_xi_1}
\end{equation}
and
\begin{equation}
  \xi_{1\pm}(x)  = \xi_{2\mp}(x), \quad  \xi_{3+}(x)  = \xi_{3-}(x).  \label{equality_xi_2}
\end{equation}
\item for  $z \in \overset{\circ}{\Delta}_2:=\Delta_2\setminus \{a_2, b_2 ,\max(a_1,a_*)\}$,
\begin{equation}
 \xi_{1\pm}(z)= \xi_{3\mp}(z),\quad  \xi_{2+}(z) = \xi_{2-}(z).  \label{equality_xi_3}
\end{equation}
\item for  $x \in \overset{\circ}{\Delta}_3:=\Delta_3\setminus \{a_1, a_* \}$ (when $\tau >\tau_c$),
\begin{equation}
  \xi_2(x)=\overline{\xi_3(x)}\in \C\setminus\R, \quad \xi_1(x)\in \R,   \label{equality_xi_4}
\end{equation}
and
\begin{equation}
 \xi_{2\pm}(z)= \xi_{3\mp}(z),\quad  \xi_{1+}(z) = \xi_{1-}(z).  \label{equality_xi_5}
\end{equation}

Moreover,
\begin{align}
\xi_1(a_1)  &= \xi_2( a_1 ),\quad \text{if } \tau < \tau_c, \label{xi_at_branchpoints1} \\
\xi_3(a_1)  &= \xi_2( a_1 ),\quad \text{if } \tau > \tau_c, \label{xi_at_branchpoints2} \\
\xi_1(b_1)  &= \xi_2( b_1 ),  \label{xi_at_branchpoints3} \\
\xi_1(a_2) & = \xi_3(a_2) \quad \text{and} \quad \xi_1(b_2)  =\xi_3(b_2). \label{xi_at_branchpoints4}
\end{align}

\end{itemize}

\end{enumerate}

\end{prop}

\begin{proof}
Recall that we deal with the case $\alpha\in (0,1/2)$, so that $0<\alpha<1-\alpha<1$. The behavior at infinity in \eqref{asymptotics_xi_functions} gives \eqref{inequality_xi_1} and 
also \eqref{inequality_xi_2}. Furthermore, we have established in Proposition \ref{proposition_description_branchpoints} that if $\tau \neq 1/12$ then $b_*\notin \{a_1, b_1, a_2, 
b_2\}$, and in this case $b_*$ is a double zero of
$\discr(z)$. 
This means that when $\tau \neq 1/12$, only two of the three values $\xi_1(b_*)$, $\xi_2(b_*)$, $\xi_3(b_*)$ coincide, and since
$$
\xi_1(b_*)+ \xi_2(b_*)+ \xi_3(b_*)=0,
$$
the coincident two differ in sign from the third one. Furthermore, since also
$$
(\xi_1   \xi_2   \xi_3)(b_*)=-D(b_*),
$$
we see that the sign of this third one is opposite to the sign of $D(b_*)$. But direct calculations show that
$$
D(b_*)=-\frac{(12 \tau -1)^3}{36 (1-4 \tau )^2}
\begin{cases}
>0, & \text{for } 0<\tau<1/12,\\
<0, & \text{for } 1/12<\tau<1/4,
\end{cases}
$$
which means that for $0<\tau<1/12$ the two coincident values of $\xi_j(b_*)$ are the largest two, and the other way around if $1/12<\tau<1/4$.
It remains to use  \eqref{inequality_xi_2} in order to establish \eqref{coinciding_nodes}.

Inequalities \eqref{inequality_xi_4}--\eqref{inequality_xi_5} follow by noticing that $b_*$ is a double zero of $\discr(z)$, so it is a simple zero of $\xi_1-\xi_3$ (for $ 
0<\tau<1/12$) and of $\xi_2-\xi_3$ (for $ 1/12<\tau<1/4$).

From Proposition~\ref{prop_zeros_D} we know that $D(a_1)\neq 0$, $D(b_1)\neq 0$ for $\tau \neq 1/12$. In particular, since for $\tau=0$ (see \eqref{abfortau0}), 
$a_1=b_1=3^{2/3}/4$ and $D(a_1)=D(b_1)=-81/2048<0$, we conclude that $D(a_1)<0$ and $D(b_1)<0$ for $0\leq\tau <1/12$. In other words, the smallest two of the three solutions 
$\xi_j$ come together at $a_1$ and at $b_1$. Now \eqref{inequality_xi_4} implies \eqref{xi_at_branchpoints3} for $0\leq\tau <1/12$. 

On the other hand, recall that by Proposition~\ref{proposition_description_branchpoints}, $D(a_1)\to -\infty$ as $\tau\to 1/4$, and that for $\tau=1/4$, 
$$
D(z)=-2 z^6+3 z^3-\frac{3}{4},\quad b_1=3^{2/3}/2,
$$
so that $D(b_1)=3/32>0$. Using the same arguments we conclude that for $1/12<\tau <1/4$, the largest two of the three solutions $\xi_j$ come together at $b_1$ (and  
\eqref{inequality_xi_5} yields \eqref{xi_at_branchpoints3}), and the smallest two become confluent at $a_1$ (and \eqref{xi_at_branchpoints2} follows from \eqref{inequality_xi_1}).

 The discriminant $\discr(x)$ is negative for $a_1<x<b_1$, so just one of the solutions $\xi_j$ is real on $a_1<x<b_1$, and the other two are complex-conjugates of each other. 
Since we have ruled out the coincidence of the three branches, equality \eqref{xi_at_branchpoints3} implies that $\xi_1,
\xi_2$ are non-real on $\overset{\circ}{\Delta}_1$, which yields \eqref{equality_xi_1}--\eqref{equality_xi_2}. By the same argument we also have \eqref{xi_at_branchpoints1}, as 
well as (now using \eqref{inequality_xi_1}) the identities \eqref{equality_xi_4}--\eqref{equality_xi_5}.

Recall that we already established that the smallest two of the three solutions $\xi_j$ come together at $a_1$, and \eqref{xi_at_branchpoints1} shows that for $\tau<\tau_c$,
\begin{equation} \label{inequ_intermediate}
\xi_1(x)<\xi_3(x), \quad \xi_2(x)<\xi_3(x),\quad \text{for } a_*<x<a_1.
\end{equation}
For $\tau<\tau_c$, when crossing $\Delta_2$ two of the three solutions $\xi_j$ are swapped, and the third one remains invariant. But the only option compatible both with 
\eqref{inequality_xi_1} and \eqref{inequ_intermediate} is \eqref{inequality_xi_3}.

Finally, for $\tau<\tau_c$, the inequalities \eqref{inequality_xi_1} and \eqref{inequality_xi_3} show that $\xi_2$ is continuous across $\Delta_2$, and as a consequence  
\eqref{equality_xi_3} has to hold, obviously for the full range of $\tau\in (0,1/4)$. This, in turn, implies the last identity \eqref{xi_at_branchpoints4}.
 \end{proof}

Proposition \ref{proposition_description_branchpoints} gives the formal proof of the fact that the Riemann surface $\mathcal R$, described in Section
\ref{section_spectral_curve_to_variational_equations}, is actually the Riemann surface of the cubic equation \eqref{spectral_curve}, and with this construction the function 
$\xi_j$ is meromorphic on $\mathcal R_j$ and satisfies the asymptotic expansion \eqref{asymptotics_xi_functions}. Moreover, $\xi:\mathcal R\to \C$, given by $\xi\equiv \xi_j$ 
on $\mathcal R_j$, is meromorphic on 
$\mathcal R$,  giving the global solution to the algebraic equation
\eqref{spectral_curve}. From Theorem~\ref{thm_genus_zero}, or alternatively Proposition~\ref{prop:algebraicCubic}, the Riemann surface $\mathcal R$ has genus $0$.

Let $\pi:\mathcal R\to \overline \C$ be the canonical projection on the Riemann surface $\mathcal R$. For a point $p\in \overline \C\setminus (\Delta_1\cup \Delta_2\cup 
\Delta_3)$, we denote by $p^{(j)}$, $j=1,2,3$, its preimage by $\pi$ on $\mathcal R_j$, that is,
$$
\{p^{(j)}\}=\pi^{-1}(\{p \})\cap \mathcal R_j,\quad j=1,2,3.
$$
This notation is trivially extended to $p\in \Delta_1\cup\Delta_2\cup \Delta_3$ by taking boundary values.

Notice that for the branch points it is valid
$$
b_1^{(1)}=b_1^{(2)},\quad a_2^{(2)}=a_2^{(3)}, \quad b_2^{(2)}=b_2^{(3)},
$$
and
$$
a_1^{(1)}=a_1^{(2)} \ (\tau\leq \tau_c),\quad a_1^{(2)}=a_1^{(3)} \ (\tau>\tau_c),
$$
whereas if $p$ belongs to a cut connecting exactly two sheets, say $\mathcal R_j$ and $\mathcal R_k$, then
$$
p^{(j)}_\pm=p^{(k)}_\mp.
$$

We insist that the construction of $\mathcal R$ is independent of the concrete choice of the cut $\Delta_2$; this freedom will be used latter to specify an appropriate $\Delta_2$. 
Namely, in the 
next section we will show that $\Delta_2$ can be made coincident with a critical trajectory of the quadratic differential $\varpi$, defined in \eqref{qd},which connects $a_2$ and 
$b_2$, see Definition~\ref{def:thebranchcut} below.

\subsection{Computation of width parameters}\label{section_widths_cubic}

Certain integrals of the function $Q$ defined in \eqref{defQqd_cubic} will play a crucial role in the upcoming analysis of the dynamics of the trajectories of the quadratic 
differential $\varpi=
-Q^2(z)\, dz^2$ on $\mathcal{R}$. They can be also formulated in terms of certain Abelian integrals on $\mathcal{R}$.

Namely, we are interested in 
\begin{equation}\label{width_parameter_cubic}
\begin{split}
\omega_1 = \omega_1(\tau) &=\re \int_{b_2^{(1)}}^{a_1^{(1)}}Q(s) ds =\re \int_{b_2}^{a_1}(\xi_2(s)-\xi_3(s))ds,\\
\omega_2 = \omega_2(\tau) &=\re \int_{b_2^{(2)}}^{a_1^{(2)}}Q(s) ds =\re \int_{b_2}^{a_1}(\xi_1(s)-\xi_3(s))ds, \\
\omega_3 = \omega_3(\tau) & =\re \int_{b_2^{(3)}}^{a_1^{(3)}}Q(s) ds =\re \int_{b_2}^{a_1}(\xi_1(s)-\xi_2(s))ds, \\
\omega_4 = \omega_4(\tau) &=\re \int_{b_2^{(1)}}^{b_*^{(1)}}Q(s) ds =\re \int_{b_2}^{b_*}(\xi_{2}(s)-\xi_3(s))ds, 
\end{split}
\end{equation}
with $\omega_4(\tau)$  defined for $\tau\geq 1/12$. In this definition we understand that we integrate between two points on a sheet $\mathcal{R}_j$ along a path that stays 
entirely in $\mathcal{R}_j$.

The values $\omega_j$ are correctly defined regardless of the precise choice of the integration paths. Indeed, the residues of the functions 
$\xi_1,\xi_2,\xi_3$ at $\infty$ are real (and independent of the value of $\tau$), see \eqref{asymptotics_xi_functions}, so the integral of $Q$ along a big loop on either 
$\mathcal{R}_j$ encircling $\infty^{(j)}$ is purely imaginary. It remains to notice that the genus of $\mathcal{R}$ is zero, so that any closed contour  around either branch cut 
$\Delta_j$ can be deformed to such a  big loop.

Analytic computation of $\omega_j$'s is a formidable task. Instead, we have computed them numerically, see
Figure~\ref{figure_width_parameters} for the result. Since the integrands in \eqref{width_parameter_cubic} are multivalued functions, the numerical integration requires to 
implement an analytic continuation of each branch of $\xi$. We give further details in the Appendix~\ref{appendix_numerics}.

\begin{figure}[htb]
	\centering \begin{overpic}[scale=1]{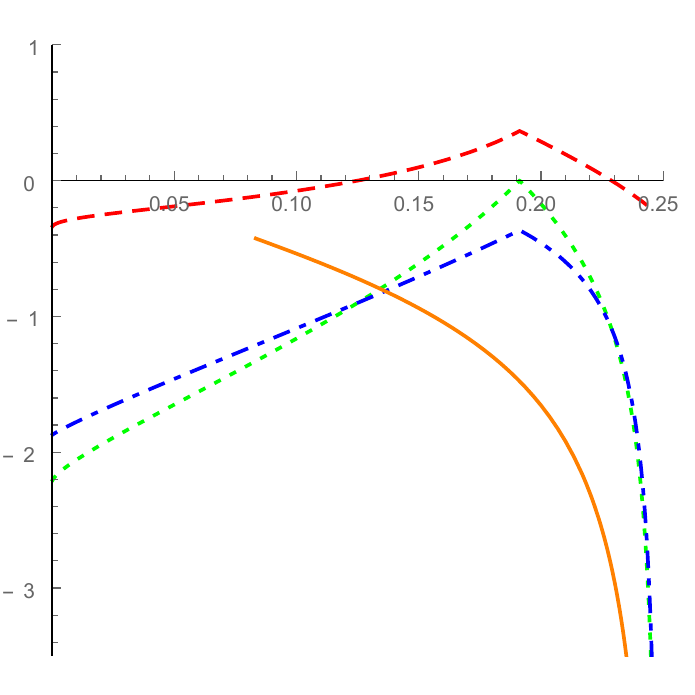}%
		\put(20,130){\small $\omega_1 $}
		\put(20,60){\small $\omega_2 $}
			\put(20,85){\small $\omega_3 $}
				\put(157,60){\small $\omega_4 $}
				\put(200,150){\small $\tau$}
				\put(100,155){\small $\downarrow $}
				\put(99,165){\small $\tau_1 $}
				\put(148,167){\small $\downarrow $}
				\put(146,177){\small $\tau_c $}
				\put(175,155){\small $\downarrow $}
				\put(173,165){\small $\tau_2 $}
				\put(72,155){\small $\downarrow $}
				\put(65,165){\small $1/12 $}
	\end{overpic}
	\caption{The graphics of the functions $\omega_1$ (upper dashed line), $\omega_2$ (short dashed line), $\omega_3$ (dashed line with dots) and $\omega_4$ (continuous line). 
}
	\label{figure_width_parameters}
\end{figure}
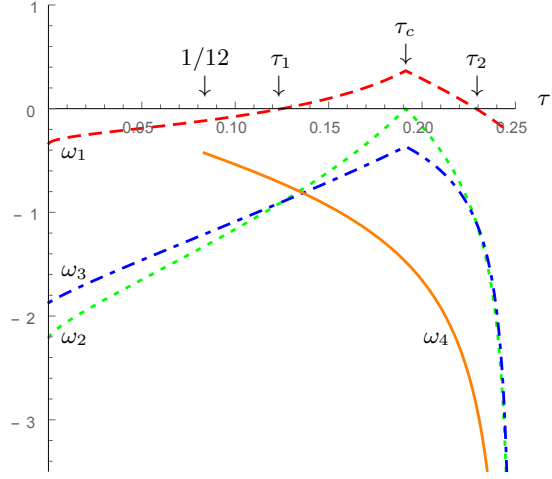

The functions $\omega_1$, $\omega_2$, have two and one zeros on $(0,1/4)$, respectively, whereas $\omega_3$ and $\omega_4$ do not vanish on the intervals $(0,1/4)$ and 
$(1/12,1/4)$, respectively. 

We denote the 
zeros of $\omega_1$ by $\tau_1,\tau_2$ and of $\omega_2$ 
by $\tau_c$. We have
$$
\tau_1\approx 0.12487351,\quad \tau_c\approx 0.1913565,\quad \tau_2\approx 0.2289555,
$$
so that they satisfy
\begin{equation}\label{valuesfortau}
0<\frac{1}{12}<\tau_1<\tau_c<\tau_2<1/4.
\end{equation}
We point out that the value $\tau_c$ is the one used in Section~\ref{section_spectral_curve_to_variational_equations} in the construction of the Riemann surface $\mathcal R$, and 
it is equivalently determined by \eqref{equation_definition_alphac}.

It is also clear from Figure~\ref{figure_width_parameters} that $\omega_1\neq \omega_4$ for $\tau>1/12$.

\subsection{Critical points of the quadratic differential}\label{section_critical_points}

We follow the construction \eqref{defQqd}, \eqref{qd} and define the meromorphic quadratic differential $\varpi=
-Q^2(z)\, dz^2$ on the Riemann surface $\mathcal R$.

We start by analyzing the character of the critical points of $\varpi$. For instance, the local parameter at $z=a_1^{(1)}$ is
$$
z=a_1^{(1)}+u^2,
$$
so that $dz^2 = 4 u^2 du^2$. In consequence, $(\xi_j-\xi_k)^2(z) dz^2$ has a double zero at $z=a_1^{(1)}$ if and only if $\xi_j(a_1^{(1)})\neq \xi_k(a_1^{(1)})$. A similar analysis 
at the rest of the points of $\mathcal{R}$ yields the following classification of the critical points of $\varpi$ on the Riemann surface (see Figure~\ref{figure_critical_points}):
\begin{enumerate}[(a)]

\item For $0<\tau<\tau_c$:
\begin{enumerate}[(i)]
 \item Simple zeros at $a_1^{(3)}$, $b_1^{(3)}$, $a_2^{(2)}$, $b_2^{(2)}$;
 \item Double zeros at $a_1^{(1)}$, $b_1^{(1)}$ (only for $\tau \neq 1/12$), $a_2^{(1)}$, $b_2^{(1)}$;
 \item Double zero at $b_*^{(2)}$ if $\tau <1/12$ and a double zero at $b_*^{(1)}$ if $\tau>1/12$;
 \item Zero of order $4$ at $b_1^{(1)}=b_*^{(1)}$ if $\tau =1/12$;
 \item Double pole at $\infty^{(1)}$ with a real residue;
 \item Poles of order $8$ at $\infty^{(2)}$, $\infty^{(3)}$.
 \end{enumerate}
 
 \item For $\tau_c\leq \tau<1/4$:
\begin{enumerate}[(i)]
 \item Simple zeros at $a_1^{(1)}$, $b_1^{(3)}$, $a_2^{(2)}$, $b_2^{(2)}$;
 \item Double zeros at $a_1^{(2)}$, $b_1^{(1)}$, $a_2^{(1)}$, $b_2^{(1)}$;
 \item Double zero at $b_*^{(1)}$;
 \item Double pole at $\infty^{(1)}$ with a real residue;
 \item Poles of order $8$ at $\infty^{(2)}$, $\infty^{(3)}$.
\end{enumerate}
\end{enumerate}

\begin{figure}
\begin{subfigure}{0.5\textwidth}
\centering
\begin{overpic}[scale=0.35]{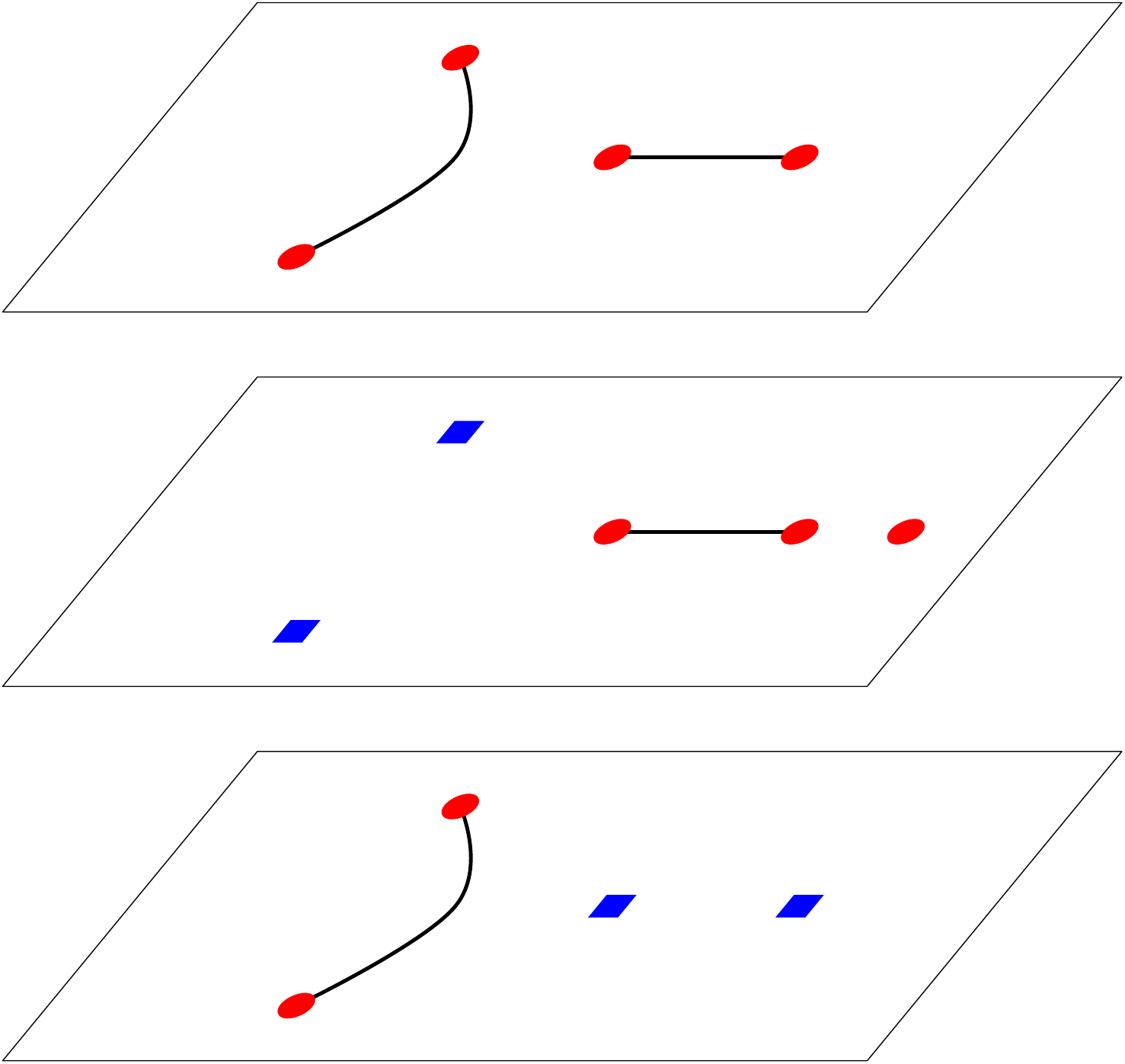}
\end{overpic}
\caption*{$0<\tau<1/12$}
\end{subfigure}%
\begin{subfigure}{0.5\textwidth}
\centering
\begin{overpic}[scale=0.35]{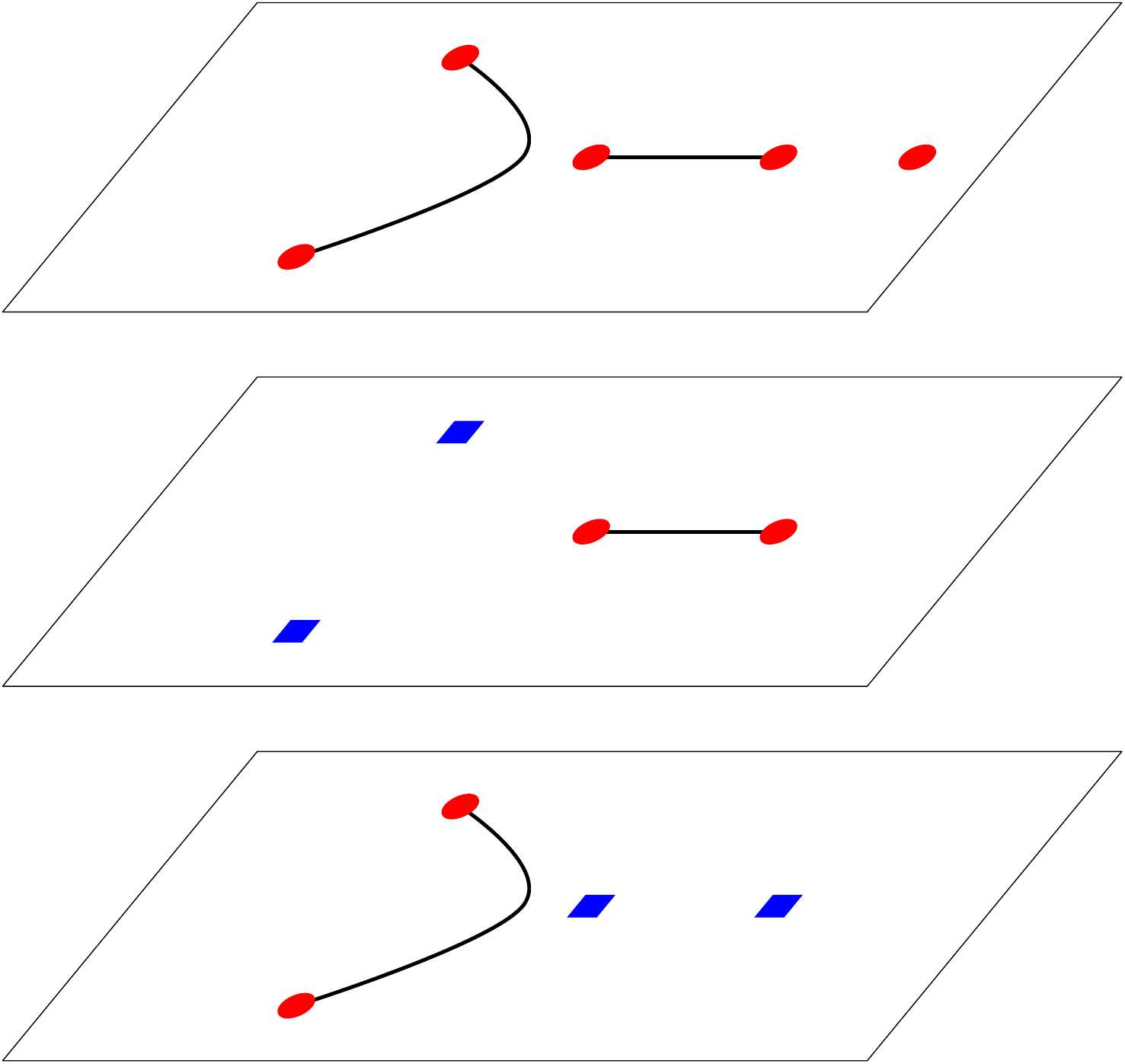}
\end{overpic}
\caption*{$1/12<\tau<\tau_c$}
\end{subfigure}
\begin{subfigure}{1\textwidth}
\vspace{0.3cm}
\centering
\begin{overpic}[scale=0.35]{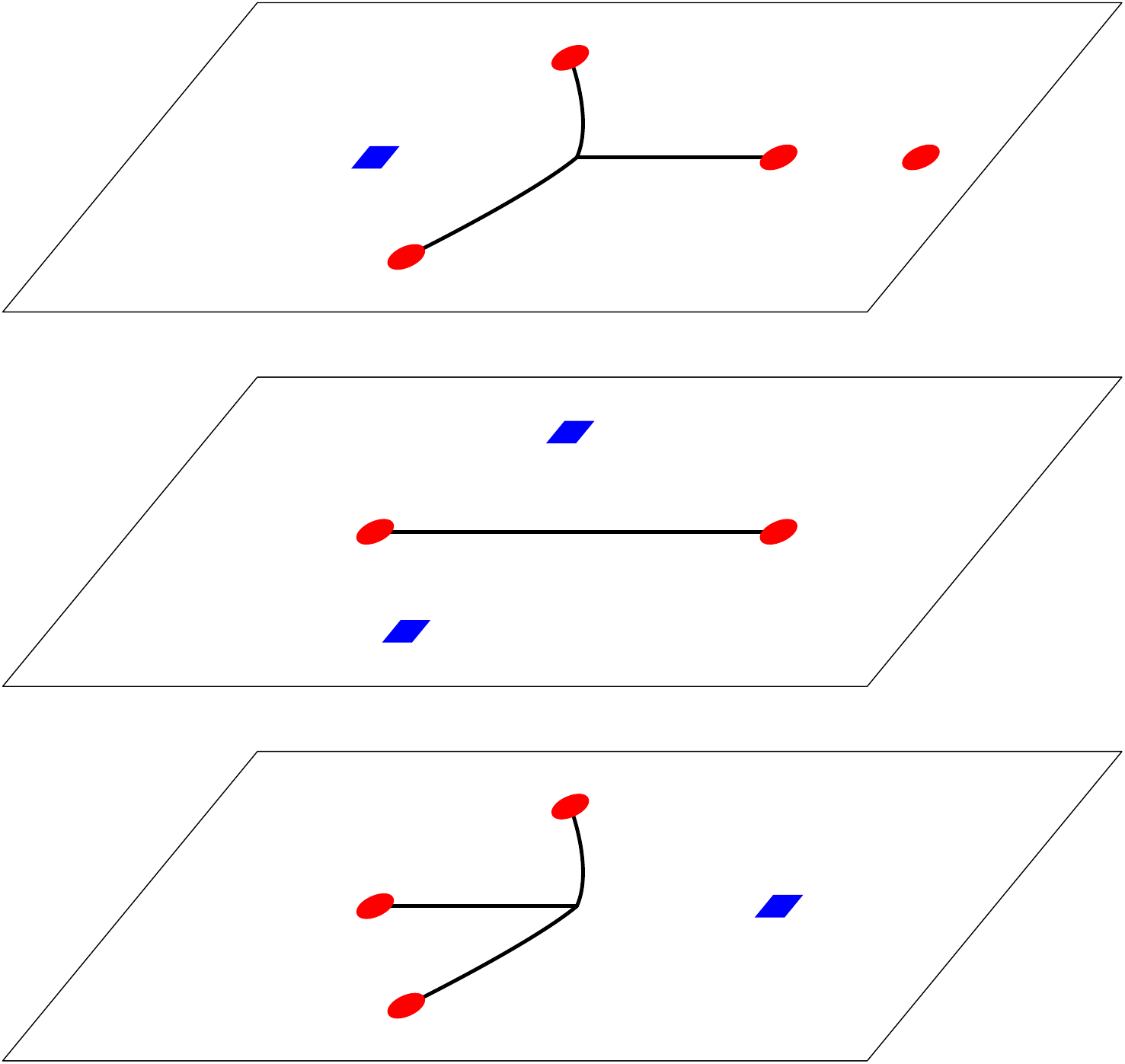}
\end{overpic}
\caption*{$\tau_c<\tau<1/4$}
\end{subfigure}
\caption{The finite critical points of $\varpi$ in each case. Simple and double zeros are represented, respectively, with squares and dots.}\label{figure_critical_points}
\end{figure}

It is instructive to think of the critical points as evolving dynamically with $\tau$. Under this perspective, Propositions~\ref{proposition_description_branchpoints} and 
\ref{proposition_distribution_branchpoints}  show that for $\tau$ small the double point (node) $b_*$ corresponds to the double zero 
$b_*^{(2)}$ on the second sheet. When $\tau=1/12$, the points $b_*$, $b_1$ coalesce, giving rise to a higher order zero at the branch point $b_1^{(1)}=b_1^{(2)}$. For larger values 
of $\tau$, the double point emerges on the first sheet: now $b_*^{(2)}$ is a regular point whereas $b_*^{(1)}$ is a double zero.

In the same spirit, the simple zero $a_1$ of the discriminant of \eqref{spectral_curve} carries two critical points of $\varpi$. For $\tau <\tau_c$ these are the simple zero 
$a_1^{(3)}$ and the double zero $a_1^{(1)}=a_1^{(2)}$. For values of $\tau$ larger than $\tau_c$, these points interchange their roles: $a_1^{(1)}$ is a simple zero and 
$a_1^{(3)}=a_1^{(2)}$ is a double zero.

\subsection{Analyzing the global structure of trajectories}

\subsubsection{General principles} \label{sec:generalprinciples}

The rest of this section is devoted to the description of the critical graph of the quadratic differential \eqref{qd} for the whole range $0\leq \tau<1/4$. One of the outcomes of 
our analysis is the following theorem:
\begin{thm}\label{thm:criticaltraj}
	For  the quadratic differential $\varpi=
	-Q^2(z)\, dz^2$ and for all values of the parameter $0\leq \tau\leq \tau_c$ there exists  a critical trajectory of $\varpi$ joining $a_2^{(2)}$ and $b_2^{(2)}$ on $\mathcal 
R_2$, whose projection by $\pi$ on $\C$ is a real-symmetric analytic arc $\Delta_2$ joining $a_2$ and $b_2$.
	
	For $\tau_c \leq \tau\leq 1/4$, there is an arc of critical trajectory of $\varpi$ joining $a_2^{(2)}$ with a point $a_*^{(2)}$ on the interval $[a_1^{(2)},b_1^{(2)}]$
which is determined by \eqref{definition_a_star_supercritical}, and the conjugate symmetric arc of trajectory joining $b_2^{(2)}$ with the same point. The projection by $\pi$ on 
$\C$ of the union of these two arcs of trajectories is also 
denoted by $\Delta_2$.
\end{thm}

Recall that $\tau_c$ was formally introduced in Section~\ref{section_widths_cubic}. In virtue of the results in Section~\ref{section_spectral_curve_to_variational_equations}, in 
particular Corollary~\ref{corol_existence_critical_measure}, Theorem~\ref{thm:criticaltraj} implies Theorem~\ref{thm_existence_critical_measures_cubic}.

We remind the reader that up to now the branch cut, separating the sheets $\mathcal R_1$ and $\mathcal R_3$, was free (see Figure~\ref{figure_sheet_structure}).  In what follows we 
agree in the following:
\begin{definition}\label{def:thebranchcut}
The curve, connecting $a_2$ and $b_2$ as part of the branch cut separating the sheets $\mathcal R_1$ and $\mathcal R_3$, is always  given by the lift of $\Delta_2$ (from 
Theorem~\ref{thm:criticaltraj}) to the sheets $\mathcal R_1$ and $\mathcal R_3$.
\end{definition}
In the next sections we will show that this definition is consistent with our construction of the Riemann surface $\mathcal R$.

One important fact is that the residues of $\varpi$ at the poles at infinity (and the local behavior of the trajectories there) are independent of $\tau$: at $\infty^{(1)}$ they 
are closed analytic curves (so that $\infty^{(1)}$ is the center of a circle domain, see Appendix \ref{appendix_quadratic_differentials}), while $\infty^{(2)}, \infty^{(3)}$ 
attract trajectories in 6  asymptotic directions, given by the angles 
\begin{equation}\label{asymptotic_directions_cubic}
\theta_j^{(\infty)}=\frac{2j-1}{6}\, \pi , \quad j=1,\hdots,6.
\end{equation}

Critical values \eqref{valuesfortau} split the interval $(0,1/4)$ into the subintervals $(0,1/12)$, $(1/12,\tau_1)$, $(\tau_1,\tau_c)$, $(\tau_c,\tau_2)$ and $(\tau_2,1/4)$. We will 
show that the topology of the critical graph remains invariant in each of these intervals.

The methodology we use can be summarized as follows:
\begin{enumerate}[\rm (i)]
\item Compute the critical graph for $\tau$ equal to one of the critical values \eqref{valuesfortau}.

\item Analyze the possible deformation of the trajectories for the values $\tau + \varepsilon$, with $\varepsilon>0$ small,  identifying the trajectories that display a phase 
transition.

\item Prove that the topology of the critical graph is invariant inside the subinterval of interest, by analyzing the behavior of the widths $\omega_j$'s and showing that the 
corresponding strip and ring domains can not disappear. 
\end{enumerate}

Along the way, will use some general \emph{principles} that we enumerate here:
\begin{enumerate}
	\item[\textbf{P.1}]  Quadratic differential $\varpi$ has no recurrent trajectory for any value of $\tau$, see Jenkin's Three Poles Theorem in Section 
\ref{sec:globalstructuretrajectoriesAppendix2}. 
	\item[\textbf{P.2}] If $\gamma$ is an arc of trajectory of $\varpi$, then $\overline \gamma$, corresponding to the lift of the complex conjugate of $\pi(\gamma)$ to the 
same sheet, is also an arc of trajectory.
	\item[\textbf{P.3}] The complement of the critical graph of $\varpi$ on $\mathcal R$ cannot have a simply connected component without poles on its boundary: that would 
contradict Corollary~\ref{corol_teichmuller_lemma} or the maximum principle for harmonic functions on a compact Riemann surface. 
	\item[\textbf{P.4}] The meromorphic function $Q^2$ depends analytically on the parameter $\tau$. Hence the critical graph of $\varpi$, and in particular all its critical 
trajectories, depend continuously (in any reasonable topology, for instance, in the Hausdorff distance) from $\tau$.
	\item[\textbf{P.5}] If for a certain value $\tau=A$, the point $p$ belongs to the half plane domain for $\infty^{(k)}$ determined by the angles 
$\theta^{(\infty)}_j,\theta^{(\infty)}_{j+1}$, then the same holds true for a small neighborhood of values $\tau\in (A-\varepsilon, A+\varepsilon)$, $\varepsilon>0$. The point $p$ 
is also allowed to depend continuously on $\tau$.
	\item[\textbf{P.6}] If for a certain value $\tau=A$, an arc of trajectory emanating from a given point $p$ intersects the real line at a {\it regular } point, then the 
same holds true for $\tau\in (A-\varepsilon, A+\varepsilon)$, $\varepsilon>0$. As before, the point $p$ is allowed to depend continuously on $\tau$.
\end{enumerate} 

There will be one more useful tool that we will employ several times in our analysis, formulated as Proposition \ref{lemma:zerosofD} below.

When describing the structure and the evolution of the trajectories of the quadratic differential $\varpi$ we face the dilemma of either a formalization of each statement, with a 
precise formulation of the behavior of every trajectory in every situation, or a much more visual description, with rigorous proofs but illustrated by a number of figures. We opted 
for the second choice\footnote{ We confess we might have been influenced by the famous quote of Vladimir Arnold \cite{arnoldinterview}: 
	\begin{quote}It is almost impossible for me to read contemporary mathematicians who, instead of saying ``Petya washed his hands,'' write simply: ``There is a $t_1<0$ such 
that the image of $t_1$ under the natural mapping $t_1 \mapsto {\rm Petya}(t_1)$ belongs to the set of dirty hands, and a $t_2$, $t_1<t_2 \leq 0$, such that the image of $t_2$ 
under the above-mentioned mapping belongs to the complement of the set defined in the preceding sentence.''\end{quote}}.

Next, we agree on some convention about trajectories. Let $p^{(j)}\in  \mathcal R_j$ be a zero of  $\varpi$. We denote by $\gamma_1(p^{(j)})$, $\gamma_2(p^{(j)})$, \dots, the 
trajectories of $\varpi$ \emph{emanating} from $p$ \emph{on} $\mathcal R_j$, in such a way that their canonical projections $\pi(\gamma_n(p^{(j)}))$, see 
Section~\ref{sec:RSassociatedalgcurve}, are enumerated in an anti-clockwise direction starting from the positive $OX$ semiaxis\footnote{ This notation is not correctly defined 
only 
if the direction of a trajectory coincides with a branch cut, situation that will be explicitly avoided in what follows.}.  
Notice that when $p^{(j)}$ is a branch point of $\mathcal R$, so that $p^{(j)}=p^{(k)}$ for some $j\neq k$, trajectories $\gamma_n(p^{(j)})$ and $\gamma_n(p^{(k)})$ are different 
because they emerge from $p^{(j)}=p^{(k)}$ on different sheets of $\mathcal R$. Otherwise, when $p^{(j)}$ belongs to a single sheet $\mathcal R_j$, we 
occasionally drop the superindex $(j)$ when it cannot lead us into confusion.

Given two points $p,q\in \mathcal \R$, the integral
$$
\int_p^q \sqrt{-\varpi}
$$
along a contour $\gamma$ connecting $p$ and $q$ is understood to be the integral of any analytic continuation of the meromorphic differential $\sqrt{-\varpi}$ along $\gamma$. This 
is well defined up to the branch of the square root, which will be clear in each context.

\subsubsection{Degenerate case \texorpdfstring{$\tau=0$}{}}  \label{section_degenerate}

For $\tau=0$, the algebraic equation \eqref{spectral_curve} reduces to
$$
\left(\xi +z^2\right) \left(4 \xi ^2-8 z^4-4 \xi  z^2+12 z - 3\times 3^{2/3}\right)=0,
$$
whose solutions, denoted in accordance with \eqref{asymptotics_xi_functions}, are 
\begin{equation*}
\xi_1(z)= \frac{z^2+\sqrt 3 \sqrt{h(z)}}{2}, \quad \xi_2(z)=-z^2,\quad \xi_3(z)= \frac{z^2-\sqrt 3 \sqrt{h(z)}}{2},
\end{equation*}
where the branch of the square root is chosen to be positive for large real values, 
$$
h(z)=3z^4-4z+3^{2/3}=3\left(z-a_2\right)\left(z-b_2\right)\left(z-b_*\right)^2,
$$ 
and the points $a_1,b_1,a_2,b_2,b_*$ are given explicitly by
\begin{equation*}
a_1=b_1=\frac{3^{2/3}}{4},\quad b_2=\overline{a_2}=\frac{1}{3^{1/3}}(-1+i\sqrt{2}),\quad b_*=\frac{1}{3^{1/3}}.
\end{equation*}

The cut $\Delta_1$ is reduced to a single point $a_1$, and the sheet $\mathcal R_2$ is detached from the others. Since \eqref{spectral_curve} is reducible, its Riemann surface is 
in fact the union of two Riemann surfaces, 
$$
\mathcal R_2=\overline\C \quad \text{and}\quad  \widetilde{\mathcal R}=\mathcal R_1 \cup \mathcal R_3.
$$
Here $\mathcal R_1=\mathcal R_3=\overline\C\setminus \Delta_2$, and $\Delta_2$ is a simple curve connecting the points $a_2, b_2$, to be precisely specified later. The 
quadratic differential \eqref{defQqd_cubic} degenerates into two quadratic differentials $\varpi_1$ on $\overline \C$ and $\varpi_2$ on $\widetilde{\mathcal R}$, namely
\begin{equation}\label{quadratic_differential_alpha0}
\begin{aligned}
\varpi_1 & = -3h(z)dz^2 \quad \mbox{on } \overline \C,  \\
\varpi_2 & = 
\begin{cases}
-\frac{1}{4}(3z^2-\sqrt{3}\sqrt{h(z)})^2dz^2 \quad \mbox{on } \mathcal R_1,  \\
-\frac{1}{4}(3z^2+\sqrt{3}\sqrt{h(z)})^2dz^2 \quad \mbox{on } \mathcal R_3.
\end{cases}
\end{aligned}
\end{equation}
We analyze the structure of their critical graphs next.
\medskip

\noindent \underline{Trajectories of $\varpi_1$}, whose only critical points are as follows:
\begin{itemize}
 \item Simple zeros at $z=a_2,b_2$;
 \item Double zero at $z=b_*$;
 \item Pole of order $8$ at $z=\infty$.
\end{itemize}

Under the change of variables $z\mapsto \frac{i}{\sqrt[3]{3}}z$, the quadratic differential $\varpi_1$ becomes $-q(z)dz^2$, where 
$$
q(z)=-(z^2-2iz-3)(z+i)^2
$$
is, up to a multiplicative factor $\frac{1}{4}$, the same polynomial obtained in \cite[eq.~(2.1)]{deano_kuijlaars_huybrechs_complex_orthogonal_polynomials}. Having in mind this 
identification, it was proven in \cite[Theorem~2.1]{deano_kuijlaars_huybrechs_complex_orthogonal_polynomials} that the trajectory $\gamma_3(b_2)$ of $\varpi_1$ connects
$b_2$ and $a_2$: in the notation introduced above, $\gamma_3(b_2)=\gamma_1(a_2)$, see Figure \ref{alpha0_3}.

From \cite{deano_kuijlaars_huybrechs_complex_orthogonal_polynomials} we also know that $\gamma_3(b_2)$ intersects the real axis at a point $a_*<b_*$, which can be calculated  
numerically: 
by
\begin{equation}\label{approximate_value_a_*}
a_*\approx -0.441782.
\end{equation}

The rest of the critical graph of $\varpi_1$  is as follows. Notice that due to the symmetry, we only need to describe the trajectories in the upper half plane.

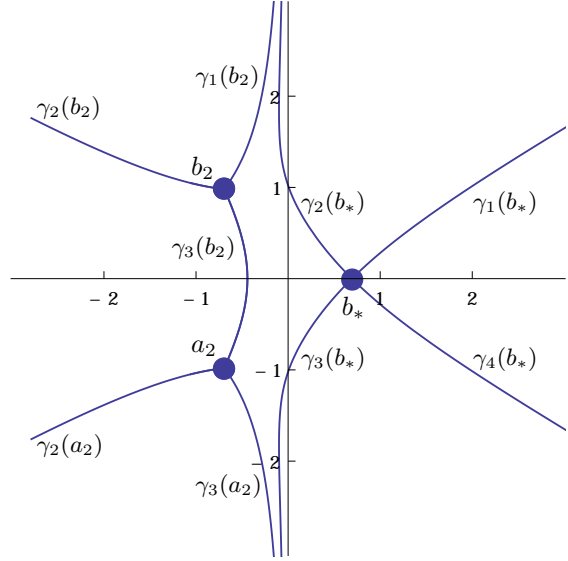
\begin{figure}
	\begin{center}
		\begin{overpic}[scale=0.7]{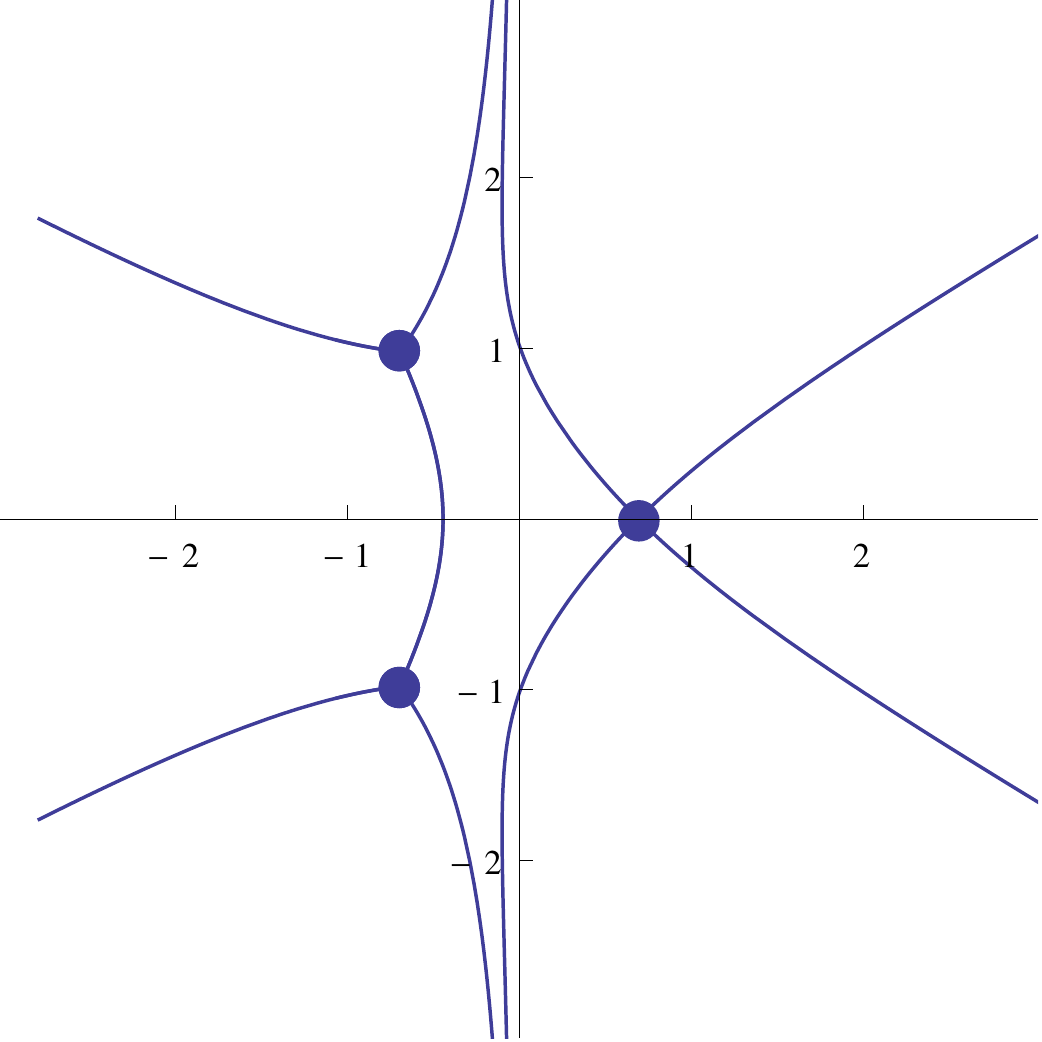}
			\put(175,132){\small $\gamma_1(b_*)$}
			\put(110,132){\small $\gamma_2(b_*)$}
			\put(175,73){\small $\gamma_4(b_*)$}
			\put(110,73){\small $\gamma_3(b_*)$}
			\put(70,180){\small $\gamma_1(b_2)$}
			\put(10,168){\small $\gamma_2(b_2)$}
			\put(70,25){\small $\gamma_3(a_2)$}
			\put(10,40){\small $\gamma_2(a_2)$}
			\put(61,115){\small $\gamma_3(b_2)$}
			\put(65,78){ $a_2$}
			\put(65,145){ $b_2$}
			\put(122,92){ $b_*$}
		\end{overpic}
		\caption{$\tau=0$: the critical graph of $\varpi_1$.}\label{alpha0_3}
	\end{center}
\end{figure}

The trajectories $\gamma_j(b_*)$, $\gamma_j(b_2)$, $j=1,2$, cannot be finite, see Principle \textbf{P.3} above. Hence, they all diverge to $\infty$ along the asymptotic directions 
\eqref{asymptotic_directions_cubic}, and according Theorem \ref{theorem_global_dissection}, all directions are represented. There are 3 asymptotic directions for 4 trajectories in 
the upper half plane, so necessarily the divergence angle for $\gamma_1(b_*)$ is $\theta_1^{(\infty)}$, while the divergence angle for $\gamma_2(b_2)$ is $\theta_3^{(\infty)}$. 
Since two consecutive trajectories emanating from a zero cannot diverge to $\infty$ in the same direction (this would contradict Theorem \ref{teichmuller_lemma}), we conclude that 
both $\gamma_2(b_*)$ and $\gamma_1(b_2)$ must diverge in the direction  $\theta_2^{(\infty)}$, see Figure~\ref{alpha0_3}. 

\medskip

\noindent \underline{Trajectories of $\varpi_2$}, whose only critical points are as follows:
\begin{itemize}
	\item Double zeros at $z=a_1^{(3)}$, $z=b_2^{(1)}$ and $z=a_2^{(1)}$;
	\item Double pole at $z=\infty^{(1)}$ with real residue;
	\item Pole of order $8$ at $\infty^{(3)}$.
\end{itemize}
The double zeros $b_2^{(1)}=b_2^{(3)}$ $a_2^{(1)}=a_2^{(3)}$ are also branch points of $\widetilde{\mathcal R}$, and the critical graph of $\varpi_2$ is made of trajectories 
$\gamma_j(a_2^{(1)})$, $\gamma_j(b_2^{(1)})$, $j=1, 2$, emanating on $\mathcal R_1$, and of trajectories $\gamma_j(a_2^{(3)})$, $\gamma_j(b_2^{(3)})$, $j=1, 2$, along with 
$\gamma_j(a_1^{(3)})$, $j=1, \dots, 4$, emanating on $\mathcal R_3$, see Figure~\ref{alpha0_4_5}.

The branch cut $\Delta_2$ connecting the sheets $\mathcal R_1$ 
and $\mathcal R_3$, so far arbitrary, is chosen as $\Delta_2=\gamma_3(b_2)$, where $\gamma_3(b_2)$ is the critical trajectory of $\varpi_1$ connecting $b_2$ 
to $a_2$, as described above. 

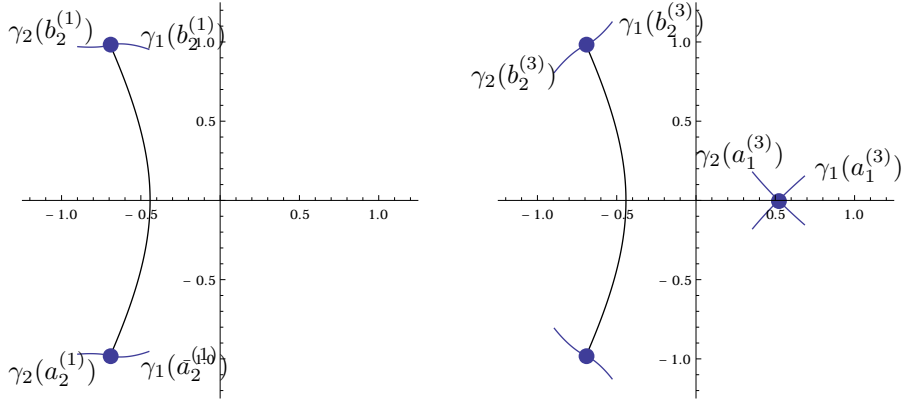
\begin{figure}
\begin{subfigure}{.5\textwidth}
  \centering
\begin{overpic}[scale=0.5]{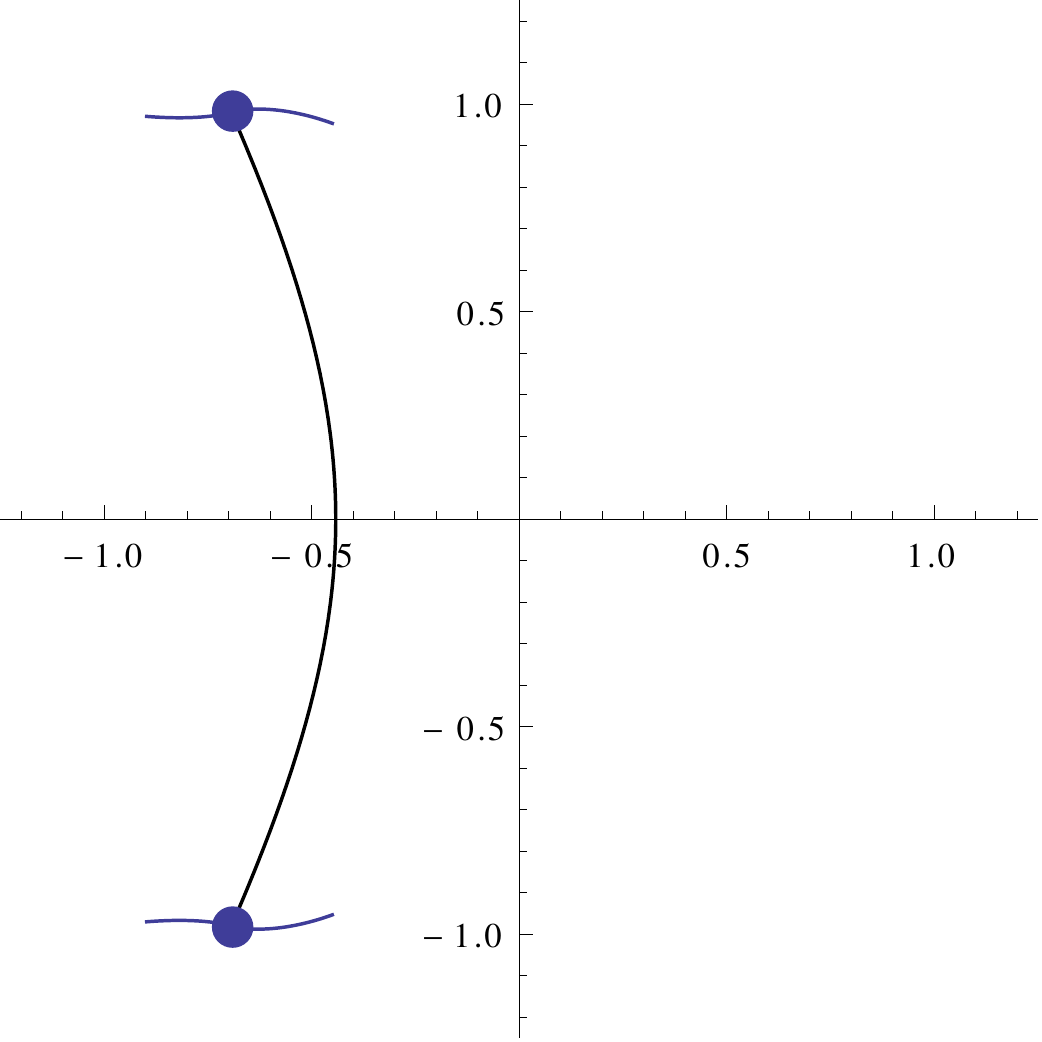}
\put(45,135){$\gamma_1(b_2^{(1)})$}
\put(-5,137){$\gamma_2(b_2^{(1)})$}
\put(45,10){$\gamma_1(a_2^{(1)})$}
\put(-5,8){$\gamma_2(a_2^{(1)})$}
\end{overpic}
\end{subfigure}%
\begin{subfigure}{.5\textwidth}
  \centering
\begin{overpic}[scale=0.5]{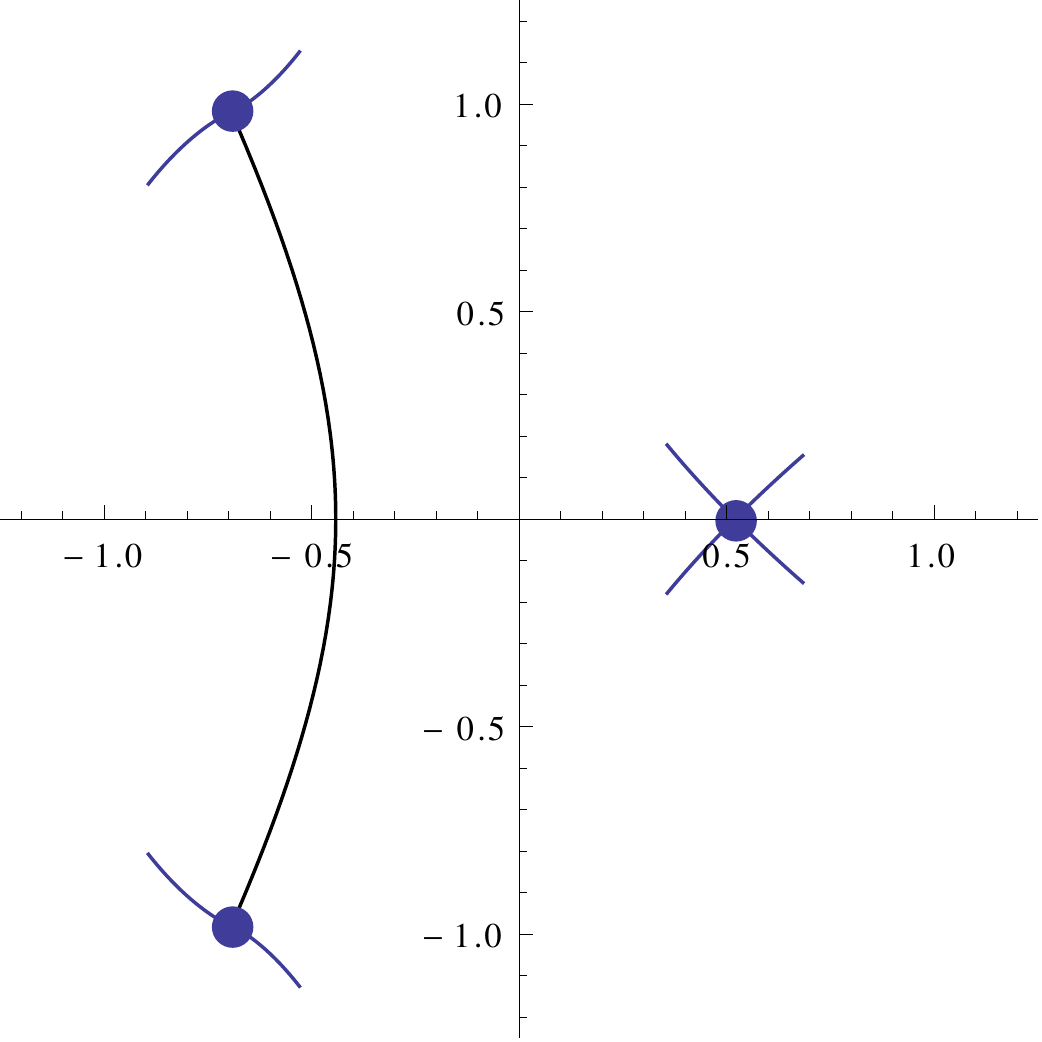}
\put(46,140){$\gamma_1(b_2^{(3)})$}
\put(-10,120){$\gamma_2(b_2^{(3)})$}
\put(120,85){$\gamma_1(a_1^{(3)})$}
\put(75,90){$\gamma_2(a_1^{(3)})$}
\end{overpic}
\end{subfigure}%
\caption{On the left it is represented the critical trajectories of $\varpi_2$ on $\mathcal R_1$ and on the right its critical trajectories on $\mathcal R_3$, all of them 
in blue and just locally at the critical points. The black curve is the cut $\Delta_2$ connecting $\mathcal R_1$ and $\mathcal R_3$.}\label{alpha0_4_5}
\end{figure}

\begin{lem}\label{lemma:tau=0}
	With the branch cut $\Delta_2$ specified above, the critical trajectories $\gamma_j(b_2^{(1)})$ and $\gamma_j(a_2^{(1)})$ of $\varpi_2$ belong entirely to the sheet 
$\mathcal R_1$. 
\end{lem}
\begin{proof}

Suppose that one of the trajectories $\gamma_j(b_2^{(1)})$ emanating from $b_2$ intersects the cut $\Delta_2$ for the fist time at a point $x$. Clearly, this 
point (actually, its canonical projection) must lie in the upper half plane: otherwise we readily get that $\gamma_j(b_2^{(1)})=\gamma_j(a_2^{(1)})$ and no 
intersection with $\Delta_2$ occur.

Integrating from $b_2$ to $x$ along $\gamma_j(b_2^{(1)})$ and using the definition of a trajectory we get
\begin{equation}\label{trajectories_zeta_eq_1}
 0 = \re \int_{b_2^{(1)}}^{x^{(1)}} \sqrt{-\varpi} =\frac{3}{4}\re\int_{b_2}^{x}s^2ds-\frac{\sqrt{3}}{4}\re\int_{b_2}^{x} \sqrt{h(s)}_\pm ds
\end{equation}
where the $\pm$ sign in the last integrand depends on the side of the cut $\Delta_2$ to which $x$ belongs. However, 
 $\Delta_2$ projects onto the trajectory $\gamma_3(b_2)$ of $\varpi_1$, so the second integral in the right-hand side of \eqref{trajectories_zeta_eq_1} is purely imaginary. Hence, 
this equation reduces to
\begin{equation}\label{point_intersection}
\re x^3=\re b_2^3 =\frac{5}{3}.
\end{equation}

It was proved in \cite{deano_kuijlaars_huybrechs_complex_orthogonal_polynomials} that the trajectory $\Delta_2$ is contained in the domain bounded by the triangle with vertices 
$b_2$, $3^{-1/3}(\sqrt 2 -1)$ and $a_2$. In particular, the part of $\Delta_2$ on the upper half plane, and hence $x$, is contained in the domain bounded by the triangle $T$ 
determined by the vertices $b_2$, $\re b_2$ and $3^{-1/3}(\sqrt 2 -1)$. The function $z\mapsto \re z^3$ has a unique maximum on $T$ at the point $z=b_2$. Since 
$\re z^3$ is harmonic, it cannot attain a maximum on the domain bounded by $T$, hence \eqref{point_intersection} can only occur if $x=b_2$, showing that 
$\gamma_j(b_2^{(1)})\setminus \{b_2^{(1)}\}$ does not intersect the cut $\Delta_2$.
\end{proof}

Recall that $\infty^{(1)}$ is the center of a circle domain, which means that all trajectories of $\varpi_2$ passing through sufficiently distant points on 
$\mathcal R_1$ are closed Jordan curves. In particular, no trajectory diverges to $\infty^{(1)}$, and every trajectory entirely contained in $\mathcal R_1$ has 
to be closed. Consequently, both trajectories  $\gamma_j(b_2^{(1)})$ are closed as well,   $\gamma_j(b_2^{(1)})=\gamma_j(a_2^{(1)})$, and $\gamma_1(b_2^{(1)})$ 
(respectively $\gamma_2(b_2^{(1)})$) intersects the real line, say at the point $c_*$ (respectively $d_*$). We claim that
\begin{equation}\label{tau_0_relative_positions}
d_*<a_*<c_*<a_1,
\end{equation}
with $a_*$ defined in \eqref{approximate_value_a_*}.

Indeed, both $c_*$ and $d_*$ cannot lie on the same side of the cut $\Delta_2$ without running into contradiction with the general principle \textbf{P.3} above. 
Hence, either $d_*<a_*<c_*$ or $c_*<a_*<d_*$. The latter is impossible without $\gamma_1(b_2^{(1)})$ and $\gamma_2(b_2^{(1)})$ intersecting somewhere in the 
upper half plane, which again contradicts \textbf{P.3}. We conclude that $d_*<a_*<c_*$.

Let us prove the inequality $c_*<a_1$. Using the definition of trajectory and \eqref{width_parameter_cubic} we get
\begin{align*}
0 & =\re \int_{b_2^{(1)}}^{c_*^{(1)}}Q(s) ds= \re \int_{b_2}^{c_*}(\xi_2(s)-\xi_3(s))ds = \omega_3+  \int_{a_1}^{c_*}(\xi_2(s)-\xi_3(s))ds.
\end{align*}
But $\omega_3=\omega_3(0)<0$, see Figure~\ref{figure_width_parameters}, so that
$$
\int_{a_1}^{c_*}(\xi_2(s)-\xi_3(s))ds>0.
$$
Function $\xi_2-\xi_3$ is continuous and non-vanishing  on $(a_*,+\infty)$, and by \eqref{asymptotics_xi_functions} it is negative for large real parameters, so it is  
negative on the whole interval $(a_*,+\infty)$. Since $c_*\in (a_*,+\infty)$, the equality above is only possible if $c_*<a_1$. This proves \eqref{tau_0_relative_positions}.

\begin{figure}[htb]
	\centering \begin{tabular}{ll} 
		\mbox{\begin{overpic}[scale=0.55]%
				{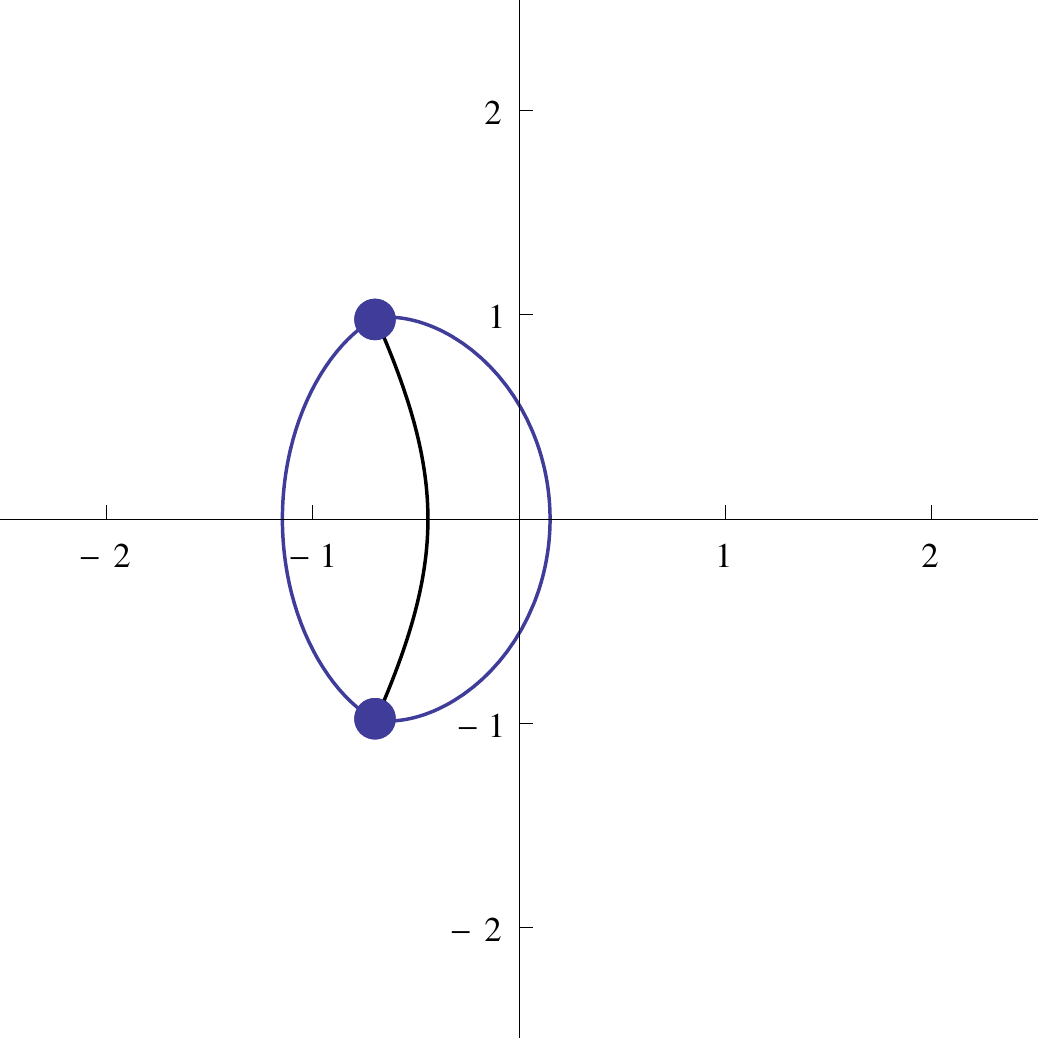}%
				\put(55,124){\small $b_2^{(1)} $}
				\put(55,37){\small $a_2^{(1)} $}
				\put(85,100){\small $\gamma_1(b_2^{(1)}) $}
				\put(18,100){\small $\gamma_2(b_2^{(1)}) $}
				\put(36,76){\small $d_* $}
				\put(90,76){\small $c_* $}
				\put(70,76){\small $a_* $}
				\put(69,90){\small $\Delta_2 $}
			\end{overpic}} &
			\mbox{\begin{overpic}[scale=0.55]%
					{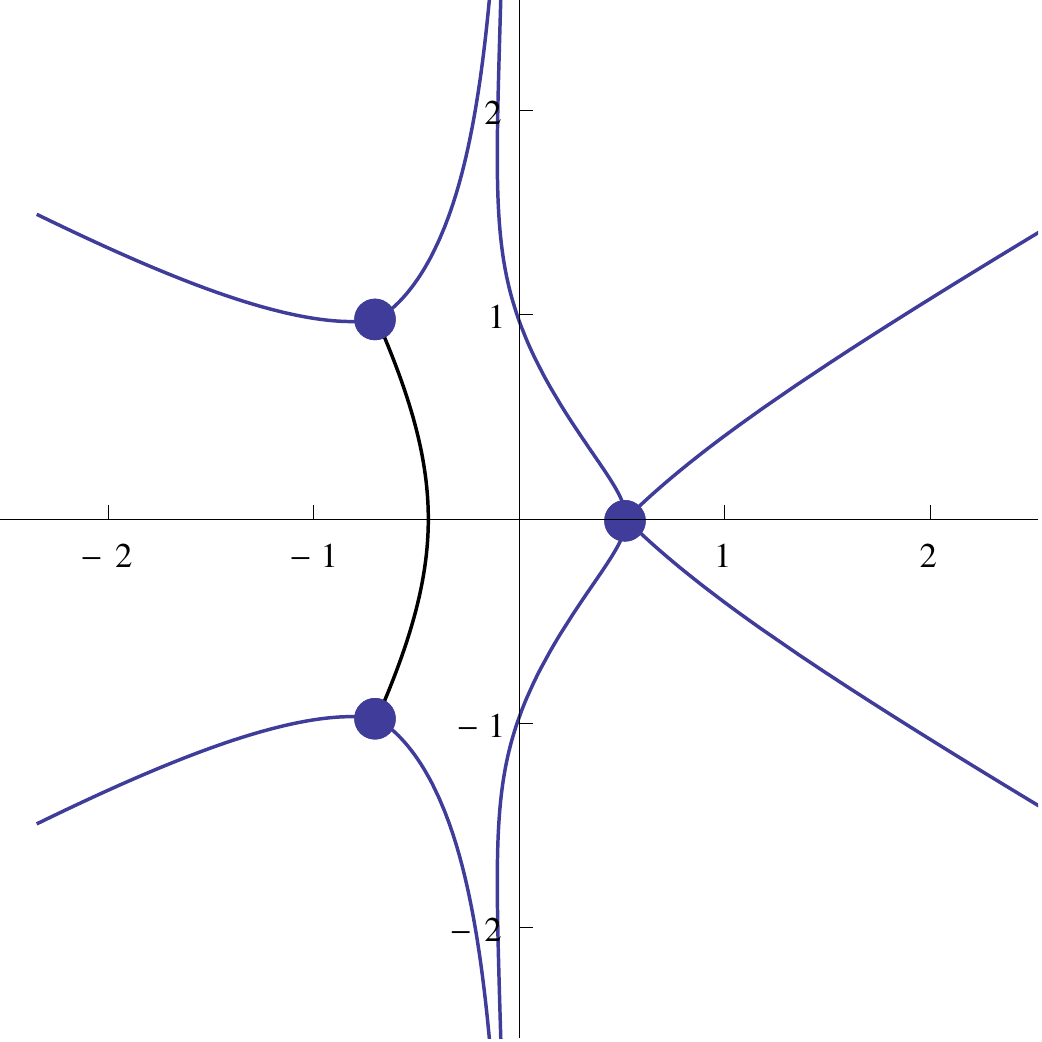}%
	\put(52,122){\small $b_2^{(3)} $}
	\put(52,37){\small $a_2^{(3)} $}
		\put(94,67){\small $a_1^{(3)} $}
		\put(45,150){\small $\gamma_1(b_2^{(3)}) $}
		\put(5,135){\small $\gamma_2(b_2^{(3)}) $}
		\put(125,128){\small $\gamma_1(a_1^{(3)}) $}
		\put(85,110){\small $\gamma_2(a_1^{(3)}) $}
		\put(85,50){\small $\gamma_3(a_1^{(3)}) $}
		\put(125,35){\small $\gamma_4(a_1^{(3)}) $}
		\put(55,90){\small $\Delta_2 $}
				\end{overpic}}
				\end{tabular}
				\caption{$\tau=0$: blue lines represent the critical graph of $\varpi_2$ on $\mathcal R_1$ (left) and $\mathcal R_3$. The black curve is the branch 
cut $\Delta_2$ connecting $\mathcal R_1$ and $\mathcal R_3$.}
				\label{alpha0_6_7}
			\end{figure}
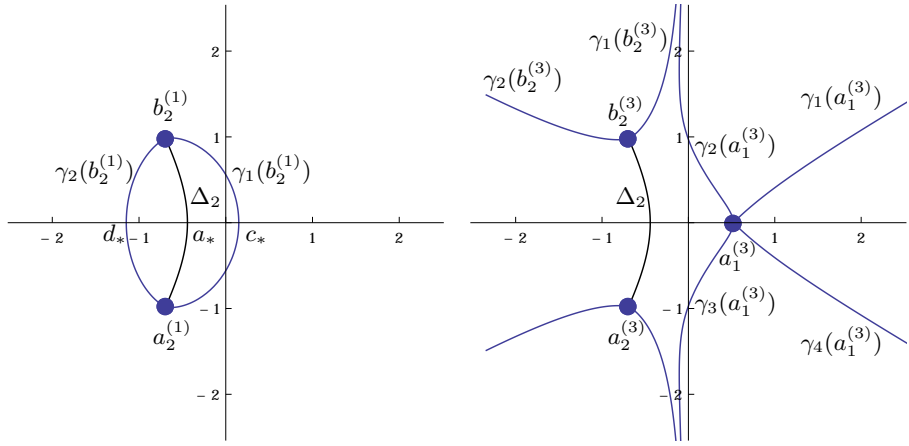

The discussion  of the structure of the trajectories $\gamma_1(a_1^{(3)})$, $\gamma_2(a_1^{(3)})$, $\gamma_1(b_2^{(3)})$, $\gamma_2(b_2^{(3)})$ is identical to the analysis of the 
trajectories $\gamma_j(b_*)$, $\gamma_j(b_2)$, $j=1,2$, on $\mathcal R_2$ for $\varpi_1$ above, so we omit the details. 

The global structure of the critical graph on $\varpi_2$ on both sheets is presented in Figure~\ref{alpha0_6_7}. The basic conclusion is that with the branch cut $\Delta_2$ 
specified above, the critical graph splits into two sets: a closed Jordan curve on $\mathcal R_1$, containing $a_2^{(1)}$ and $ b_2^{(1)}$, and 4 analytic arcs on $\mathcal R_3$, 
starting and ending at $\infty^{(3)}$, each passing through one of the branch points $a_2^{(3)}$, $ b_2^{(3)}$ and $a_1^{(3)}$.

\subsubsection{Trajectories for \texorpdfstring{$0<\tau<\frac{1}{12}$}{}}\label{section_trajectories_0_tau_1}

A combination of the general principles \textbf{P.2} and \textbf{P.6} assures us that the finite critical trajectories for 
$\tau=0$ remain finite for small perturbations of $\tau$, and that the behavior of the trajectories described for $\tau=0$ is preserved for $\tau$ small.

Let $\varepsilon>0$ be sufficiently small.
The general principle \textbf{P.4} above tells us that if we consider the domains $\Omega_\varepsilon^{(1)}$, (respectively $\Omega_\varepsilon^{(2)}$ and 
$\Omega_\varepsilon^{(3)}$), swept by trajectories of \eqref{quadratic_differential_alpha0} passing through points in the  $\varepsilon$-neighborhood of $a_1^{(1)}$ (respectively 
$a_1^{(2)}$ and $a_1^{(3)}$), then there exists a $\delta>0$ such that the critical trajectories for $\varpi$ and $0<\tau <\delta$, passing through $a_1^{(j)}$, $j=1, 2, 3$, 
belong to $\Omega_\varepsilon=\Omega_\varepsilon^{(1)} \cup \Omega_\varepsilon^{(2)} \cup \Omega_\varepsilon^{(3)}$. These domains are depicted schematically on 
Figure~\ref{figure_traj_strip_0_1-12}.

\begin{figure}
\begin{subfigure}{.5\textwidth}
\centering
\begin{overpic}[scale=1]{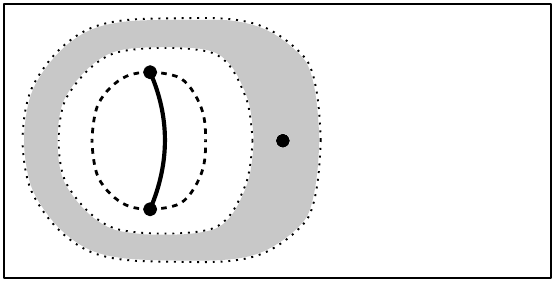}
	\put(70,60){\small $\Omega_\varepsilon^{(1)} $}
\end{overpic}
\caption*{$\mathcal R_1$}
\vspace{0.5cm}
\end{subfigure}%
\begin{subfigure}{.5\textwidth}
\centering
\begin{overpic}[scale=1]{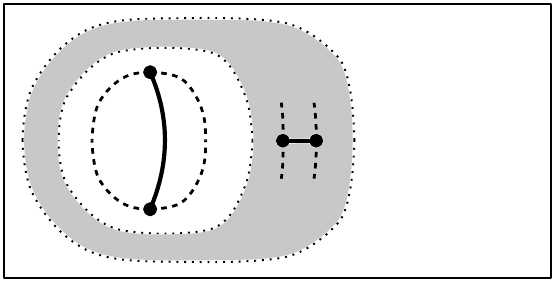}
	\put(70,60){\small $\Omega_\varepsilon^{(1)} $}
\end{overpic}
\caption*{$\mathcal R_1$}
\vspace{0.5cm}
\end{subfigure}
\begin{subfigure}{.5\textwidth}
\centering
\begin{overpic}[scale=1]{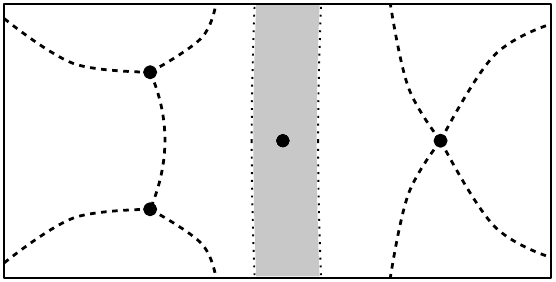}
	\put(75,60){\small $\Omega_\varepsilon^{(2)} $}
\end{overpic}
\caption*{$\mathcal R_2$}
\vspace{0.5cm}
\end{subfigure}%
\begin{subfigure}{.5\textwidth}
\centering
\begin{overpic}[scale=1]{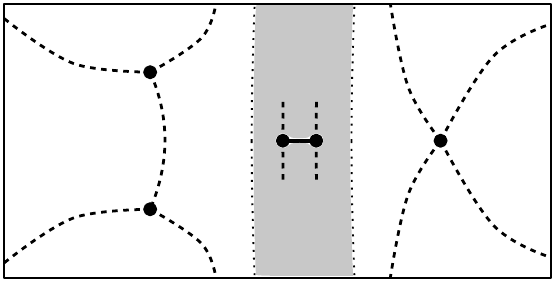}
	\put(80,60){\small $\Omega_\varepsilon^{(2)} $}
\end{overpic}
\caption*{$\mathcal R_2$}
\vspace{0.5cm}
\end{subfigure}
\begin{subfigure}{.5\textwidth}
\centering
\begin{overpic}[scale=1]{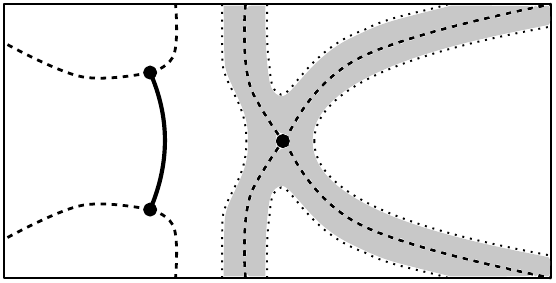}
	\put(120,70){\small $\Omega_\varepsilon^{(3)} $}
\end{overpic}
\caption*{$\mathcal R_3$}
\vspace{0.5cm}
\end{subfigure}%
\begin{subfigure}{.5\textwidth}
\centering
\begin{overpic}[scale=1]{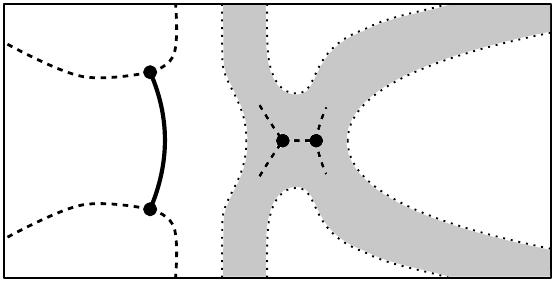}
	\put(120,68){\small $\Omega_\varepsilon^{(3)} $}
\end{overpic}
\caption*{$\mathcal R_3$}
\vspace{0.5cm}
\end{subfigure}
\caption{Left: critical graph of $\varpi$ for $\tau=0$, with the domains $\Omega_\varepsilon^{(j)}$, $j=1, 2, 3$, in gray. Right: local behavior or the critical trajectories for 
$\varpi$ and $\tau =\delta>0$, passing through $a_1^{(j)}$, $j=1, 2, 3$, with the same domains superimposed. }
\label{figure_traj_strip_0_1-12}
\end{figure}

We should keep in mind that the branch cut $\Delta_2$ now is completely specified by Definition~\ref{def:thebranchcut}, which is   consistent as long as the critical trajectory 
$\gamma_3(b_2^{(2)})$ joining $a_2^{(2)}$ and $b_2^{(2)}$, \emph{exists} and \emph{remains on $\mathcal R_2$} for the full range of the parameter $\tau$ under consideration. The 
forthcoming analysis shows that this is the case for $\tau\leq \tau_c$.

According to the general principle \textbf{P.4}, there exists a $\delta>0$ small enough such that for $0<\tau<\delta$, both $a_1^{(j)},b_1^{(j)}\in \Omega_\varepsilon^{(j)}$, 
$j=1, 2, 3$, and the critical trajectories emerging from  $a_1^{(j)},b_1^{(j)}$ stay in $\Omega_\varepsilon^{(j)}$. 

For instance, since for $\tau=0$ the critical trajectories  $\gamma_j(a_1^{(3)})$, $j=1, \dots, 4$, define three half-plane domains (bounded by  $\gamma_1(a_1^{(3)})\cup 
\gamma_2(a_1^{(3)})$, $\gamma_3(a_1^{(3)})\cup \gamma_4(a_1^{(3)})$ and $\gamma_1(a_1^{(3)})\cup \gamma_4(a_1^{(3)})$, see Figure~\ref{alpha0_6_7}, right), they must be persistent 
under small perturbation of $\tau$, and either $a_1^{(3)}$ or $b_1^{(3)}$, or both, must belong to their boundaries. Taking into account the structure of $\Omega_\varepsilon^{(3)}$ 
it is straightforward to conclude that for $0<\tau<\delta$, the trajectories of $\varpi$ through $a_1^{(3)}$ and $b_1^{(3)}$ are as shown in Figure~\ref{traj_final_0_tau_tau_1}.

We now examine the trajectories from $a_1^{(1)}$ and $b_1^{(1)}$. The following result comes in very handy: 
\begin{prop}\label{lemma:zerosofD}
Let 
\begin{equation}\label{eqRefirstsheet}
h(x,y):=	 \begin{cases} \displaystyle 
\int_{x }^y \re (\xi_{1+}(s)-\xi_3(s))\, ds=\int_{x }^y \re (\xi_{2+}(s)-\xi_3(s))\, ds, & \text{if } x,y \in \Delta_1, \\[3mm]
\displaystyle 
\int_{x }^y \re (\xi_1(s)- \xi_{2+}(s) )\, ds=\int_{x }^y \re (\xi_1(s) - \xi_{3+}(s))\, ds, & \text{if } x,y \in \Delta_3,
\end{cases}
\end{equation}
where we integrate along each interval. 

If $0<\tau <1/12$, then there exists no pair of values $x\neq y$, $x, y\in \Delta_1$,  such that $h(x,y)=0$.

If  $1/12<\tau <1/4$, and there exists a pair of values $x\neq y$, $x, y\in \Delta_1$ (resp., $x, y\in \Delta_3$) such that $h(x,y)=0$, then there exists no such pair of values on 
$\Delta_3$ (resp., on $\Delta_1$).

Furthermore, if  $x_1< y_1$, $x_2< y_2$ are two such pairs, $(x_1,y_1)\cap (x_2, y_2)\neq \emptyset$. 
\end{prop}

Notice that $h$ is well defined on $\Delta_1$ and $\Delta_3$ due to the symmetry relations \eqref{equality_xi_1} and \eqref{equality_xi_4}, and that $\Delta_3=\emptyset$ for 
$1/12<\tau<\tau_c$.
\begin{proof}
Assume that  there does exist a pair of values $x< y$, $x, y\in \Delta_1$,  such that $h(x,y)=0$ (same analysis is valid for $x, y \in \Delta_3$). By the mean value theorem, there 
exists a $u\in (x, y)$ such that 
$$
\re \xi_{1+}(u)=\re \xi_{2+}(u)= \xi_3(u),
$$
hence,
$$
0=\re \xi_{1+}(u)+\re \xi_{2+}(u)+\xi_{3}(u)=3\xi_3(u).
$$
But 
$$
D(u)=-\xi_{1+}(u)\xi_{2+}(u)\xi_3(u)=0,
$$
and the assertion follows from Proposition~\ref{prop_zeros_D}, keeping in mind that $D$ has exactly one zero on $\Delta_1\cup \Delta_3$ for  $1/12<\tau <1/4$.
\end{proof}

One of the consequences of Proposition~\ref{lemma:zerosofD} is that for $0<\tau<\delta$, 
the  trajectories emanating from $a_{1}^{(1)}$ and $b_{1}^{(1)}$ cannot cut $\Delta_1$ and must stay on the sheet $\mathcal R_1$. Thus, using again the general principle 
\textbf{P.4} we conclude that the trajectories $\gamma_k(a_1^{(1)})$, $\gamma_k(b_1^{(1)})$, $k=1,2$, are closed and encircle the cut $\Delta_2$, see 
Figure~\ref{traj_final_0_tau_tau_1}.

Similar considerations can be applied to get the behavior for the trajectories emanating from $a_1^{(2)}$, $b_1^{(2)}$, and the final result for $\tau$ small is seen in 
Figure~\ref{traj_final_0_tau_tau_1}. We skip the details. 

\begin{figure}
\begin{subfigure}{.5\textwidth}
\centering
\begin{overpic}[scale=1]{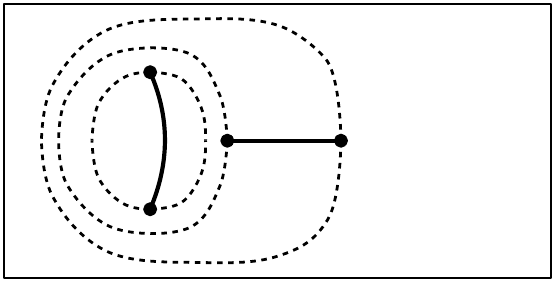}
\put(18,40){\scriptsize $S_1$}
\put(33,40){\scriptsize $S_3$}
\put(77,50){\scriptsize $S_4$}
\end{overpic}
\caption*{$\mathcal R_1$}
\vspace{0.5cm}
\end{subfigure}%
\begin{subfigure}{.5\textwidth}
\centering
\begin{overpic}[scale=1]{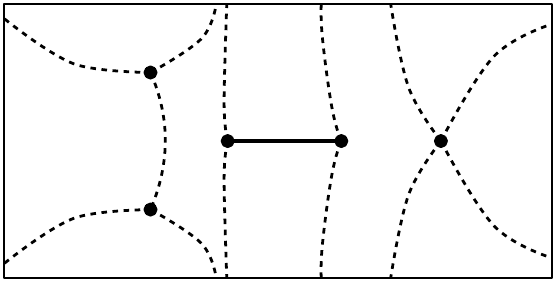}
\put(52,40){\scriptsize $S_2$}
\put(77,50){\scriptsize $S_4$}
\put(77,25){\scriptsize $S_4$}
\put(105,40){\scriptsize $S_5$}
\end{overpic}
\caption*{$\mathcal R_2$}
\vspace{0.5cm}
\end{subfigure}
\begin{subfigure}{1\textwidth}
\centering
\begin{overpic}[scale=1]{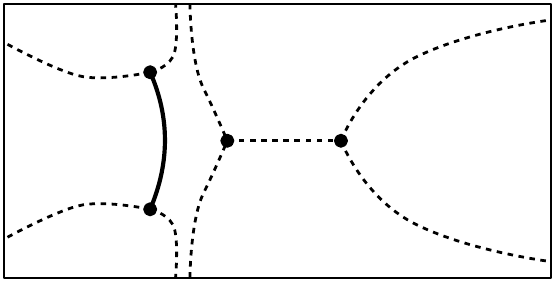}
\put(50,40){\scriptsize $S_3$}
\end{overpic} 
\caption*{$\mathcal R_3$}
\end{subfigure}
\caption{Critical graph of $\varpi$ for $0<\tau<1/12$, with the strip and ring domains labeled by $S_j$'s. Notice that some of these domains intersect more than one 
sheet.}\label{traj_final_0_tau_tau_1}
\end{figure}
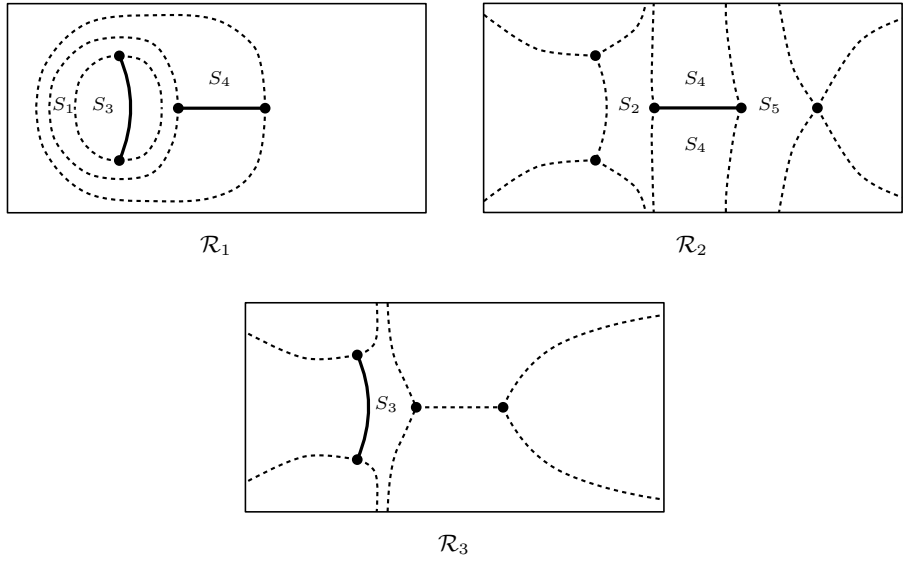

The outcome of our analysis is that the critical graph of $\varpi$ has the structure showed in Figure~\ref{traj_final_0_tau_tau_1}, at least for $0<\tau<\delta$. Our next goal is 
to prove that this is actually valid for $\tau \in (0,1/12)$. The continuity principle \textbf{P.4} yields that this is the case as long as 
\begin{enumerate}[\rm (i)]
\item No collision of the critical points occur: this is true indeed for $\tau \in (0,1/12)$, see Section~\ref{section_critical_points}.

\item No new domains emerge, which amounts to say that the finite trajectories for $0<\tau<\delta$ remain critical for $0<\tau<1/12$: this is assured by a combination of 
\textbf{P.2} and \textbf{P.6}.

\item No connected components of the complement of the critical graph ``disappear''. More precisely, it means that no width of any strip or ring domain becomes zero. These domains 
for $0<\tau<\delta$ are identified on Figure~\ref{traj_final_0_tau_tau_1}: there is one ring domain $S_1$ and 4 strip domains, $S_2,\dots, 
S_5$. There widths $\sigma(S_j)$ (see the definition \eqref{definition_width} in Section~\ref{sec:globalstructuretrajectoriesAppendix2}) are given by:
\begin{itemize}
\item $\sigma(S_j)=|\omega_j|$, $j=1,2,3$, as defined in \eqref{width_parameter_cubic}. They do not vanish for $\tau \in (0,1/12)$, see Figure~\ref{figure_width_parameters}.
\item $\sigma(S_4)=|h(a_1, b_1)| $, with $h$ defined in \eqref{eqRefirstsheet}, which does not vanish for $\tau \in (0,1/12)$, see Proposition~\ref{lemma:zerosofD}.
\item $\sigma(S_5)=\left|\int_{b_1}^{b_*}(\xi_1(s)-\xi_3(s))ds \right|$, which does not vanish for $\tau \in (0,1/12)$, see \eqref{inequality_xi_4}.
\end{itemize}

\end{enumerate}

We conclude that the critical graph of $\varpi$, depicted in Figure~\ref{traj_final_0_tau_tau_1}, is valid for the whole range $0<\tau<1/12$.  In particular, the critical 
trajectory $\gamma_1(a_2^{(2)})$ connects the points $a_2^{(2)}$ and $b_2^{(2)}$, which proves Theorem~\ref{thm:criticaltraj} for $0<\tau<1/12$.

\subsubsection{Trajectories for \texorpdfstring{$\frac{1}{12}<\tau<\tau_1$}{}}\label{section_1-12_to_tau1}

When $\tau=1/12$, the double point $b_*$ coincides with $b_1$, and the strip domain $S_5$ disappears ($\sigma(S_5)\searrow 0$ as $\tau \nearrow 1/12$), see 
Figure~\ref{figure_traj_strip_1-12_tau1}, left. Clearly, this transition has no impact on the structure of trajectories on the third sheet. Moreover, again a combination of 
\textbf{P.2} and \textbf{P.6} assures the finite trajectories for $\tau=1/12$ remain finite for $1/12<\tau<1/12+\delta$.

Let $\varepsilon>0$ be sufficiently small. Similarly to what has been done in the previous interval, the general principle \textbf{P.4}  tells us that if we consider the domains
$\Omega_\varepsilon^{(1)}$ and $\Omega_\varepsilon^{(2)}$, swept by trajectories of $\varpi$ passing through points in the  $\varepsilon$-neighborhood of $b_1^{(1)}$ and  
$b_1^{(2)}=b_*^{(2)}$, then there exists a $\delta>0$ such that the critical trajectories for $1/12<\tau <1/12+\delta$, passing through $b_1^{(j)}$, $j=1, 2$, and $b_*^{(1)}$, 
belong to $\Omega_\varepsilon=\Omega_\varepsilon^{(1)} \cup \Omega_\varepsilon^{(2)}$. These domains are also depicted schematically on Figure~\ref{figure_traj_strip_1-12_tau1}, 
left.

\begin{figure}
\begin{subfigure}{.5\textwidth}
\centering
\begin{overpic}[scale=1]{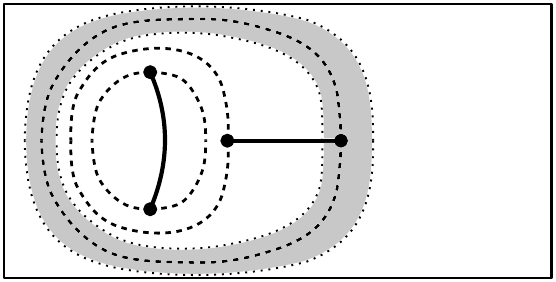}
\end{overpic}
\caption*{$\mathcal R_1$}
\vspace{0.5cm}
\end{subfigure}%
\begin{subfigure}{.5\textwidth}
\centering
\begin{overpic}[scale=1]{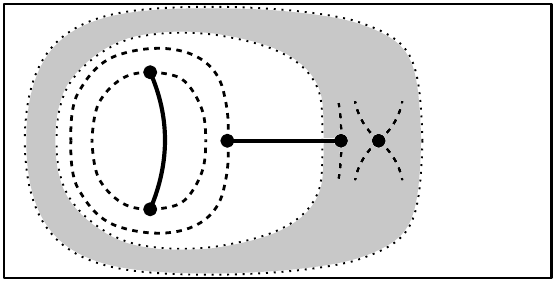}
\end{overpic}
\caption*{$\mathcal R_1$}
\vspace{0.5cm}
\end{subfigure}
\begin{subfigure}{.5\textwidth}
\centering
\begin{overpic}[scale=1]{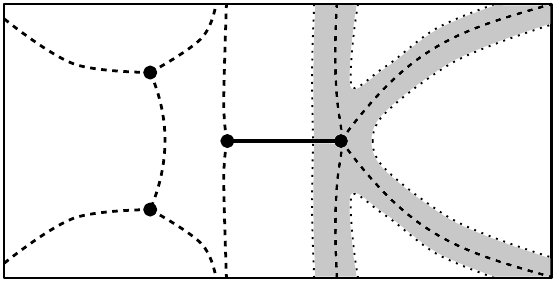}
\end{overpic}
\caption*{$\mathcal R_2$}
\vspace{0.5cm}
\end{subfigure}%
\begin{subfigure}{.5\textwidth}
\centering
\begin{overpic}[scale=1]{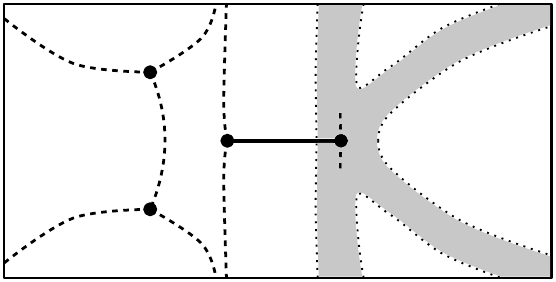}
\end{overpic}
\caption*{$\mathcal R_2$}
\vspace{0.5cm}
\end{subfigure}
\begin{subfigure}{.5\textwidth}
\centering
\begin{overpic}[scale=1]{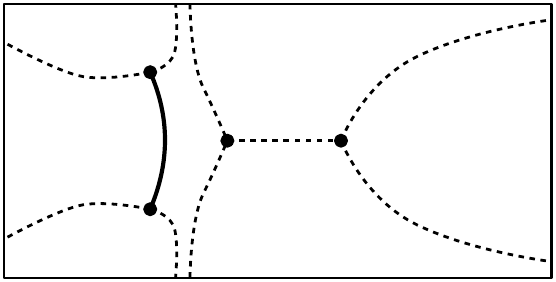}
\end{overpic}
\caption*{$\mathcal R_3$}
\vspace{0.5cm}
\end{subfigure}%
\begin{subfigure}{.5\textwidth}
\centering
\begin{overpic}[scale=1]{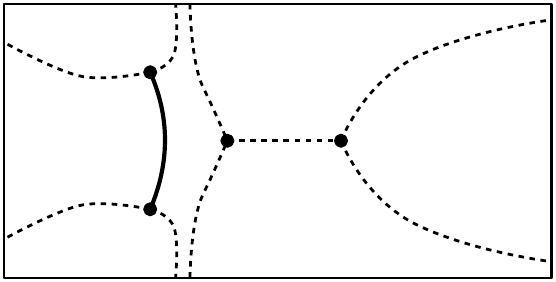}
\end{overpic}
\caption*{$\mathcal R_3$}
\vspace{0.5cm}
\end{subfigure}
\caption{Left: critical graph of $\varpi$ for $\tau=1/12$, with the domains $\Omega_\varepsilon^{(j)}$, $j=1, 2$, in gray. Right: local behavior or the critical trajectories for 
$\varpi$ and $\tau =1/12+\delta>0$, passing through $b_1^{(j)}$, $j=1, 2$, and $b_*^{(1)}$, with the same domains superimposed. }
\label{figure_traj_strip_1-12_tau1}
\end{figure}

For $1/12<\tau <1/12+\delta$ we consider the first sheet and the trajectories emanating from $b_1^{(1)}$ and $b_*^{(1)}$; thanks to principle {\bf P.2}, we concentrate on the 
upper half plane $\C_+$ (or to be precise, on its pre-image by $\pi$ on $\mathcal R_1$), namely $\gamma_1(b_*^{(1)})$, $\gamma_2(b_*^{(1)})$, $\gamma_1(b_1^{(1)})$, see 
Figure~\ref{figure_traj_strip_1-12_tau1}, right. These trajectories must stay in $\Omega_\varepsilon$, so they have to 
intersect $\pi^{-1}(\R)$ on $\mathcal R_1$. Let us denote the points of intersection of $\gamma_1(b_*^{(1)})$, $\gamma_2(b_*^{(1)})$, $\gamma_1(b_1^{(1)})$ by 
$x_1^{(1)}$, $x_2^{(1)}$, $x_3^{(1)}$, respectively. Using the general principles \textbf{P.2} and \textbf{P.3} we must immediately discard the following possibilities: (i) 
$x_j\geq b_1$ for some $j$,  (ii) $x_1\leq a_*^{(1)}$ and $x_2\leq a_*^{(1)}$ (recall that $a_*=\Delta_2 \cap \R$). Since trajectories cannot intersect, it holds $x_1< x_2< x_3$ 
and we conclude that necessarily $x_2\in (a_1,b_1)$, and consequently, $x_3\in (a_1,b_1)$ as well. In particular, $h(x_3, b_1)=0$, in the notation \eqref{eqRefirstsheet}.

Since $x_1$ and $x_2$ belong to trajectories with a common point $b_*^{(1)}$,  the assumption $x_1\in (a_1,b_1)$ yields that $h(x_1, x_2)=0$, and since $(x_1, x_2)\cap (x_3, 
b_1)=\emptyset$, this contradicts Proposition~\ref{lemma:zerosofD}.

From the considerations above, it follows that $\gamma_1(b_*^{(1)})$ is closed, stays on $\mathcal R_1$, and intersects $\pi^{-1}(\R)$ to the left of $\Delta_2$, and the 
trajectories 
$\gamma_2(b_*^{(1)})$, $\gamma_1(b_1^{(1)})$ intersect the cut $\Delta_1$ and move to the second sheet $\mathcal R_2$. We keep denoting these points of 
intersection by $x_2,x_3$ as before. Clearly, $x_3>x_2$, and Proposition~\ref{lemma:zerosofD} implies that these are the only points of intersection of these trajectories with the 
interval 
$(a_1,b_1)$.  

Let us turn to the second sheet, $\mathcal R_2$ and consider $\gamma_1(b_1^{(2)})$: from the structure of $\Omega_\varepsilon^{(2)}$ it is clear that it either diverges to 
$\infty^{(2)}$, or intersects the branch cut between $a_1$ and $b_1$ and moves to $\mathcal R_1$. If we assume the latter, $\Omega_\varepsilon^{(1)}$ shows that either it will 
return to $\mathcal R_2$ at a different point in $(a_1,b_1)$ (in contradiction with Proposition~\ref{lemma:zerosofD}), or it bounds a simply connected domain, which contradicts 
\textbf{P.3}. Hence, $\gamma_1(b_1^{(2)})$ must diverge to $\infty^{(2)}$, and thus $b_1^{(2)}$ lies on the boundary of the half plane domain bounded for $\tau=1/12$ by 
$\gamma_1(b_1^{(2)})$ and $\gamma_4(b_1^{(2)})$. In particular, by \textbf{P.5}, for $1/12<\tau<1/12+\delta$, $\gamma_1(b_1^{(2)})$ diverges to $\infty^{(2)}$ in the same 
asymptotic direction given by the angle $\theta_1^{(\infty)} $ from \eqref{asymptotic_directions_cubic}.

\begin{figure}
	\begin{subfigure}{.5\textwidth}
		\centering
		\begin{overpic}[scale=1]{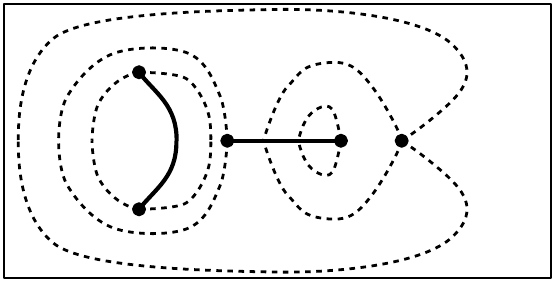}
			\put(18,40){\scriptsize $S_1$}
			\put(36,40){\scriptsize $S_3$}
			\put(68,50){\scriptsize $S_4$}
			\put(102,40){\scriptsize $S_5$}
		\end{overpic}
		\caption*{$\mathcal R_1$}
		\vspace{0.5cm}
	\end{subfigure}%
	\begin{subfigure}{.5\textwidth}
		\centering
		\begin{overpic}[scale=1]{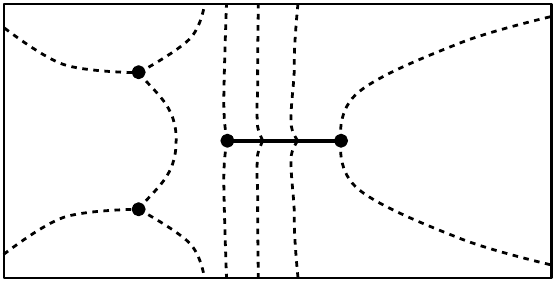}
			\put(52,40){\scriptsize $S_2$}
			\put(75,50){\scriptsize $S_5$}
			\put(75,25){\scriptsize $S_5$}
			\put(65,50){\scriptsize $S_4$}
			\put(65,25){\scriptsize $S_4$}
		\end{overpic}
		\caption*{$\mathcal R_2$}
		\vspace{0.5cm}
	\end{subfigure}
	\begin{subfigure}{1\textwidth}
		\centering
		\begin{overpic}[scale=1]{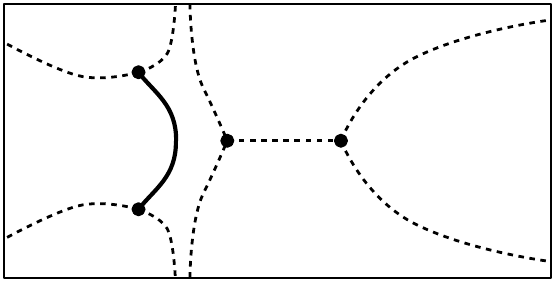}
			\put(53,40){\scriptsize $S_3$}
		\end{overpic} 
		\caption*{$\mathcal R_3$}
	\end{subfigure}
	\caption{Critical graph of $\varpi$ for $1/12<\tau<\tau_1$, with the strip and ring domains labeled by $S_j$'s. Notice that some of these domains intersect more than one 
		sheet.}\label{traj_final_1-12_to_tau1}
\end{figure}
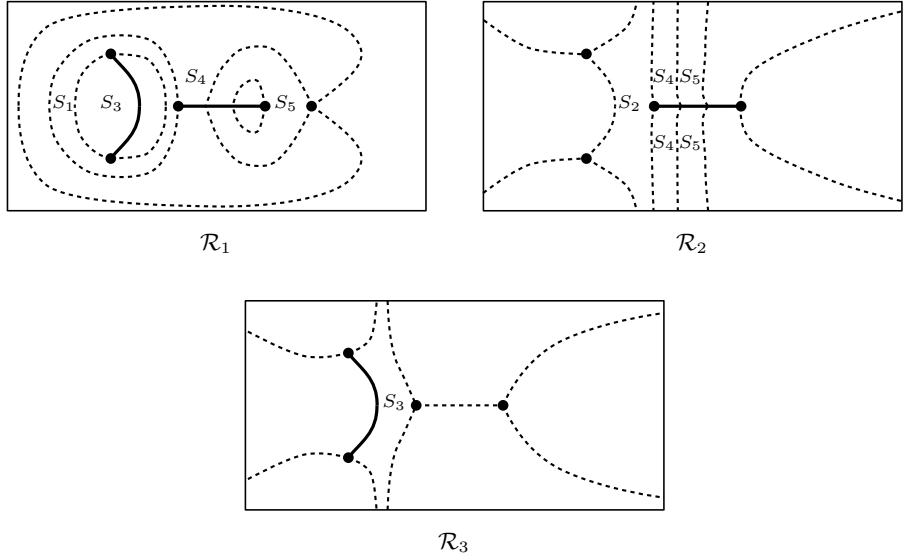

Recall that we concluded that the trajectory $\gamma_1(b_1^{(1)})$ enters $\mathcal R_2$ through the cut $(a_1, b_1)$ at a point $x_3$. The only possibility left for it is to go 
to 
$\infty^{(2)}$. Applying Theorem~\ref{teichmuller_lemma} to the $\varpi$-polygon bounded by the trajectories $\gamma_1(b_1^{(1)})$ and $\gamma_2(b_1^{(2)})$, we get that 
$\gamma_1(b_1^{(1)})$ goes to $\infty^{(2)}$ in the asymptotic direction given by the angle $\theta_5^{(\infty)}$, and consequently $\gamma_1(b_*^{(1)})$ extends to $\infty^{(2)}$ 
with angle $\theta_5$ as well. 

The outcome of our analysis is that for $1/12<\tau<1/12+\delta$ the critical graph of $\varpi$ has the structure showed in Figure~\ref{traj_final_1-12_to_tau1}. We prove that this 
is actually valid for $\tau \in (1/12, \tau_1)$. Again, the continuity principle \textbf{P.4} yields that this is the case as long as (i) no collision of the critical points occur 
(this is true indeed for $\tau \in (1/12, 1/4)$, see Section~\ref{section_critical_points}), (ii) finite critical trajectories for $1/12<\tau<1/12+\delta$ remain finite for 
$1/12<\tau<\tau_1$ (assured by a combination of \textbf{P.2} and \textbf{P.6}); and (iii) no width of any strip and ring domains become zero. These domains for 
$1/12<\tau<1/12+\delta$ are identified on Figure~\ref{traj_final_1-12_to_tau1}: there is one ring domain $S_1$ and 4 strip domains, $S_2,\dots, S_5$. There widths $\sigma(S_j)$ 
are 
given by:
\begin{itemize}
	\item $\sigma(S_j)=|\omega_j|$, $j=1,2,3$, as defined in \eqref{width_parameter_cubic}. They do not vanish for $1/12<\tau<\tau_1$, see Figure~\ref{figure_width_parameters}, 
although $\sigma(S_1)$ does vanish for $\tau=\tau_1$.
	\item $\sigma(S_4) $ is given by the absolute value of
	\begin{align}
	\re \int_{a_1}^{b_*}(\xi_2(s)-\xi_3(s)) & =\re \int_{a_1}^{b_2}(\xi_2(s)-\xi_3(s))ds+\re \int_{b_2}^{b_*}(\xi_2(s)-\xi_3(s))ds \nonumber \\
	& =-\omega_1+\omega_4, \label{omega1omega4}
	\end{align}
	and $\omega_1\neq \omega_4$ for $\tau>1/12$, see Figure~\ref{figure_width_parameters} in Section~\ref{section_widths_cubic}.
	\item $\sigma(S_5)=\left|\int_{b_1}^{b_*}(\xi_{2}(s)-\xi_3(s))ds \right|$, which does not vanish for $\tau> 1/12$, see \eqref{inequality_xi_5}.
\end{itemize}

We conclude that the critical graph of $\varpi$, depicted in Figure~\ref{traj_final_1-12_to_tau1}, is valid for the whole range  $1/12<\tau<\tau_1$. In particular this yields
Theorem~\ref{thm:criticaltraj} in the mentioned range of $\tau$.

\subsubsection{Trajectories for \texorpdfstring{$\tau_1<\tau<\tau_c$}{}}\label{section_trajectories_tau1_tauc}

\begin{figure}
\begin{subfigure}{.5\textwidth}
\centering
\begin{overpic}[scale=1]{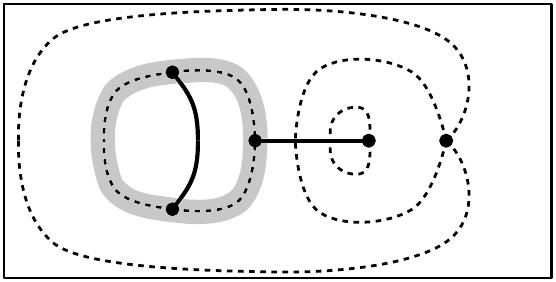}
\end{overpic}
\caption*{$\mathcal R_1$}
\vspace{0.5cm}
\end{subfigure}%
\begin{subfigure}{.5\textwidth}
\centering
\begin{overpic}[scale=1]{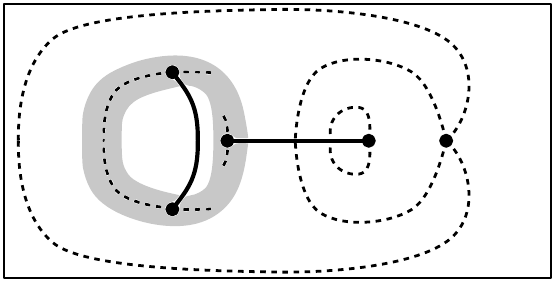}
\end{overpic}
\caption*{$\mathcal R_1$}
\vspace{0.5cm}
\end{subfigure}
\begin{subfigure}{.5\textwidth}
\centering
\begin{overpic}[scale=1]{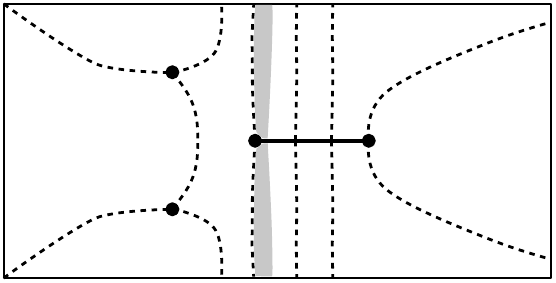}
\end{overpic}
\caption*{$\mathcal R_2$}
\vspace{0.5cm}
\end{subfigure}%
\begin{subfigure}{.5\textwidth}
\centering
\begin{overpic}[scale=1]{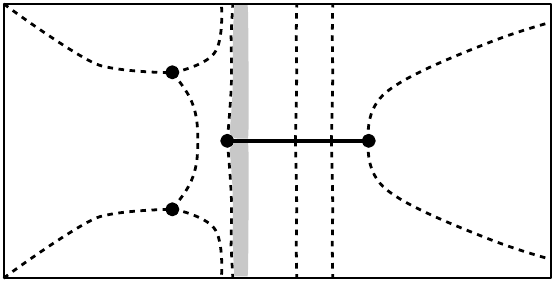}
\end{overpic}
\caption*{$\mathcal R_2$}
\vspace{0.5cm}
\end{subfigure}
\begin{subfigure}{.5\textwidth}
\centering
\begin{overpic}[scale=1]{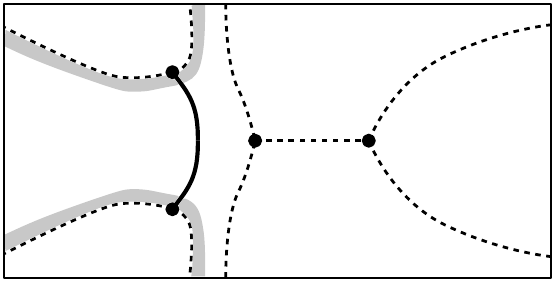}
\end{overpic}
\caption*{$\mathcal R_3$}
\vspace{0.5cm}
\end{subfigure}%
\begin{subfigure}{.5\textwidth}
\centering
\begin{overpic}[scale=1]{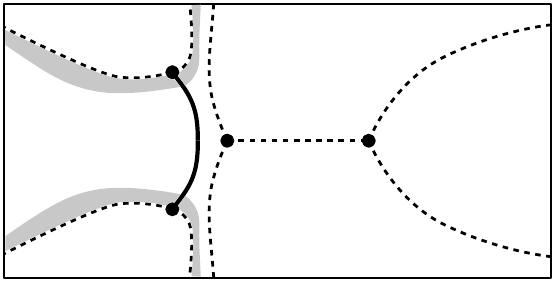}
\end{overpic}
\caption*{$\mathcal R_3$}
\vspace{0.5cm}
\end{subfigure}
\caption{Left: critical graph of $\varpi$ for $\tau=\tau_1$, with the domain $\Omega_\varepsilon$ in gray. Right: local behavior or the critical trajectories for $\varpi$ and $\tau 
=\tau_1+\delta>0$, passing through $b_2^{(1)}$ and  $a_1^{(1)}$, with the same domain superimposed.}\label{figure_traj_strip_tau1_tauc}
\end{figure}

At the value $\tau=\tau_1$ the critical trajectory $\gamma_1(b_2^{(1)})$ on the first sheet hits the branch point $a_1^{(1)}$, so that the ring domain $S_1$ disappears 
($\sigma(S_1)\searrow 0$ as $\tau \nearrow \tau_1$), see Figure~\ref{traj_final_1-12_to_tau1}. From the analysis of the behavior of the rest of the widths $\sigma(S_j)$ and the 
other finite critical trajectories, it follows that this fact does not affect the rest of the strip domains - note that there are no other ring domains. In particular, the 
trajectories emerging from $a_1$, $a_2$, and $b_2$ on the sheets $\mathcal R_2$, $\mathcal R_3$ do not display any phase transition.

The critical graph for $\tau=\tau_1$ is depicted in Figure~\ref{figure_traj_strip_tau1_tauc}. In accordance with the methodology we have followed so far, we fix an $\varepsilon>0$  
sufficiently small and consider the domain
$\Omega_\varepsilon$ swept by trajectories of $\varpi$ passing through points in the  $\varepsilon$-neighborhood of $b_2^{(1)}$ and  $a_1^{(1)}=a_*^{(1)}$. Notice that 
$\Omega_\varepsilon$ no longer lives on the single sheet, and its boundary now also contains critical trajectories, 
namely $\gamma_j(a_1^{(2)})$, $\gamma_j(a_2^{(3)})$, $\gamma_j(b_2^{(3)})$, $j=1,2$. This is so because, as already mentioned, there is no transition for these trajectories. 
Observe also that when $\tau\nearrow \tau_1$ the trajectory $\gamma_2(b_2^{(1)})=\gamma_2(a_2^{(1)})$ does not collide with any critical point other than its endpoints, hence its 
topology is unchanged under small perturbations of $\tau$ around $\tau_1$. In fact, there exists a $\delta>0$ such that the critical trajectories for $\tau_1<\tau <\tau_1+\delta$, 
passing through $b_1^{(1)}$, belong to $\Omega_\varepsilon$. This domain is  also depicted schematically on Figure~\ref{figure_traj_strip_tau1_tauc}, left.

We now consider the possible behavior of $\gamma_1(b_2^{(1)})$ for $\tau_1<\tau <\tau_1+\delta$, having in mind that it cannot leave the shaded region $\Omega_\varepsilon$, which 
shows that either $\gamma_1(b_2^{(1)})$ moves immediately to $\mathcal R_3$ through $\Delta_2$, or it intersects the (preimage of) the real line near $a_1$. 

In the first case, $\gamma_1(b_2^{(1)})$ extends to $\infty^{(3)}$ with angle $\theta_3^{(\infty)}$. Thus $\gamma_1(b_2^{(1)})\cup \gamma_2(b_2^{(3)})$  is 
the boundary of a $\varpi$-polygon for which $\kappa=1$ and $\lambda=0$, contradicting \eqref{teichmuller_formula}.

Assume otherwise, so that $\gamma_1(b_2^{(1)})$ intersects the (preimage of the) real line  at a point 
$x^{(1)}$, $x$ close to $a_1$. If $x\leq a_1$, we integrate from $b_2$ to $x$ over $\gamma_1(b_2^{(1)})$ and then from $x$ to $a_1$ over the real line to get
\begin{equation}\label{calculationOmega1}
\omega_1=\re\int_{b_2}^{x}(\xi_2(s)-\xi_3(s))ds+\int_{x}^{a_1}(\xi_2(s)-\xi_3(s))ds=\int_{x}^{a_1}(\xi_2(s)-\xi_3(s))ds.
\end{equation}
In the range of $\tau$ considered, $\omega_1>0$, while the last integral is $\leq 0$, see  \eqref{inequality_xi_3}, which leads us into a contradiction. Hence, $x \in (a_1,b_1)$, 
so that $\gamma_1(b_2^{(1)})$ moves to the second sheet. Recall that if $x_3^{(1)}$ is the point of intersection of $\gamma_2(b_1^{(1)})$ with $\Delta_1$, we already have 
$h(x_3,b_1)=0$, so by Proposition~\ref{lemma:zerosofD}, $\gamma_1(b_2^{(1)})$ cannot return to $\mathcal R_1$; in consequence, it stays on the second sheet and diverges to 
$\infty^{(2)}$ in the asymptotic direction corresponding to the angle $\theta_5^{(\infty)}$, see Figure~\ref{figure_traj_prel2_tau1_to_tau_2}.

\begin{figure}
	\begin{subfigure}{.5\textwidth}
		\centering
		\begin{overpic}[scale=1]{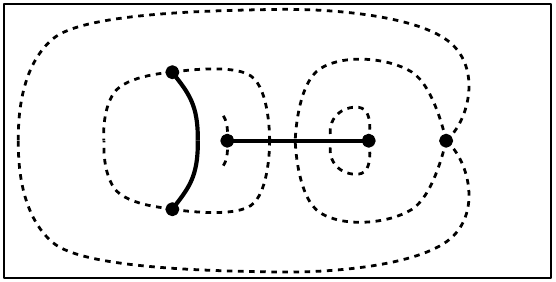}
		\end{overpic}
		\caption*{$\mathcal R_1$}
		\vspace{0.5cm}
	\end{subfigure}%
	\begin{subfigure}{.5\textwidth}
		\centering
		\begin{overpic}[scale=1]{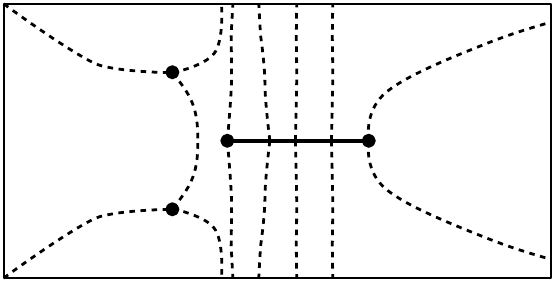}
		\end{overpic}
		\caption*{$\mathcal R_2$}
		\vspace{0.5cm}
	\end{subfigure}
	\caption{Some trajectories of $\varpi$ on the first two sheets, for $\tau_1<\tau<\tau_c$.}\label{figure_traj_prel2_tau1_to_tau_2}
\end{figure}

On the other hand, trajectories $\gamma_j(a_1^{(1)})$, $j=1,2$, are ``trapped'' between $\gamma_1(b_2^{(1)})$ and the branch cut $\Delta_2^{(2)}$, see 
Figure~\ref{figure_traj_prel2_tau1_to_tau_2},  left. Thus, they cannot stay on $\mathcal R_1$ without violating the general principle \textbf{P.3}, so they   move to the third 
sheet and diverge to $\infty^{(3)}$  in the asymptotic directions corresponding to the angles $\theta_3^{(\infty)}$ and $\theta_4^{(\infty)}$,  respectively.

The outcome of our analysis is that for $\tau_1<\tau<\tau_1+\delta$ the critical graph of $\varpi$ has the structure showed in Figure~\ref{traj_final_tau1_to_tauc}. Same analysis 
as in the previous section shows that  this is actually valid for the whole range $\tau_1<\tau<\tau_c$, in particular implying Theorem~\ref{thm:criticaltraj} for this range of 
$\tau$.

\begin{figure}
\begin{subfigure}{.5\textwidth}
\centering
\begin{overpic}[scale=1]{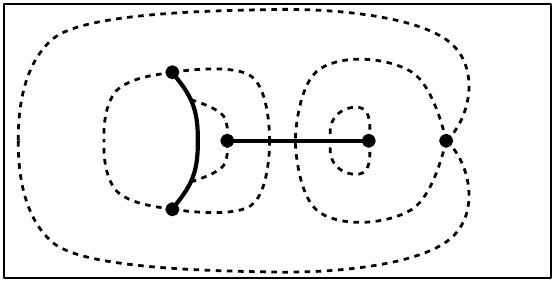}
\end{overpic}
\caption*{$\mathcal R_1$}
\vspace{0.5cm}
\end{subfigure}%
\begin{subfigure}{.5\textwidth}
\centering
\begin{overpic}[scale=1]{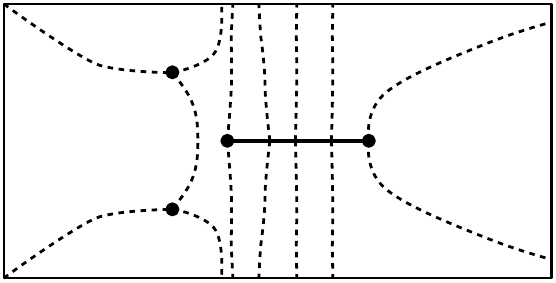}
\put(56,55){\scriptsize $S_2$}
\end{overpic}
\caption*{$\mathcal R_2$}
\vspace{0.5cm}
\end{subfigure}
\begin{subfigure}{1\textwidth}
\centering
\begin{overpic}[scale=1]{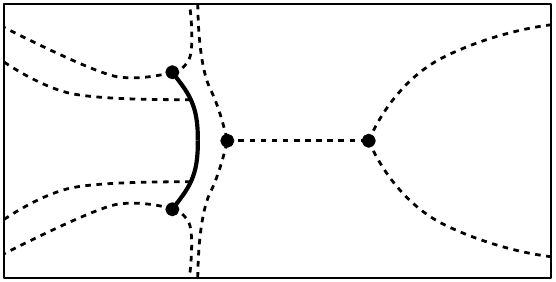}
\end{overpic} 
\caption*{$\mathcal R_3$}
\end{subfigure}
\caption{Global structure of trajectories for $\tau_1<\tau<\tau_c$.}\label{traj_final_tau1_to_tauc}
\end{figure}

\subsubsection{Trajectories for \texorpdfstring{$\tau_c<\tau<\tau_2$}{}}\label{section_trajectories_tauc_tau2}

We come to the topologically most important phase transition. According to Definition~\ref{def:thebranchcut}, we use the lift of the trajectory joining $a_2^{(2)}$ and $b_2^{(2)}$ 
on $\mathcal R_2$ as the branch cut connecting the sheets $\mathcal R_1$ and $\mathcal R_3$. When $\tau=\tau_c$, the strip domain $S_2$ disappears ($\sigma(S_2)\searrow 0$ as 
$\tau \nearrow \tau_c$), and this trajectory finally reaches the branch point $a_1^{(2)}$,  see Figure~\ref{traj_final_tau1_to_tauc}. Note that this transition corresponds to 
\eqref{equation_definition_alphac}. We fix an $\varepsilon>0$  sufficiently small and consider the domain $\Omega_\varepsilon$ swept by trajectories of $\varpi$ passing through 
points in the  $\varepsilon$-neighborhood of $a_1^{(2)}$, $a_2^{(2)}$, and $b_2^{(2)}$. The  critical graph for $\tau=\tau_c$ along with $\Omega_\varepsilon$ is displayed in 
Figure~\ref{traj_strip_tauc}.

\begin{figure}
\begin{subfigure}{.5\textwidth}
\centering
\begin{overpic}[scale=1]{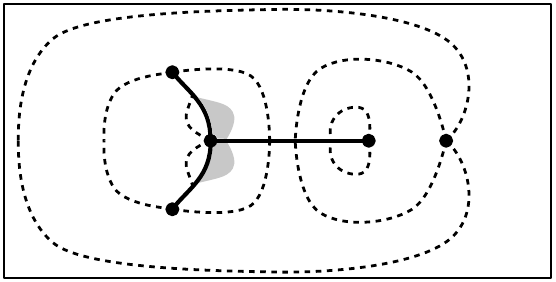}
\end{overpic}
\caption*{$\mathcal R_1$}
\vspace{0.5cm}
\end{subfigure}%
\begin{subfigure}{.5\textwidth}
\centering
\begin{overpic}[scale=1]{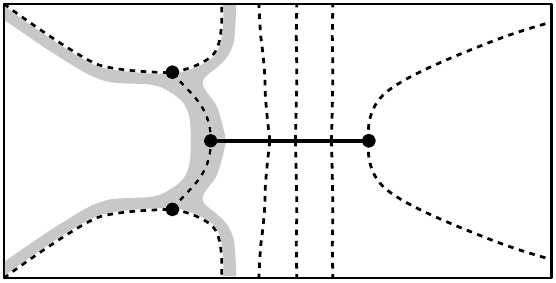}
\end{overpic}
\caption*{$\mathcal R_2$}
\vspace{0.5cm}
\end{subfigure}
\begin{subfigure}{1\textwidth}
\centering
\begin{overpic}[scale=1]{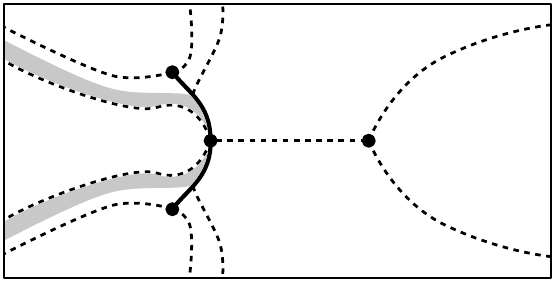}
\end{overpic} 
\caption*{$\mathcal R_3$}
\end{subfigure}
\caption{Critical graph of $\varpi$ for $\tau=\tau_c$, with the domain $\Omega_\varepsilon$ in gray, consisting of the trajectories for $\varpi$ passing through the  
$\varepsilon$-neighborhood of $a_1^{(2)}$, $a_2^{(2)}$, and $b_2^{(2)}$.}\label{traj_strip_tauc}
\end{figure}

For $\tau_c<\tau <\tau_c+\delta$ we examine first the second sheet and the trajectories emanating from $a_1^{(2)}$, $a_2^{(2)}$, and $b_2^{(2)}$. For a sufficiently small 
$\delta>0$, these trajectories must stay in $\Omega_\varepsilon$, see Figure~\ref{figure_traj_prel1_tauc_tau2}.

\begin{figure}
\centering
\begin{overpic}[scale=1]{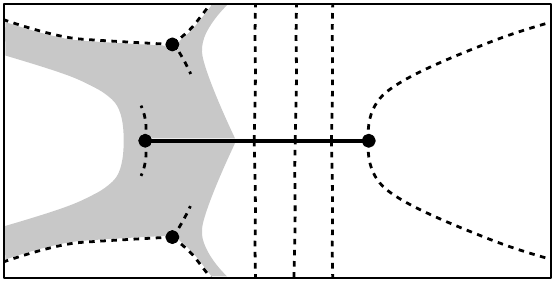}
\end{overpic}
\caption{The domain $\Omega_\varepsilon$ on $\mathcal R_2$ for $\tau_c<\tau<\tau_c+\delta$.}\label{figure_traj_prel1_tauc_tau2}
\end{figure}

Analyzing $\gamma_1(a_1^{(2)})$ we have to discard the following possibilities:
\begin{itemize}
 \item  $\gamma_1(a_1^{(2)})$ cannot intersect (the preimage by $\pi$ of) the real line to the left of $a_1$ without violating the general principle \textbf{P.3}.
 
 \item it cannot intersect (the preimage by $\pi$ of) the real line to  on the cut $(a_1,b_1)$ either:  otherwise the equation $h(x,a_1)=0$ has a solution in $x\in (a_1, b_1)$, 
along with the identity $h(x_3,b_1)=0$, where  $x_3^{(1)}$ is the point of intersection of $\gamma_2(b_1^{(1)})$ with $\Delta_1$, and since $(a_1,x)\cap (x_3, b_1)=\emptyset$, this 
would contradict Proposition~\ref{lemma:zerosofD}.
 
 \item $\gamma_1(a_1^{(2)})\neq \gamma_3(b_2^{(2)})$, because $\omega_2\neq 0$  for $\tau>\tau_c$, see Figure~\ref{figure_width_parameters}.
 
 \item $\gamma_1(a_1^{(2)})$ cannot diverge  to $\infty^{(2)}$ in the asymptotic direction $\theta_2^{(\infty)}$. Indeed,  otherwise either $\gamma_3(b_2^{(2)})$ also diverges to 
$\infty^{(2)}$ in the same direction, or it intersects the real axis to the left of $a_1^{(2)}$. In the former case, we get a $\varpi$-polygon for which $\kappa=1$ and $\lambda=0$, 
contradicting \eqref{teichmuller_formula}, and 
in the latter one we proceed as in \eqref{calculationOmega1}  (with $\omega_1$ replaced by $\omega_2$) to get a contradiction.
\end{itemize}

The only possibility left for $\gamma_1(a_1^{(2)})$ is to diverge  to $\infty^{(2)}$ in the asymptotic direction  $\theta_3^{(\infty)}$. Since it was already observed that   
$\gamma_3(b_1^{(2)})$ cannot diverge to $\infty^{(2)}$, it must intersect $\pi^{(-1)}(\R)$ to the right of $a_1$. 

The outcome of our analysis on the second sheet, as well as the region $\Omega_\varepsilon$ on the remaining sheets, is displayed in Figure~\ref{figure_traj_prel2_tauc_to_tau2}. 
The cut $\Delta_2$ is chosen in such 
a way that its projection on $\mathcal R_2$ coincides with $\left(\gamma_3(b_2^{(2)})\cup \gamma_1(a_2^{(2)})\right)\cap \mathcal R_2$.

\begin{figure}
\begin{subfigure}{.5\textwidth}
\centering
\begin{overpic}[scale=1]{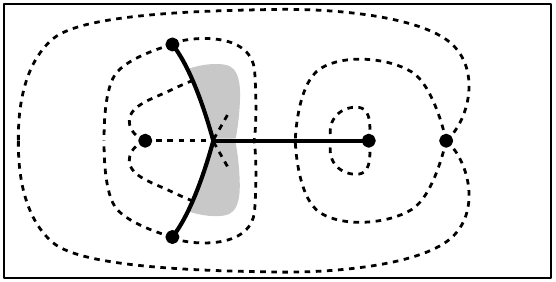}
\end{overpic}
\caption*{$\mathcal R_1$}
\vspace{0.5cm}
\end{subfigure}%
\begin{subfigure}{.5\textwidth}
\centering
\begin{overpic}[scale=1]{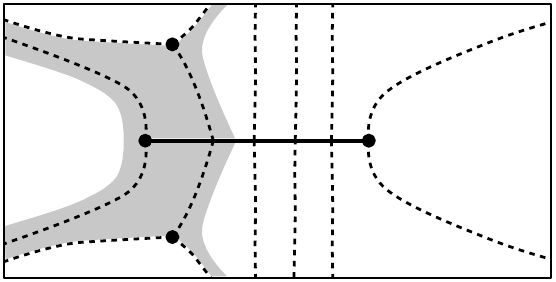}
\end{overpic}
\caption*{$\mathcal R_2$}
\vspace{0.5cm}
\end{subfigure}
\begin{subfigure}{1\textwidth}
\centering
\begin{overpic}[scale=1]{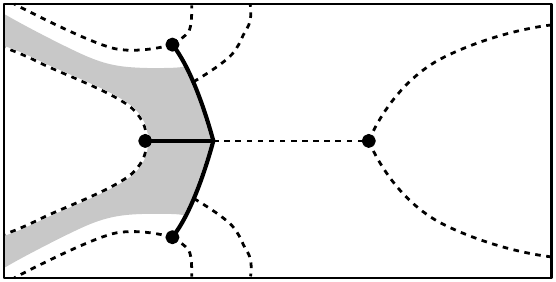}
\end{overpic} 
\caption*{$\mathcal R_3$}
\end{subfigure}
\caption{Part of the critical graph of $\varpi$ for $\tau_c<\tau<\tau_c+\delta$, with the domain $\Omega_\varepsilon$ in gray, consisting of the trajectories for $\varpi$ passing 
through the  $\varepsilon$-neighborhood of $a_1^{(2)}$, $a_2^{(2)}$ and $b_2^{(2)}$.}
\label{figure_traj_prel2_tauc_to_tau2}
\end{figure}

What is left is to describe the  behavior of $\gamma_3(b_2^{(2)})$ on the rest of the sheets. We already saw that this trajectory has to move to $\mathcal R_1$ through the cut 
$\Delta_1$. It cannot intersect $\Delta_1$ again (see Proposition~\ref{lemma:zerosofD}), and it must stay in  the region $\Omega_\varepsilon$ displayed in 
Figure~\ref{figure_traj_prel2_tauc_to_tau2}. Hence, the only possibility is that $\gamma_3(b_2^{(2)})$ intersects the cut $\Delta_2$, moves to the sheet $\mathcal R_3$ and 
diverges to $\infty^{(3)}$ in the asymptotic direction $\theta_4^{(\infty)}$. 

\begin{figure}
\begin{subfigure}{.5\textwidth}
\centering
\begin{overpic}[scale=1]{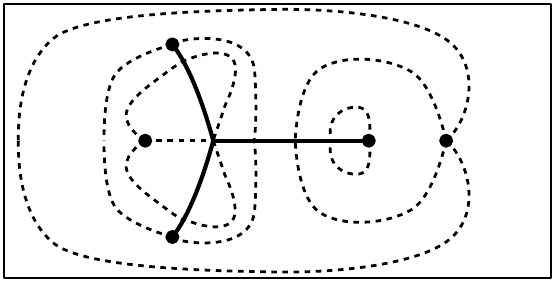}
\put(30,40){\scriptsize $S_1$}
\put(64.5,43){\scriptsize $S_3$}
\put(64.5,34.5){\scriptsize $S_2$}
\put(80,65){\scriptsize $S_4$}
\put(112,40){\scriptsize $S_5$}
\put(58,57){\scriptsize $S_7$}
\put(58,20){\scriptsize $S_6$}
\end{overpic}
\caption*{$\mathcal R_1$}
\vspace{0.5cm}
\end{subfigure}%
\begin{subfigure}{.5\textwidth}
\centering
\begin{overpic}[scale=1]{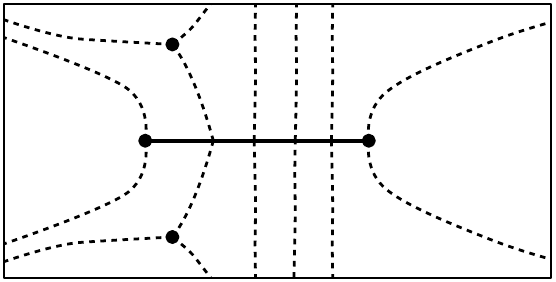}
\put(60,55){\scriptsize $S_2$}
\put(60,20){\scriptsize $S_3$}
\put(40,55){\scriptsize $S_6$}
\put(40,20){\scriptsize $S_7$}
\put(75,50){\scriptsize $S_4$}
\put(75,25){\scriptsize $S_4$}
\put(86,50){\scriptsize $S_5$}
\put(86,25){\scriptsize $S_5$}
\end{overpic}
\caption*{$\mathcal R_2$}
\vspace{0.5cm}
\end{subfigure}
\begin{subfigure}{1\textwidth}
\centering
\begin{overpic}[scale=1]{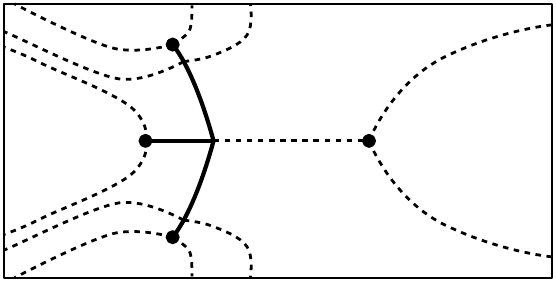}
\put(43,50){\scriptsize $S_7$}
\put(43,30){\scriptsize $S_6$}
\put(32,61){\scriptsize $S_3$}
\put(32,16){\scriptsize $S_2$}
\put(57,70){\scriptsize $S_1$}
\put(57,5){\scriptsize $S_1$}
\end{overpic} 
\caption*{$\mathcal R_3$}
\end{subfigure}
\caption{Critical graph of $\varpi$ for $\tau_c<\tau<\tau_2$, with the strip domains labeled by $S_j$'s.}\label{traj_final_tauc_to_tau2}
\end{figure}
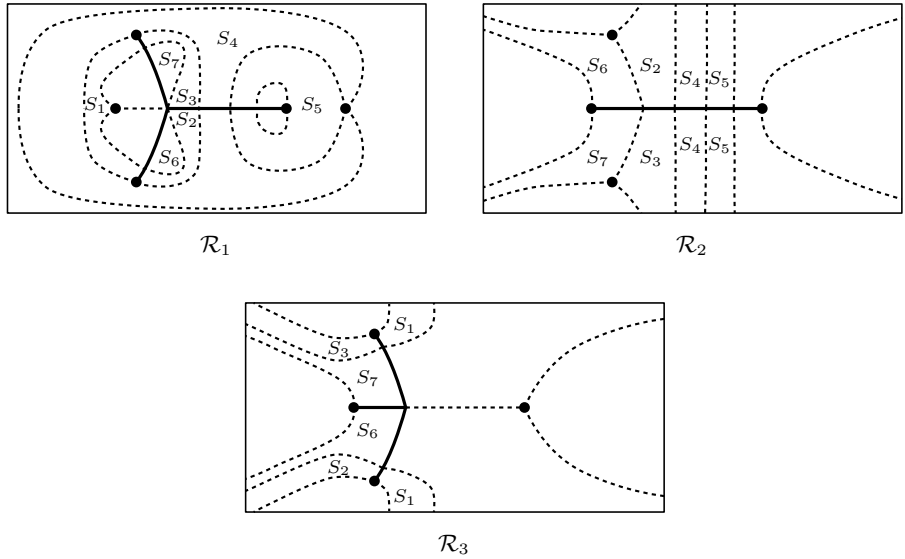

The critical graph of $\varpi$ for $\tau_c<\tau<\tau_c+\delta$ has the structure showed in Figure~\ref{traj_final_tauc_to_tau2}, and we now prove that this is actually valid for 
$\tau \in (\tau_c, \tau_2)$. Again, the continuity principle \textbf{P.4} yields that this is the case as long as (i) no collision of the critical points occur (this is true 
indeed 
for $\tau \in (\tau_c, \tau_2)$, see Section~\ref{section_critical_points}); (ii) finite trajectories for $\tau_c<\tau<\tau_c+\delta$ remain finite for $\tau 
\in (\tau_c, \tau_2)$ (assured by a combination of {\bf P.2} and {\bf P.6}) and (iii) no width of any strip 
domains become zero. These domains for $\tau_c<\tau<\tau_c+\delta$ are identified on Figure~\ref{traj_final_tauc_to_tau2}; observe that there are now 7 strip domains  $S_j$, and 
no 
ring domains. Let us compute their widths $\sigma(S_j)$.

\begin{lem}
	\label{lemmaaboutwidhts}
	For $\tau_c<\tau<\tau_2$,
	\begin{equation*}
		\sigma(S_1)=\sigma(S_2)=\sigma(S_3)=|\omega_1(\tau)|,
	\end{equation*}
	with $\omega(\tau)$ defined in \eqref{width_parameter_cubic}.
\end{lem}
\begin{proof}
The fact that $\sigma(S_1)=|\omega_1|$ is the straightforward consequence of the definition of $\omega_1$. Also $\sigma(S_2)=\sigma(S_3)$ by the symmetry under conjugation. So, it 
remains to compute the width $\sigma(S_2)$, given by the absolute value of the real part of the integral
\begin{align*}
\int_{b_2^{(2)}}^{a_2^{(1)}=a_2^{(3)}}\sqrt{-\varpi}& = \int_{b_2^{(2)}}^{a_1^{(2)}=a_1^{(3)}}Q(s)ds+\int_{a_1^{(3)}}^{a_2^{(3)}}Q(s)ds \\ 
& = \int_{b_2}^{a_1}(\xi_1(s)-\xi_3(s))ds+\int_{a_1}^{a_2}(\xi_1(s)-\xi_2(s))ds.
\end{align*}

Symmetry under conjugation tells us
$$
\re \int_{a_1}^{a_2}(\xi_1(s)-\xi_2(s))ds=\re \int_{a_1}^{b_2}(\xi_1(s)-\xi_2(s))ds,
$$
and using it in the previous identity, we get
\begin{align*}
\sigma(S_2) & =\left| \re \int_{b_2}^{a_1}(\xi_1(s)-\xi_3(s))ds - \re \int_{b_2}^{a_1}(\xi_1(s)-\xi_2(s))ds\right| \\
& = \left| \re \int_{b_2}^{a_1}(\xi_3(s)-\xi_2(s))ds \right| = |\omega_1|.
\end{align*}
\end{proof}

Lemma~\ref{lemmaaboutwidhts} implies that $\sigma(S_j)\neq 0$, $j=1, 2, 3$, for $\tau_c\leq \tau <\tau_2$, although $\sigma(S_j)\searrow 0$ as $\tau \nearrow \tau_2$, $j=1, 2, 3$, 
see Figure~\ref{figure_width_parameters}.

Regarding the rest of the strip domains,
\begin{itemize}

\item $\sigma(S_4)=|\omega_4|\neq 0$  for $\tau_c\leq \tau\leq \tau_2$, see again Figure~\ref{figure_width_parameters}.

\item $\sigma(S_5)=\left|\int_{b_1}^{b_*}(\xi_{2}(s)-\xi_3(s))ds \right|$, which does not vanish for $\tau> 1/12$, see \eqref{inequality_xi_5}.

\item From the symmetry by complex conjugation, $\sigma(S_6)=\sigma(S_7)$, and the structure of $S_6$ on $\mathcal R_2$ shows that $\sigma(S_6)=|\omega_2|$, so that $\sigma(S_6)= 
\sigma(S_7)\neq 0$  for 
$\tau_c<\tau\leq \tau_2$.
\end{itemize}

Consequently, the critical graph displayed in Figure~\ref{traj_final_tauc_to_tau2} is valid for $\tau_c<\tau<\tau_2$. In particular it implies Theorem~\ref{thm:criticaltraj} for 
$\tau_c\in (\tau_c,\tau_2)$.

\subsubsection{Trajectories for \texorpdfstring{$\tau_2<\tau<1/4$}{}}\label{section_trajectories_tau2_1-4}

Lemma~\ref{lemmaaboutwidhts} shows that at  $\tau=\tau_2$, the strip domains $S_1$, $S_2$ and $S_3$ displayed in Figure~\ref{traj_final_tauc_to_tau2} disappear 
\emph{simultaneously}, which happens because at that moment the branch point $a_1^{(1)}$ hits the critical trajectory $\gamma_1(a_2^{(1)})=\gamma_2(b_2^{(1)})$ and the point $a_*$ 
of intersection of $\Delta_2$ with the real line collides with the critical trajectories $\gamma_2(a_2^{(1)})$ and $\gamma_1(b_2^{(1)})$. The resulting critical graph is shown in 
Figure~\ref{traj_strip_tau3}. 

As before, we fix an $\varepsilon>0$  sufficiently small and consider the domain
$\Omega_\varepsilon$ swept by trajectories of $\varpi$ passing through points in the  $\varepsilon$-neighborhood of $a_1^{(1)}$, $a_2^{(1)}$, $b_2^{(1)}$ and $a_*^{(j)}$. Observe 
that for a small perturbation of $\tau=\tau_2$,  $\gamma_3(b_2^{(2)})\cap \mathcal R_2$ does not coalesce with critical points other than the starting point $b_2^{(2)}$, hence this 
arc of critical trajectory 
displays the same structure as for $\tau=\tau_2$. In particular, the cut $\Delta_2$ is well defined as the projection of $(\gamma_3(b_2^{(2)})\cup \gamma_1(a_2^{(2)}))\cap 
\mathcal R_2$ on the other sheets. Thus, there exists a $\delta>0$ such that the critical trajectories for $\tau_2<\tau <\tau_2+\delta$, passing through the above mentioned 
critical points, belong to $\Omega_\varepsilon$. This domain is  also depicted schematically on Figure~\ref{traj_strip_tau3}, left.

\begin{figure}
\begin{subfigure}{.5\textwidth}
\centering
\begin{overpic}[scale=1]{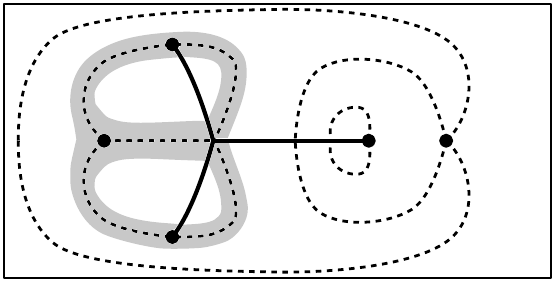}
\end{overpic}
\caption*{$\mathcal R_1$}
\vspace{0.5cm}
\end{subfigure}%
\begin{subfigure}{.5\textwidth}
\centering
\begin{overpic}[scale=1]{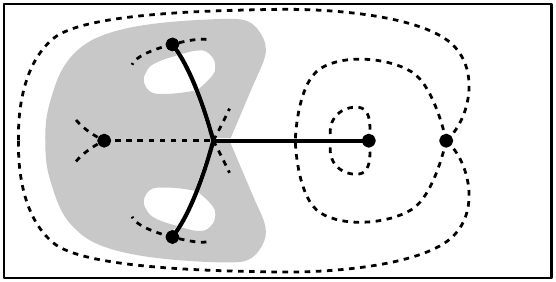}
\end{overpic}
\caption*{$\mathcal R_1$}
\vspace{0.5cm}
\end{subfigure}
\begin{subfigure}{.5\textwidth}
\centering
\begin{overpic}[scale=1]{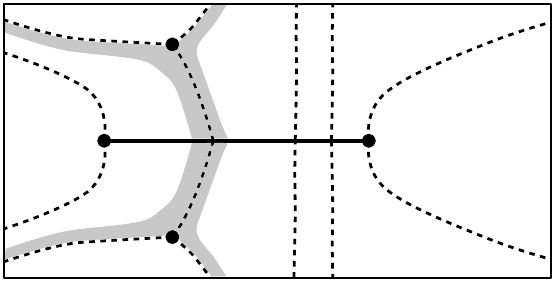}
\end{overpic}
\caption*{$\mathcal R_2$}
\vspace{0.5cm}
\end{subfigure}%
\begin{subfigure}{.5\textwidth}
\centering
\begin{overpic}[scale=1]{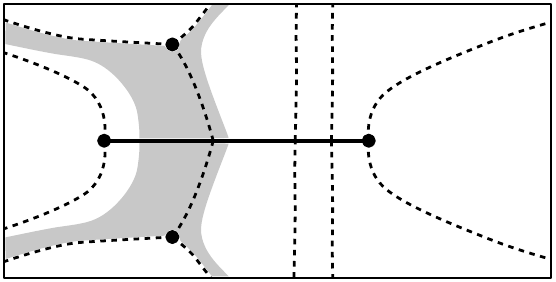}
\end{overpic}
\caption*{$\mathcal R_2$}
\vspace{0.5cm}
\end{subfigure}
\begin{subfigure}{.5\textwidth}
\centering
\begin{overpic}[scale=1]{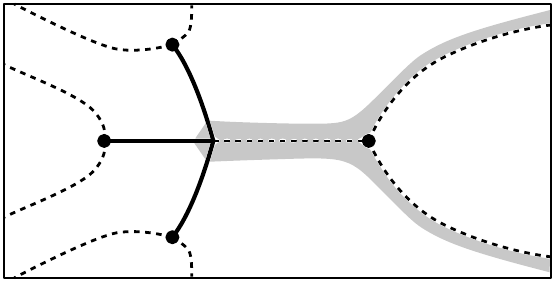}
\end{overpic}
\caption*{$\mathcal R_3$}
\vspace{0.5cm}
\end{subfigure}%
\begin{subfigure}{.5\textwidth}
\centering
\begin{overpic}[scale=1]{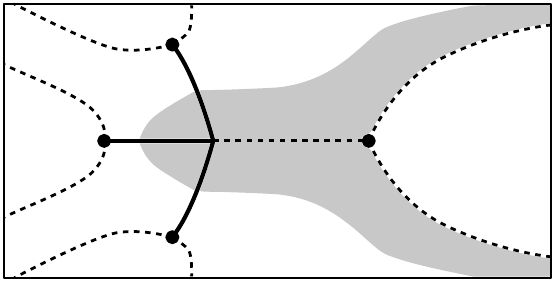}
\end{overpic}
\caption*{$\mathcal R_3$}
\vspace{0.5cm}
\end{subfigure}
\caption{Left: critical graph of $\varpi$ for $\tau=\tau_2$, with the domain $\Omega_\varepsilon$ in gray. Right: local behavior or the critical trajectories for $\varpi$ and $\tau 
=\tau_2+\delta>0$, passing through $a_1^{(1)}$, $a_2^{(1)}$, $b_2^{(1)}$ and $a_*$, with the same domain superimposed.}\label{traj_strip_tau3}
\end{figure}

\begin{lem}\label{lemma:zerocuvrelast}
	Let $t$ be a point on the part of the curve $\Delta_2^{(1)}$ joining $b_2^{(1)}$ with $a_*^{(1)}$, and $\gamma$ a Jordan curve on $\mathcal R_1$ connecting the boundary 
values $t_\pm$ on $\Delta_2^{(1)}$ and containing the only branch point $b_2^{(1)}$ inside. Then
	$$
	\re \int_{\gamma} Q(s)ds=\re \int_{\gamma} (\xi_{2+}(s)-\xi_{3+}(s))ds=0.
	$$
\end{lem}
\begin{proof}
	We can deform $\gamma$ to the cut $\Delta_2$. If we denote by $\Delta_2(t)$ the arc of $\Delta_2$ from $b_2$ to $t$, we get 
	\begin{align*}
	\re \int_{\gamma} Q(s)ds & = \re \int_{\Delta_2(t)} (\xi_{2+}(s)-\xi_{3+}(s))ds - \re \int_{\Delta_2(t)}(\xi_{2-}(s)-\xi_{3-}(s))ds   \\
	& = \re \int_{\Delta_2(t)} (\xi_{3-}(s)-\xi_{3+}(s))ds    = \re \int_{\Delta_2(t)} (\xi_{1+}(s)-\xi_{3+}(s))ds \\
	& =  \re \int_{\Delta_2^{(2)}(t)} Q_{+} ds =0,
	\end{align*}
where we have used the jump condition \eqref{equality_xi_3} and the fact that $\Delta_2^{(2)}$ is an arc of trajectory of $\varpi$. 
\end{proof}

Now we describe the critical trajectories for $\tau_2<\tau <\tau_2+\delta$, starting with  $\gamma_2(a_1^{(1)})$. Notice first that in this case $\omega_1\neq 0$, so that  
$\gamma_2(a_1^{(1)})$ cannot contain $b_2^{(1)}$. Furthermore, if $\gamma_2(a_1^{(1)})$ intersects the cut $\Delta_2$ (and thus diverges to $\infty^{(2)}$ in the asymptotic 
direction  $\theta_1^{(\infty)}$), then $\gamma_1(a_1^{(1)})\cup\gamma_2(a_1^{(1)})\cup 
\gamma_1(b_1^{(3)})$ determines a $\varpi$-polygon for which $\kappa=1$ and $\lambda=0$, in a contradiction with \eqref{teichmuller_formula}.

Keeping in mind that $\gamma_2(a_1^{(1)})$ must belong to $\Omega_\varepsilon$ we conclude that it has to intersect the real axis in one of the 
sheets. Let us denote by $x$ the point of the first intersection of this trajectory with $\pi^{-1}(\R)$. Again, we discard some cases:
\begin{itemize}
	\item $x$ cannot be on $\mathcal R_1$ to the left of $a_1$: this yields (by the general principle \textbf{P.2}) that $\gamma_2(a_1^{(1)})=\gamma_3(a_1^{(1)})$, violating  
\textbf{P.3}. 
	\item $x$ cannot be on $\mathcal R_1$ to the right of $a_*$. 
	
	Indeed, if $x>a_*$, we form a curve $\gamma$ given by the union of three pieces: $\gamma_1$ is the arc of $\gamma_1(a_1^{(1)})$ from $a_*^{(1)}$ to $a_1^{(1)}$ , 
$\gamma_2$ is the arc of $\gamma_2(a_1^{(1)})$ from $a_1^{(1)}$ to $x$, and $\gamma_3$ is the interval from $x$ to $a_*^{(1)}$. This curve satisfies the assumptions of 
Lemma~\ref{lemma:zerocuvrelast}, so that 
	$$
	\re \int_{\gamma} (\xi_{2+}(s)-\xi_{3+}(s))ds=0.
	$$
	Since 
	$$
	\re \int_{\gamma_1\cup \gamma_2} (\xi_{2+}(s)-\xi_{3+}(s))ds=0
	$$
	by the definition of trajectories, we conclude that
	$$
	\re \int_{a_*}^x (\xi_{2+}(s)-\xi_{3+}(s))ds=0,
	$$
	that is, $h(a_*, x)=0$, in the notation \eqref{eqRefirstsheet}. 	
	But if $x_3^{(1)}$ is the point of intersection of $\gamma_2(b_1^{(1)})$ with $\Delta_1$, we also have $h(x_3,b_1)=0$, and since $(a_*,x)\cap (x_3, b_1)=\emptyset$, this 
would contradict Proposition~\ref{lemma:zerosofD}.
	
	\item $x$ cannot be on $\mathcal R_2$ to the left of $a_*$ (in other words, $x$ cannot belong to $\Delta_3$, which could occur if the trajectory $\gamma_2(a_1^{(1)})$ had 
slipped to the second sheet through the cut $\Delta_2$ before hitting the real line): this yields 
	$$
	\re \int_{x}^{a_*} (\xi_{1+}(s)-\xi_{3+}(s))ds=0,
	$$
	that is, $h(x, a_*)=0$, in the notation \eqref{eqRefirstsheet}, with $x, a_*\in \Delta_3$.  	
	But if $x_3^{(1)}$ is the point of intersection of $\gamma_2(b_1^{(1)})$ with $\Delta_1$, we also have $h(x_3,b_1)=0$, in contradiction with 
Proposition~\ref{lemma:zerosofD}.
\end{itemize}

The only possibility left is that $\gamma_2(a_1^{(1)})$ hits the real line \emph{precisely} at the point $a_*$. This means also that $\gamma_2(a_1^{(1)})=\gamma_1(a_2^{(2)})$, as 
shown in Figure~\ref{traj_final_tau2_1-4}.
 
Regarding the trajectories $\gamma_1(b_2^{(1)})$ and $\gamma_2(b_2^{(1)})$, they cannot satisfy $\gamma_1(b_2^{(1)})=\gamma_2(b_2^{(1)})$, and due to principle {\bf P.3} they must 
belong to $\Omega_\varepsilon$. Hence, they have to behave as shown in Figure~\ref{traj_final_tau2_1-4}. We skip the details.

\begin{figure}
\begin{subfigure}{.5\textwidth}
\centering
\begin{overpic}[scale=1]{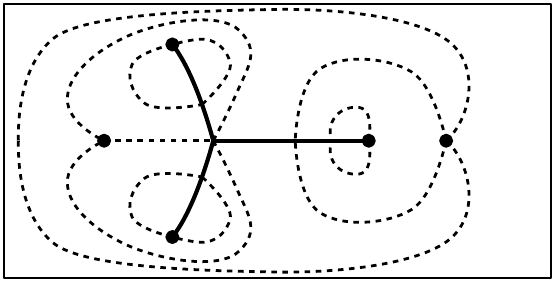}
\put(30,46){\scriptsize $S_1$}
\put(30,30){\scriptsize $S_2$}
\put(55,60){\scriptsize $S_3$}
\put(55,15){\scriptsize $S_4$}
\put(80,65){\scriptsize $S_6$}
\put(112,40){\scriptsize $S_5$}
\end{overpic}
\caption*{$\mathcal R_1$}
\vspace{0.5cm}
\end{subfigure}%
\begin{subfigure}{.5\textwidth}
\centering
\begin{overpic}[scale=1]{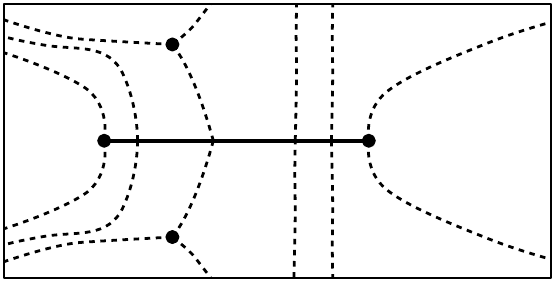}
\put(63,55){\scriptsize $S_6$}
\put(63,20){\scriptsize $S_6$}
\put(40,55){\scriptsize $S_2$}
\put(40,20){\scriptsize $S_1$}
\put(25,57){\scriptsize $S_4$}
\put(25,20){\scriptsize $S_3$}
\put(86,50){\scriptsize $S_5$}
\put(86,25){\scriptsize $S_5$}
\end{overpic}
\caption*{$\mathcal R_2$}
\vspace{0.5cm}
\end{subfigure}
\begin{subfigure}{1\textwidth}
\centering
\begin{overpic}[scale=1]{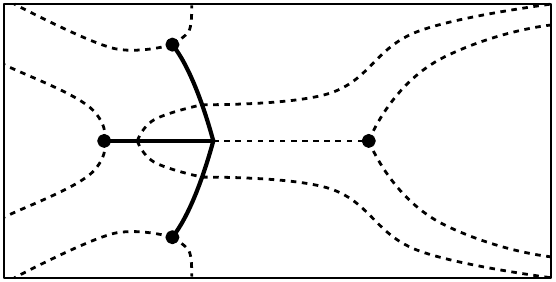}
\put(50,43){\scriptsize $S_1$}
\put(50,34){\scriptsize $S_2$}
\put(32,55){\scriptsize $S_3$}
\put(32,20){\scriptsize $S_4$}
\put(95,45){\scriptsize $S_1$}
\put(95,32){\scriptsize $S_2$}
\end{overpic} 
\caption*{$\mathcal R_3$}
\end{subfigure}
\caption{Critical graph of $\varpi$ for $\tau_2<\tau<1/4$, with the strip domains labeled by $S_j$'s.}\label{traj_final_tau2_1-4}
\end{figure}
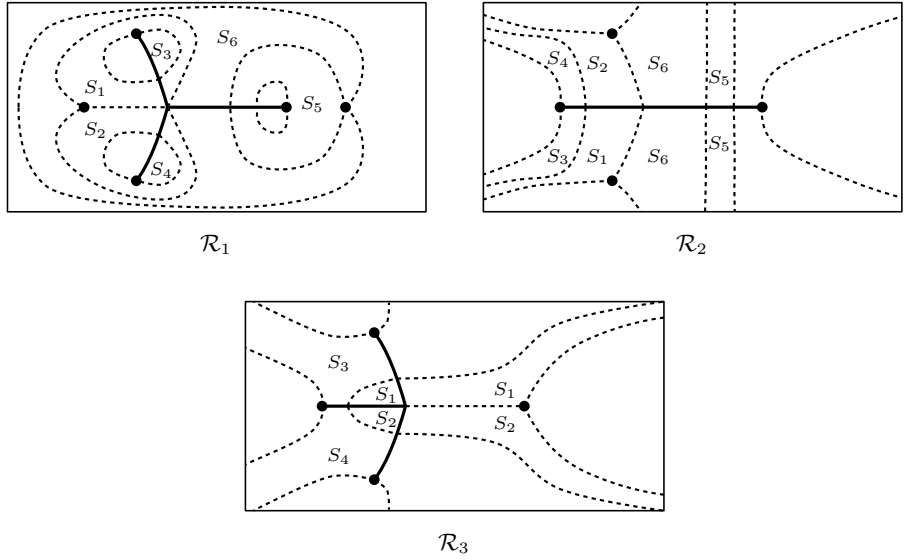

The critical graph of $\varpi$ for $\tau_2<\tau<\tau_2+\delta$ has the structure showed in Figure~\ref{traj_final_tau2_1-4}, and we prove that this is actually valid 
for $\tau \in (\tau_2, 1/4)$. The continuity principle \textbf{P.4} yields that this is the case as long as (i) no collision of the critical points occur (this is true indeed for 
$\tau \in (\tau_2, 1/4)$, see Section~\ref{section_critical_points}), (ii) the finite critical trajectories for  $\tau_2<\tau<\tau_2+\delta$ remain finite in the range 
$\tau\in (\tau_2,1/4)$ (certainly true by a combination of {\bf P.2} and {\bf P.6}) and (iii) no width of any strip domains become zero. These domains for 
$\tau_2<\tau<\tau_2+\delta$ are 
identified on Figure~\ref{traj_final_tau2_1-4}; observe that there are now 6 strip domains  $S_j$, and no ring domains. Their widths $\sigma(S_j)$ are:
\begin{itemize}
 \item $\sigma(S_1)=\sigma(S_2)=|\omega_1|\neq 0$ for $\tau \in (\tau_2, 1/4)$.
 
 \item $\sigma(S_3)=\sigma(S_4)=|\omega_3|\neq 0$ for $\tau \in (\tau_2, 1/4)$.
 
 \item $\sigma(S_5)=\left|\int_{b_1}^{b_*}(\xi_{2}(s)-\xi_3(s))ds \right|$, which does not vanish for $\tau> 1/12$, see \eqref{inequality_xi_5}.
 
 \item As in \eqref{omega1omega4}, we get that   $\sigma(S_6)=|\omega_4-\omega_1|$, and $\omega_1\neq \omega_4$ for $\tau>1/12$, see Figure~\ref{figure_width_parameters} in 
Section~\ref{section_widths_cubic}.
 
\end{itemize}

We conclude that the structure of the critical graph of $\varpi$, depicted in Figure~\ref{traj_final_tau2_1-4}, is actually valid for the whole range $\tau \in (\tau_2, 1/4)$. This 
finishes the proof of  Theorem~\ref{thm:criticaltraj}.

\begin{remark}
 The attentive reader might notice that in Figure~\ref{traj_final_tauc_to_tau2}, for instance, the critical trajectories $\gamma_1(a_2^{(2)})$ and 
$\gamma_2(a_1^{(1)})$ intersect the cut $\Delta_2$ in $\mathcal R_1$ on pairs of opposite points $t_{\pm}$. This phenomenon, which also occurs in 
Figure~\ref{traj_final_tau2_1-4}, is easily explained by Lemma~\ref{lemma:zerocuvrelast}.
\end{remark}

\begin{remark}\label{remark_orthogonal_trajectories}
 It follows from the results of Sections~\ref{section_trajectories_0_tau_1}--\ref{section_trajectories_tau2_1-4} that the trajectories $\gamma_1(b^{(2)}_2)$ and 
$\gamma_2(b^{(2)}_2)$ determine a half plane domain $H$. In particular, this implies that there is an orthogonal critical trajectory $\gamma_1$ emerging from $b_2^{(2)}$ which is 
entirely contained in $H$ and extends to $\infty^{(2)}$ along the angle $2\pi/3$. Similarly, there is an orthogonal critical trajectory $\gamma_2$ (which is the complex conjugate 
of $\gamma_1$) emerging from $a_2^{(2)}$ and extending to $\infty^{(2)}$ along the angle $-2\pi/3$. Then the projected contour 
$$
\Gamma=\pi(\gamma_1\cup\gamma_2)\cup \Delta_2
$$
satisfies the conditions stated in \eqref{condition_trajectory_1}--\eqref{extension_support_mu_2_property_2}. 
\end{remark}

\appendix 

\section{Proof of Theorem~\ref{thm_reciprocal_spectral_curve} in the general case}
\label{appendix1}

The proof of Theorem~\ref{thm_reciprocal_spectral_curve} presented in Section~\ref{section_general_variational} is valid when the set $\widehat S_\alpha\subset S_\alpha$ of double 
poles of the coefficient $D$ in \eqref{spectral_curve_general} is empty (see Proposition~\ref{lemma_variations_without_double_poles}). Our goal now is to show that for a vector of 
measures $\vec\mu\in\mathcal M_\alpha$ whose components are supported on a finite union of analytic arcs, and such the associated functions $\xi_j$ in 
\eqref{shifted_resolvents_pols_V} satisfy   \eqref{spectral_curve_general} for a polynomial $R$ and a rational function $D$, the equality
\begin{equation}\label{varconditionAppendix}
\mathcal D_{h}(\vec \mu)=0
\end{equation}
is valid for every function $h\in C^2(\C)$, without any further restriction on the poles of $D$. Recall that this fact was established so far for the Cauchy kernels $h_z$ defined 
in \eqref{shiffer_variation}, and when $
\supp h\cap \widehat S_\alpha =\emptyset$, see Proposition~\ref{lemma_variations_without_double_poles}.

As a first step, we extend \eqref{varconditionAppendix} to polynomials:
\begin{lem}\label{lemma_variations_polynomial}
	Under the assumptions of Theorem~\ref{thm_reciprocal_spectral_curve},  $\mathcal D_q(\vec\mu)=0$ for every algebraic polynomial $q$.
\end{lem}
\begin{proof}

	Fix $\rho>0$ for which 
	$$
	\supp\mu_1\cup\supp\mu_2\cup\supp\mu_3\subset V_\rho:=\{x\in\C \; \mid \; |x|<\rho\}.
	$$
	
	For $x\in V_\rho$ and $z$ sufficiently large, we can expand
	$$
	h_{z}(x)=-\sum_{j=0}^\infty \frac{x^j}{z^{j+1}},
	$$
	which implies the identity
	\begin{equation}\label{aux_equation_10}
	p_{n-1}(x,z)-z^{n+1} h_{z}(x)= x^n+w_z(x),
	\end{equation}
	where the function
	$$
	w_z(x)=\sum_{j=n+1}^\infty \frac{x^j}{z^{j-n}}
	$$
	converges to $0$ as $z\to \infty$ uniformly for $x\in V_\rho$, and $p_{n-1}(x,z)$ is a polynomial of degree $n-1$ in $x$, given 
	explicitly by
	$$
	p_{n-1}(x,z)=-\sum_{j=0}^{n-1}x^j z^{n-j}.
	$$
	
	Set $p_{-1}\equiv 0$. If
	$$
	q(x)=\sum_{n=0}^N a_nx^n
	$$
	is a polynomial of degree $N$, then \eqref{aux_equation_10} shows that 
	\begin{equation}\label{identity_pols_rest}
	q(x)=-zq(z)h_z(x) + Q_z(x)+W_z(x),
	\end{equation}
	where
	$$
	Q_z(x)=\sum_{n=0}^N a_np_{n-1}(x,z),\quad W_z(x)=-\sum_{n=0}^{N} a_nw_z(x).
	$$
	Note that $Q$ is a polynomial of degree $N-1$ in $x$.
	
	Moreover, $W_z$ and $W_z'$ both converge to $0$ as $z\to\infty$ uniformly for $x\in V_\rho$. Hence, the convergences
	$$
	\frac{W_z(x)-W_z(y)}{x-y}\to 0, \quad   \Phi'(x)W_z(x)\to 0 \quad \mbox{ as }z\to\infty
	$$
	hold uniformly for $x,y\in V_\rho$, and from the definition of $\mathcal D_h$ in \eqref{variations_energy} we get
	$$
	\mathcal D_{W_z}(\vec\mu)\to 0 \quad \mbox{ as } z\to\infty.
	$$
	
	The quantity $\mathcal D_h(\vec \mu)$ is linear in $h$, so \eqref{aux_equation_10} implies
	$$
	\mathcal D_q(\vec \mu)=-zq(z)\mathcal D_{h_z}(\vec \mu)+ \mathcal D_{Q_z}(\vec\mu) +\mathcal D_{W_z}(\vec \mu) =\mathcal D_{Q_z}(\vec\mu) +\mathcal D_{W_z}(\vec \mu),
	$$
	where for the last equality we used Corollary~\ref{corollary_converse_critical_measure}, and hence from the previous limit we get
	$$
	\mathcal D_q(\vec \mu)=\lim_{z\to \infty} \mathcal D_{Q_z}(\vec\mu).
	$$
	
	The result now follows easily by induction on the degree $N$ of $q$. If $N=0$, then $Q_z$ is identically zero, and the equality above implies that $\mathcal D_q(\vec \mu)=0$. 
	Assuming now that $\mathcal D_h(\vec\mu)=0$ for every polynomial $h$ of degree at most $N-1$, we get that $\mathcal D_{Q_z}(\vec \mu)=0$ because $Q_z(x)$ is a polynomial of degree 
	at most $N-1$ in $x$, and the equality above implies $\mathcal D_q(\vec\mu)=0$, concluding the proof. 
\end{proof}

If $p\in \widehat S_\alpha$, then as it is discussed in Remark~\ref{remark:Angelesco-Nikishin}, the point $p$ belongs to exactly two of the supports of 
$\mu_1$, $\mu_2$ and $\mu_3$, and locally the union of these sets is an analytic arc. That is, there exists an open disk $U_p$ centered at $p$ such that
\begin{equation}\label{local_support_double_poles}
\gamma_p=U_p\cap (\supp\mu_1\cup \supp\mu_2\cup \supp\mu_3)
\end{equation}
is an analytic arc passing through $p$. 

Additionally, given any point $z \in (\supp\mu_1\cup\supp\mu_2\cup\supp\mu_3)\setminus \widehat S_\alpha$, there exists a small disk $B_z$ centered $z$, disjoint from 
$\widehat S_\alpha$ 
and such that the set 
$$
B_z\cap (\supp\mu_1\cup \supp\mu_2\cup\supp\mu_3)
$$
is a finite union of analytic arcs, which can only intersect at the common point $z$. In case $z\notin S_\alpha$, this intersection reduces to a single analytic arc. The collection
$$
\{B_z\}\cup \{U_p\}
$$
constructed above is an open cover of the compact set $\supp\mu_1\cup\supp\mu_2\cup\supp\mu_3$, from which we extract a finite subcover
$$
\left\{B_j\right \}_{j=1}^m\cup \left\{U_p\right\}_{p\in \widehat S_p},
$$
where for $j=1,\hdots,m$ we have $U_j=U_z$ for some  $z=z_j \in (\supp\mu_1\cup\supp\mu_2\cup\supp\mu_3)\setminus \widehat S_\alpha$. Set 
$$
B=\bigcup_{j=1}^m B_j,\quad U=\bigcup_{p\in \widehat S_\alpha}U_p.
$$
It follows from their construction that these sets satisfy
\begin{equation}\label{decomposition_supports}
B\cap \widehat S_\alpha=\emptyset,\quad \supp\mu_1\cup\supp\mu_2\cup\supp\mu_3\subset B\cup U.
\end{equation}

Consider a smooth partition of unity $\{\psi_k\}$ of $\supp\mu_1\cup\supp\mu_2\cup\supp\mu_3$ subordinated to the open cover $B\cup U$. That is, each function $\psi_k$ is 
real, belongs to $C_\infty(\C)$ and additionally satisfies the following properties.

\begin{itemize}
	\item $0\leq \psi_k(z)\leq 1$, for every $z\in\C$.
	\item For every $k$, $\supp\psi_k\subset U$ or $\supp\psi_k \subset B$.
	\item Every $z\in \supp\mu_1\cup\supp\mu_2\cup\supp\mu_3$ belongs to the support of a finite number of functions in the collection $\{\psi_k\}$.
	\item $\sum_{k} \psi_k(z)=1$, for every $z\in \supp\mu_1\cup\supp\mu_2\cup\supp\mu_3$.
\end{itemize}

Since $\supp\mu_1\cup\supp\mu_2\cup\supp\mu_3$ is compact, the collection $\{\psi_k\}$ can be assumed to be finite. Moreover, we can refine $\{\psi_k\}$ and assume that 
$\supp\psi_k\subset U$ whenever $\supp\psi_k\cap \widehat S_\alpha\neq \emptyset$. Set  
$$
\widehat \psi (z)=\sum_{\supp\psi_k\cap \widehat S_\alpha\neq \emptyset}\psi_k(z),\quad \psi(z)=\sum_{k} \psi_k(z)-\widehat \psi(z).
$$
The functions $\psi$ and $\widehat \psi$ belong to $C_\infty(\C)$, satisfy
\begin{equation}\label{conditions_support_psis}
\supp\psi\cap \widehat S_\alpha=\emptyset, \quad \supp\widehat\psi\subset U
\end{equation}
and 
\begin{equation}\label{splitting_supports}
\psi(z)+\widehat \psi(z)=1, \quad z\in \supp\mu_1\cup\supp\mu_2\cup\supp\mu_3.
\end{equation}

\begin{lem}\label{lemma_variations_polynomial_local}
	Under the conditions of Theorem~\ref{thm_reciprocal_spectral_curve}, if $q$ is a polynomial, then $\mathcal D_{\widehat \psi q}(\vec\mu)=0$, where $\widehat \psi$ is the function 
	constructed above.
\end{lem}
\begin{proof}
	From \eqref{splitting_supports}, it follows that $q\equiv \psi q+\widehat \psi q$ on $\supp\mu_1\cup\supp\mu_2\cup\supp\mu_3$. From the definition of $\mathcal 
	D_h$ in \eqref{variations_energy} we then get 
	$$
	\mathcal D_{\widehat \psi q} (\vec\mu)=\mathcal D_q(\vec\mu) - \mathcal D_{\psi q}(\vec\mu).
	$$

	Using the first condition in \eqref{conditions_support_psis}, we get $\supp(\psi q)\cap \widehat S_\alpha=\emptyset$, so 
	Lemma~\ref{lemma_variations_without_double_poles} gives us $\mathcal D_{\psi q}(\vec\mu)=0$. Since $q$ is a polynomial, we learn from Lemma~\ref{lemma_variations_polynomial} 
	that $\mathcal D_q(\vec\mu)=0$, concluding the proof. 
\end{proof}

We are finally able to prove Theorem~\ref{thm_reciprocal_spectral_curve} in its full generality.
\begin{proof}[Proof of Theorem~\ref{thm_reciprocal_spectral_curve}]
	Recall the definition of the arcs $\{\gamma_p\}$, $p\in \widehat S_\alpha$, given in \eqref{local_support_double_poles}. Each of these arcs is a simple contour on the complex 
	plane, and we can find a smooth arc $\gamma\subset \C$ for which $\cup \gamma_p\subset \gamma$ and $\C\setminus \gamma$ is connected. In particular, the second condition in 
	\eqref{conditions_support_psis} implies
	\begin{equation}\label{equation_1_lemma_analytic_arc}
	\supp (\widehat \psi )\cap (\supp\mu_1\cup \supp\mu_2\cup \supp\mu_3) \subset \bigcup_{p\in \widehat S_\alpha}\gamma_p \subset \gamma.
	\end{equation}
	
	Consider a parametrization $\gamma:[0,1]\to \C$ of $\gamma$ by arc length, set $a=\gamma(0)$ and define a continuous function $g:\gamma\to \C$ by $g(z)=\gamma'(\gamma^{-1}(z))$. 
	That is, $g(z)$ is a unit vector tangent to $\gamma$ at the point $z$, varying continuously with $z$. Given $h\in C^2(\C)$, define
	$$
	G:\gamma \to \C,\quad G(z)=\frac{1}{g(z)}\left( \frac{\partial h}{\partial x}(z)\re g(z)+\frac{\partial h}{\partial y}(z)\im g(z) \right)
	$$
	
	$\gamma$ is a simple smooth arc, so it has empty interior. Since $\C\setminus \gamma$ is connected and $G$ is continuous, Mergelyan's Theorem tells 
	us that there exists a sequence of polynomials $(p_n)$ converging to $G$ uniformly on $\gamma$. In particular, for $\gamma(t)=z$ this implies that the convergence
	$$
	p_n(\gamma(t))\gamma'(t)=p_n(z)g(z)\to \frac{\partial h}{\partial x}(z)\re g(z)+\frac{\partial h}{\partial y}(z)\im g(z)=\frac{d}{dt}(h(\gamma(t))) \Big|_{t=\gamma^{-1}(z)}
	$$
	holds uniformly on $\gamma$, and as a consequence the sequence of polynomials

	$$
	q_n(z)=\int_a^z p_n(s)ds +h(a)
	$$
	converges to 
	$$
	\int_0^t\frac{d}{du}(h(\gamma(u)))du+h(a)=h(z)
	$$
	uniformly for $z\in \gamma$.
	
	In summary, we constructed a sequence of polynomials $(q_n)$ converging uniformly to $h$ on $\gamma$, and for which the sequence of derivatives $(q_n')=(p_n)$ converges to $G$ 
	uniformly on $\gamma$; in particular there exists $M>0$, independent of $n$, such that
	\begin{equation}\label{uniform_bound_derivatives}
	\left|q'_n(x)\right| \leq M,\quad x\in \gamma, \quad n\geq 1.
	\end{equation}
	
	Hence, the convergence
	\begin{equation}\label{convergence_polynomial_to_test_function}
	\widehat \psi(x)q_n(x)\to  \widehat \psi(x)h(x)
	\end{equation}
	holds true uniformly along $\gamma$. Due to \eqref{equation_1_lemma_analytic_arc}, this is enough to conclude that the convergence above holds uniformly on 
	$\supp\mu_1\cup\supp\mu_2\cup\supp\mu_3$, so that
	\begin{equation}\label{equation_2_lemma_analytic_arc}
	\int \Phi'_j \widehat \psi q_n \; d\mu_j \to  \int \Phi'_j \widehat \psi h \; d\mu_j,\quad j=1,2,3.
	\end{equation}
	
	The measures $\mu_1,\mu_2$ and $\mu_3$ do not have point masses, so the diagonal $\{x=y\}$ has zero $\mu_j\times \mu_k$ measure. Thus the limit 
	\eqref{convergence_polynomial_to_test_function} also implies that the convergence
	\begin{equation}\label{convergence_divided_difference}
	\frac{\widehat \psi(x)q_n(x)-\widehat \psi(y)q_n(y)}{x-y}\to \frac{\widehat \psi(x)h(x)-\widehat \psi(x)h(y)}{x-y} \\
	\end{equation}
	holds true pointwise $(\mu_j\times \mu_k)$-a.e.
	
	Since the arc $\gamma$ is connected, we also know that
	$$
	\left|\frac{\widehat \psi(x)q_n(x)-\widehat \psi(y)q_n(y)}{x-y}\right| \leq \sup_{s\in [0,1]} \left| \frac{d}{dt}\left( \left((\widehat \psi q_n )\circ \gamma \right)(t) \right)
	\right|_{t=s}
	$$
	whenever $x,y\in \gamma$. In virtue of \eqref{uniform_bound_derivatives} and \eqref{convergence_polynomial_to_test_function}, the right-hand side in the inequality above is 
	uniformly bounded in $n$. Hence the left-hand side of 
	\eqref{convergence_divided_difference} is uniformly bounded along $\gamma$, and using \eqref{equation_1_lemma_analytic_arc} and once again 
	\eqref{convergence_polynomial_to_test_function}, we can extend this conclusion to $\supp\mu_1\cup\supp\mu_2\cup\supp\mu_3$. Using the 
	Dominated Convergence Theorem 
	and \eqref{convergence_divided_difference} we conclude 
	$$
	\iint \frac{\widehat \psi(x)q_n(x)-\widehat \psi(y)q_n(y)}{x-y} d\mu_j(x)d\mu_k(y)\to \iint \frac{\widehat \psi(x)h(x)-\widehat \psi(y)h(y)}{x-y} d\mu_j(x)d\mu_k(y),
	$$
	for $j,k=1,2,3$. Combined with \eqref{equation_2_lemma_analytic_arc}, we finally get
	$$
	\mathcal D_{\widehat \psi q_n}(\vec\mu )\to \mathcal D_{\widehat \psi h}(\vec\mu ).
	$$
	
	From Lemma~\ref{lemma_variations_polynomial_local}, we know that $\mathcal D_{\widehat \psi q_n}(\vec\mu )=0$ for every $n$, hence $\mathcal D_{\widehat \psi h}(\vec\mu )=0$. On 
	the other hand, due to the first condition in \eqref{conditions_support_psis}, we have $\supp(\psi h)\cap \widehat S_\alpha=\emptyset$, so from
	Lemma~\ref{lemma_variations_without_double_poles} we get $\mathcal D_{ \psi h}(\vec\mu )=0$. Thus,
	$$
	\mathcal D_{h}(\vec\mu )=\mathcal D_{\widehat \psi h}(\vec\mu )+\mathcal D_{\psi h}(\vec\mu )=0,
	$$
	where for the first equality we used \eqref{splitting_supports}. Since $h\in C^2(\C)$ is arbitrary, Corollary~\ref{corollary_critical_measure} gives us that the measure $\vec \mu$ 
	is critical, concluding the proof.
\end{proof}


\section{Quadratic differentials}\label{appendix_quadratic_differentials}

A meromorphic quadratic differential $\varpi$ on a Riemann surface $\mathcal R$ is a differential form of type $(2,0)$, given locally by an expression $f(z)dz^2$, where $f$ is a 
meromorphic function of a local coordinate $z$. If $z=z(\zeta)$ is a conformal change of variables, then
$$
\widetilde{f}(\zeta)d\zeta^2 = f(z(\zeta)) (dz/d\zeta)^2 d\zeta^2
$$
represents $\varpi$ in the local coordinates $\zeta$. 

In this Appendix, we sketch the minimal background on quadratic differentials used throughout the paper. The general references are the monographs by 
Strebel \cite{strebel_book} and Jenkins \cite{jenkins_book}; some additional information can be found in \cite{MR1929066,Pommerenkebook,MR2724628}.

\subsection{Critical points and trajectories}

The {\it critical} or {\it singular points} of $\varpi=fdz^2$ are the zeros and poles of $f$; recall that a zero (resp., a pole) of $\varpi$ is a point $p$ where in a local chart 
sending $p$ to $0$ we have $f(z) = z^n \psi(z)$, with $\psi(0)\neq 0$, and with the integer $n\geq 1$ (resp., $n\leq -1$). The value $n$ is the \textit{order} of the critical point 
$p$, and is denoted by $\eta(p)$.   The rest of the points of $\mathcal R$ are called regular, and their order is assumed to be $\eta(p)=0$.

Critical points of order $\leq -2$ (i.e., poles of order 2 and higher) are called {\it infinite}, and the rest of the critical points are {\it  finite}. 

In a neighborhood of any regular point $p$, the primitive 
\begin{equation}\label{local_primitive}
\Upsilon(z)=\int_p^z \sqrt{-\varpi}=\int_p^z\sqrt{-f(s)}ds
\end{equation}
is well defined by specifying the branch of the square root at $p$ and continuing it analytically along the path of integration. Function $\Upsilon(z)$ provides a 
\textit{distinguished} or a \textit{natural} parameter on $\mathcal R$ in a neighborhood of $p$. 

We are mostly interested in the {\it trajectories} of a quadratic differential $\varpi$. A Jordan arc $\gamma\subset \mathcal R$ is called an {\it arc of trajectory} of
$\varpi$ if it is locally mapped by $\Upsilon$ to a vertical line. More precisely, this means that for any point $p\in \gamma$, there exists a neighborhood $U$ where the primitive 
$\Upsilon$ above is well defined and satisfies
\begin{equation}\label{definition_local_trajectory}
\re \Upsilon(z)=\const,\quad z\in \gamma\cap U.
\end{equation}
A maximal arc of trajectory is called a {\it trajectory} of $\varpi$.

Analogously, the {\it orthogonal trajectories} of $\varpi$ are 
trajectories of $-\varpi$; they can be equivalently defined by replacing ``$\re$'' by ``$\im$'' in \eqref{definition_local_trajectory}.

A trajectory $\gamma$ extending to a finite critical point along at least one of its directions is called {\it critical}; in the case when it happens in both directions, we call 
this trajectory {\it bounded} (also {\it finite} or \textit{short}), and {\it 
unbounded} (or {\it infinite}) otherwise. Notice that both ends of a short trajectory may coincide, in which case it forms a loop on $\mathcal R$.

A \textit{$\varpi$-chain} is a connected set on $\mathcal R$ made of a finite union of arcs of trajectories or orthogonal trajectories of $\varpi$. If no curves in a 
$\varpi$-chain 
belong to orthogonal trajectories,  we refer to it as a {\it path of trajectories} of $\varpi$ (or a \textit{$\varpi$-path}). In this case, in order to avoid the trivial 
situation, 
two consecutive arcs of a $\varpi$-path are required to intersect at a singular point of $\varpi$.

\subsection{The local structure of trajectories} \label{sec:localstructuretrajectoriesAppendix2}

The local behavior of trajectories of a meromorphic quadratic differential $\varpi$ is well understood.

From a point $p$ of order $\eta(p)=n\geq -1$ emanate $n+2$ trajectories, forming equal angles $\frac{2\pi}{n+2}$ at $p$. This covers also regular points, meaning that through any 
regular point passes exactly 
one trajectory, which is locally an analytic arc  (see Figure \ref{figure_traj_locally_1}). 

\begin{figure}
\begin{subfigure}{.33\textwidth}
\centering
\begin{overpic}[scale=1]{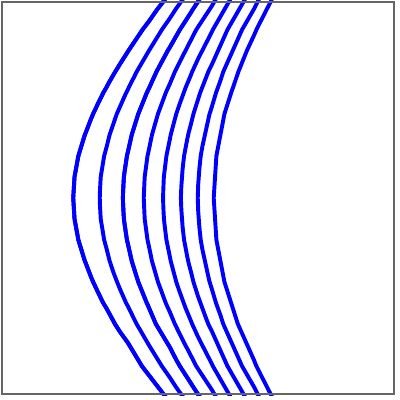}
\end{overpic}
\end{subfigure}%
\begin{subfigure}{.33\textwidth}
\centering
\begin{overpic}[scale=1]{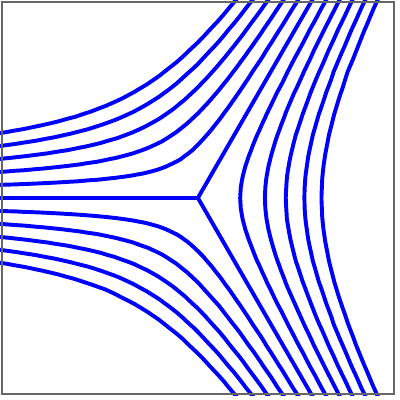}
\end{overpic}
\end{subfigure}%
\begin{subfigure}{0.33\textwidth}
\centering
\begin{overpic}[scale=1]{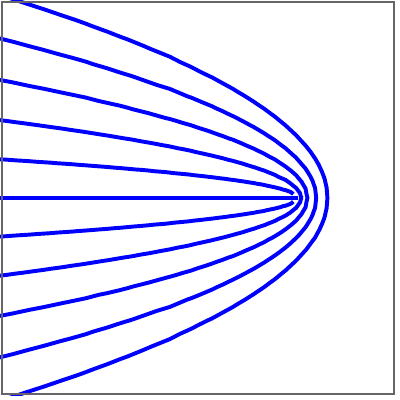}
\end{overpic}
\end{subfigure}
\caption{Structure of trajectories in a neighborhood of a regular point (left), simple zero (middle) and simple pole (right).}\label{figure_traj_locally_1}
\end{figure}

An infinite critical point $p$ or order $n\leq  -3$ has a neighborhood $G$ with the following property: there are $-(n+2)$ asymptotic directions, henceforth called 
{\it critical directions}, forming equal angles $\frac{2\pi}{-n-2}$ at $p$, such that each trajectory entering $G$ stays in $G$ and tends to $p$ in one of the critical 
directions \cite[Theorem 3.3]{jenkins_book}. If a trajectory is fully contained in $G$, then it tends to $p$ in two consecutive critical directions (see Figure 
\ref{figure_traj_locally_2}).

\begin{figure}
\begin{subfigure}{.5\textwidth}
\centering
\begin{overpic}[scale=1]{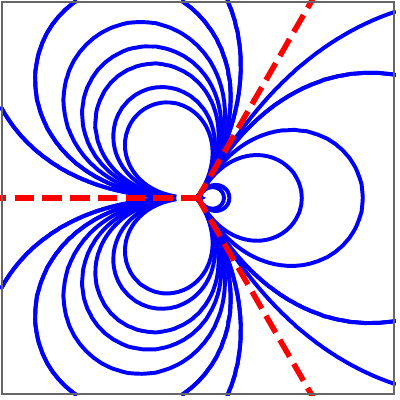}
\end{overpic}
\end{subfigure}%
\begin{subfigure}{.5\textwidth}
\centering
\begin{overpic}[scale=1]{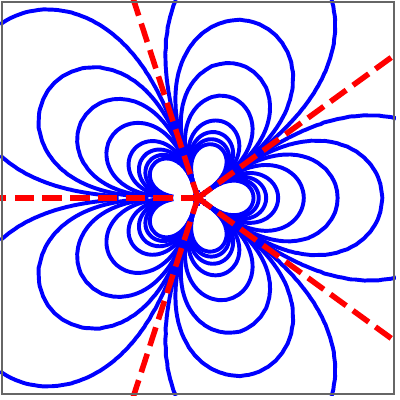}
\end{overpic}
\end{subfigure}
\caption{Structure of trajectories (solid lines) and critical directions (dashed lines) in a neighborhood of poles of order $3$ (left) and $5$ 
(right).}\label{figure_traj_locally_2}
\end{figure}

At a double pole $p$ there are three possibilities. For $\varpi=f(z)dz^2$, we define the {\it residue} $c$ of $\varpi$ at $z=p$ to be the residue of $\sqrt{f(z)}$ at $z=p$, 
which is well defined up to a sign. If $c\in \R$ then there are no trajectories emanating from $p$ and the trajectories near $p$ are closed loops. If $c\in i\R$, then there are 
trajectories emanating from $c$ in every direction. In the rest of the cases, the trajectories near $p$ converge to $p$ in a spiral form (see Figure \ref{figure_traj_locally_3}).

\begin{figure}
\begin{subfigure}{.33\textwidth}
\centering
\begin{overpic}[scale=1]{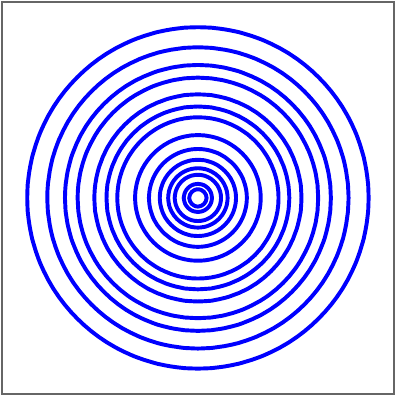}
\end{overpic}
\end{subfigure}%
\begin{subfigure}{.33\textwidth}
\centering
\begin{overpic}[scale=1]{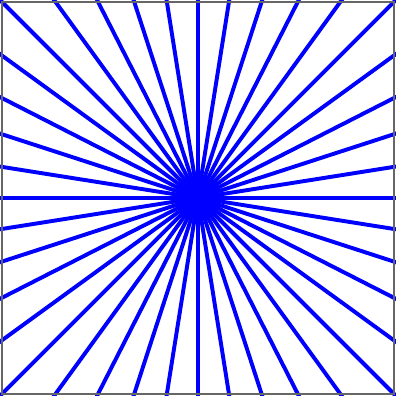}
\end{overpic}
\end{subfigure}%
\begin{subfigure}{0.33\textwidth}
\centering
\begin{overpic}[scale=1]{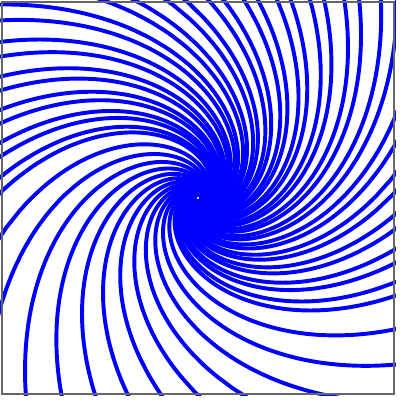}
\end{overpic}
\end{subfigure}
\caption{Structure of trajectories in a neighborhood of a double pole for $c\in \R$ (left), $c\in i\R$ (middle) and $c\in \C\setminus(\R\cup i\R)$ 
(right).}\label{figure_traj_locally_3}
\end{figure}

\subsection{Global structure of trajectories} \label{sec:globalstructuretrajectoriesAppendix2}

There are three possible behaviors for a trajectory $\gamma$ in the large,
\begin{enumerate}[(i)]
\item $\gamma$ is a closed curve containing no critical points.

\item $\gamma$ is an arc connecting two critical points (which may coincide; in this case $\gamma$ is a closed curve).

\item $\gamma$ is an arc that has no limit along at least one of its directions.
\end{enumerate}

Trajectories satisfying {\it (ii)} are called {\it short} or {\it finite}. Trajectories satisfying {\rm (iii)}  are called {\it recurrent}, and they are usually a major source of 
troubles when studying the global structure of trajectories of a given quadratic differential. Fortunately (for us)  in this paper we  deal with quadratic differentials with at 
most $3$ poles on a genus $0$ compact Riemann surface, and the absence of recurrent trajectories in this case is assured by Jenkin's Three Poles Theorem \cite[Thm.~8.5, 
page~226]{Pommerenkebook}.

It is intuitively clear that a non-recurrent and non-closed trajectory $\gamma$ has two limiting (extremal) values when we travel it in both opposite directions, and that can be 
distinct or equal. For convenience, we will denote these extremal values by $p(\gamma)$ and $q(\gamma)$. For instance, for a short trajectory $\gamma$ both $p(\gamma)$ and 
$q(\gamma)$ are finite critical points.

The set of the critical trajectories of a quadratic differential $\varpi$ (together with their limit points, i.e.~the critical points of $\varpi$) is the {\it critical graph} of 
$\varpi$, denoted by $\mathcal G=\mathcal G_\varpi$. 
According to \cite[Theorem 3.5]{jenkins_book} (see also \cite[\S 10]{strebel_book}), the complement of the
closure of $\mathcal G_\varpi$ in $\mathcal R$ consists of a finite number of domains called the \textit{domain configuration} of $\varpi$. 
The 
knowledge of $\mathcal G$ (or of the domain configuration of $\varpi$) is sufficient to fully understand the global structure of trajectories of $\varpi$, as evidenced by the 
following theorem: 
\begin{thm}[{Basic Structure Theorem \cite[Theorem~3.5]{jenkins_book}}]\label{theorem_global_dissection}
Let $\varpi$ be a meromorphic quadratic differential on a compact Riemann surface $\mathcal R$. Suppose in addition that $\varpi$ has no recurrent trajectories.

Then $\mathcal R\setminus\mathcal G$ decomposes into a finite union of disjoint domains $\cup \mathcal D$, each of them bounded by a finite number of critical 
trajectories. Each domain $\mathcal D$ lies into one of the following four classes. 

\begin{enumerate}[\it (i)]
 \item {\bf Half plane} (or \textbf{end}) \textbf{domain}: It is swept by trajectories converging to a pole $p$ of order $\geq 3$ in its two ends, and along consecutive critical 
directions. Its boundary consists of a $\varpi$-path with two unbounded critical trajectories and a finite number of short trajectories. For some choice of the branch of the 
square root, the natural parameter $\Upsilon$ in \eqref{local_primitive} is a conformal map from $\mathcal D$ to a vertical half plane 
 $$
 H=\{w\in\C \; \mid \; \re w>c\},
 $$
 for some $c\in \R$, and it extends continuously to the boundary of $\mathcal D$ with the identification $\Upsilon(p)=\infty$.
 
 \item {\bf Strip domain}: 
 It is swept by trajectories which both ends tend to poles $p$, $q$ of order $\geq 2$, possibly with $p=q$. The boundary $\partial \mathcal D\setminus \{p,q\}$ is a disjoint union 
of two $\varpi$-paths, each of them consisting of two unbounded critical trajectories converging to $p$ and $q$, and possibly a finite number of short trajectories.

 For some real constants $c_1<c_2$, $\Upsilon$ maps $\mathcal D$ conformally to a vertical strip
 \begin{equation}\label{uniformization_strip_domain}
 S=\{ w\in\C \; \mid \; c_1<\re w < c_2 \},
 \end{equation}
 and it extends continuously to the boundary of $\mathcal D$, with appropriate identification of the points $p,q$ with the directions $\pm i\infty$.

 \item {\bf Ring domain}: 
 It is swept by closed trajectories. Its boundary consists of two connected components, where each of them is a closed $\varpi$-path. For a suitably chosen real constant $c$ and 
some real numbers $0<r_1<r_2$, the function $z\mapsto e^{c\Upsilon(z)}$ maps $\mathcal D$ conformally to 
an annulus
 \begin{equation}\label{uniformization_ring_domain}
 R=\{ w\in\C \; \mid \; r_1<|w|<r_2 \}.
 \end{equation}
and it extends continuously to the boundary of $\mathcal D$.
 
 \item {\bf Circle domain}: It is swept by closed trajectories and contains exactly one double pole, with purely real residue. Its boundary is a 
closed $\varpi$-path.

For a suitably chosen real constant 
$c$ and some real number $r>0$, the function $z\mapsto e^{c\Upsilon(z)}$ is a conformal map from $\mathcal D$ to the circle centered at origin and radius $r$; it extends 
continuously to $\partial D$ and sends the double pole to the origin $w=0$.

\end{enumerate}
\end{thm}

In case $\varpi$ has also recurrent trajectories, a fifth class of domains has to be added to the domain configuration of $\varpi$; we refer the reader to \cite{jenkins_book} for 
further details.

For a given short trajectory $\gamma$ connecting two finite critical points $p$, $q$ (which coincide if $\gamma$ is closed) we define its {\it length} by
\begin{equation}\label{definition_length}
\ell(\gamma)=\left|\int_\gamma \sqrt{-\varpi}\right|=\left|\im \int_\gamma \sqrt{-\varpi}\right|>0,
\end{equation}

The {\it width  of a strip domain} $\mathcal S$ is defined as $\sigma(\mathcal S)=|c_2-c_1|$, where the constants $c_1,c_2$ are as in \eqref{uniformization_strip_domain}. 
Alternatively, it can be computed as
\begin{equation}\label{definition_width}
\sigma(\mathcal S) = \left| \int_{p}^q \re \sqrt{-\varpi} \right|,
\end{equation}
where we integrate along any path in $\mathcal S$ connecting two points, $p$ and $q$, lying on different connected components of $\partial \mathcal S\setminus \{\mbox {poles of 
$\varpi$}\}$.

In the same spirit, the {\it width of a ring domain} $\mathcal D$ is defined as 
\begin{equation*}
\sigma(\mathcal D)=\frac{1}{c}\log\frac{r_2}{r_1}= \left| \re\int_{p}^q \sqrt{-\varpi} \right|,
\end{equation*} 
where $r_1,r_2$ and $c$ are as in \eqref{uniformization_ring_domain}, and we integrate along any path in $\mathcal D$ connecting two points, $p$ and $q$,  lying on different 
connected components of its boundary.

For any open simply connected domain $\mathcal D\subset \mathcal R$ bounded by a $\varpi$-chain (that is called a \emph{$\varpi$-polygon}) we define two values: the total order of 
the singular points of $\varpi$ in $\mathcal D$ is
\begin{equation*}
\lambda(\mathcal D)= \sum_{p_j\in \mathcal D}\eta(p_j),
\end{equation*}
where the summation is along all the singular points $p_j$ of   $\varpi$ in $\mathcal D$, while the  contribution from the singular points of $\varpi$ on the boundary is
\begin{equation*}
\kappa(\mathcal D)= \sum_{p_j\in \partial \mathcal D} \beta(p_j),
\end{equation*}
where the summation is along all the corners $p_j$ of  $\partial \mathcal D$, 
$$
\beta(p)=1-\theta(p_j) \,\frac{\eta(p)+2}{2\pi}, 
$$
and $\theta(p)$ is the inner angle (in radians) at the corner $p$.

Both values have a simple relation, as shown by the following simple consequence of the argument principle (also known as the Teichm\"uller lemma, see \cite[Theorem 
14.1]{strebel_book}):
\begin{thm}[Teichm\"uller lemma]\label{teichmuller_lemma}
If $\mathcal D$ is a $\varpi$-polygon, then
\begin{equation}\label{teichmuller_formula}
\kappa(\mathcal D) = 2+\lambda(\mathcal D).
\end{equation}
\end{thm}
A straightforward corollary of this lemma (and also a direct consequence of the maximum principle for holomorphic functions) is the following fact:
\begin{cor}\label{corol_teichmuller_lemma}
 If $\varpi$ is analytic (has no poles) in a $\varpi$-polygon $\mathcal D$, then $\partial \mathcal D$ must contain at least one pole of $\varpi$.
\end{cor}
This corollary is also the basis for the general principle \textbf{P.3} in Section~\ref{sec:generalprinciples}.

\section{Numerical experiments}\label{appendix_numerics}

In this Appendix we look under the hood of the calculation of the functions $\omega_j$ in \eqref{width_parameter_cubic}, as well as of the  numerical procedures used to find and 
plot the trajectories of the quadratic differential $\varpi$ from Section~\ref{section:dynamics}.

As it was mentioned,  Figure~\ref{figure_width_parameters} was obtained by means of a numerical evaluation of the integrals defining the functions $\omega_j$ in 
\eqref{width_parameter_cubic}. For that, we compute the integrals of the form
$$
\int_{z_1}^{z_2}\xi_j(s)ds
$$
along the line segment joining chosen roots $z_1$, $z_2$ of the discriminant in \eqref{defDiscriminant} by means of the composite trapezoidal rule. 

For $\tau$ fixed, we compute the points $z_k=z_k(\tau)$, choose a value $m\in \N$, and consider a grid of $m+1$ equally spaced nodes $\{p_j\}_{j=1}^{m+1}$,
$$
p_j=z_1 + \frac{j-1}{m}(z_2-z_1),\quad j=1,\cdots,m+1,
$$
that will be used as the quadrature nodes for the composite trapezoidal rule. At each $z=p_j$ we solve \eqref{spectral_curve} numerically, obtaining an (unordered) set of three 
solutions $\xi_{k} (p_j)$. Comparing their values at consecutive 
quadrature points, we collect them into a sequence of (ordered) vectors $\vec v_j=(\xi_{\sigma(1)} ,\xi_{\sigma(2)} ,\xi_{\sigma(3)} )^T(p_j)$, $j=1,\hdots,m+1$, where $\sigma$ 
is a permutation of $\{1,2,3\}$ that does not depend on $j$. We can then determine the permutation $\sigma$ using
\cref{coinciding_nodes,xi_at_branchpoints1,xi_at_branchpoints2,xi_at_branchpoints3,xi_at_branchpoints4} as boundary conditions.

This procedure is repeated for $n$ equally spaced values of $\tau$,
$$
\tau = \tau_k = \frac{k}{4(n+1)},\quad k=1,\dots,n,
$$
(notice that $\tau_k\in [0,1/4)$; we stop at $k=n$ in order to avoid the degenerate situation at $\tau=1/4$ for which some of the endpoints of integration diverge to $\infty$). 

The result of these calculations with $n=1000$ and $m=10000$ is plotted in Figure~\ref{figure_width_parameters}.

On the other hand, we also performed numerical experiments that helped us to build the intuition to predict (and confirm) the structure of the critical graphs in Section 
\ref{section:dynamics}. A sample of such graphs is presented in \cref{numerics_critical_graph_1,numerics_critical_graph_2,numerics_critical_graph_3}.

\begin{figure}
	\begin{subfigure}{0.33\textwidth}
		\centering
		\begin{overpic}[scale=1]{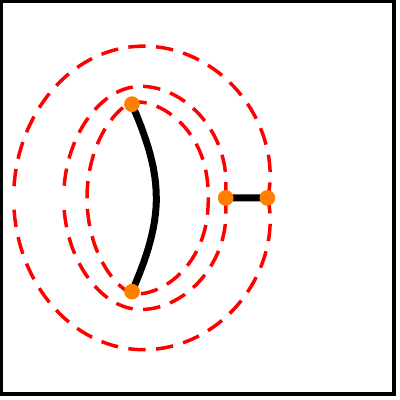}
		\end{overpic}
		\caption*{$\mathcal R_1$}
	\end{subfigure}%
	\begin{subfigure}{0.33\textwidth}
		\centering
		\begin{overpic}[scale=1]{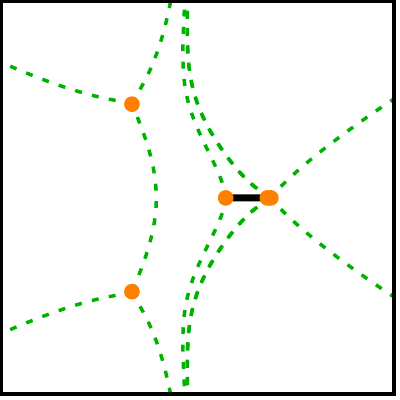}
		\end{overpic}
		\caption*{$\mathcal R_2$}
	\end{subfigure}%
	\begin{subfigure}{0.33\textwidth}
		\centering
		\begin{overpic}[scale=1]{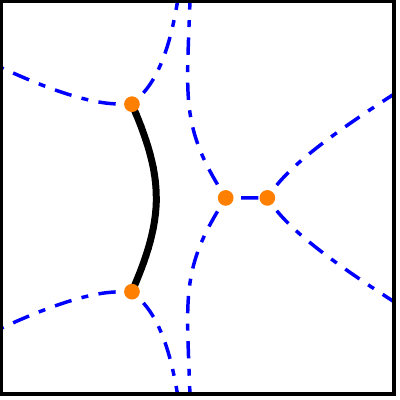}
		\end{overpic} 
		\caption*{$\mathcal R_3$}
	\end{subfigure}
	\caption{Numerical evaluation of the critical graph for $\tau=1/25$.}\label{numerics_critical_graph_1}
\end{figure}


\begin{figure}
	\begin{subfigure}{0.33\textwidth}
		\centering
		\begin{overpic}[scale=1]{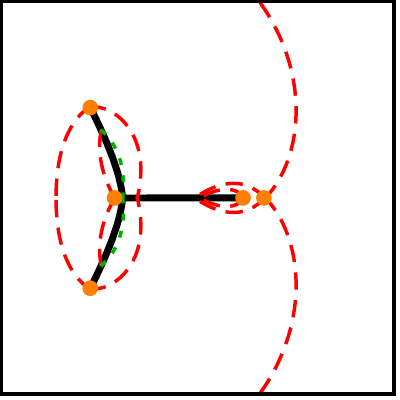}
		\end{overpic}
		\caption*{$\mathcal R_1$}
	\end{subfigure}%
	\begin{subfigure}{0.33\textwidth}
		\centering
		\begin{overpic}[scale=1]{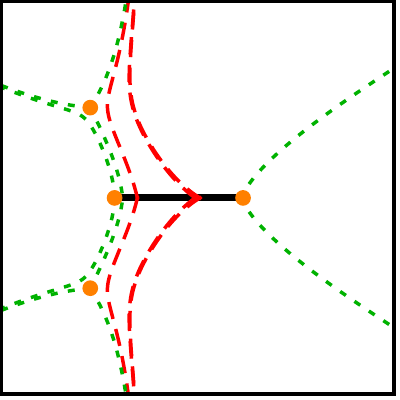}
		\end{overpic}
		\caption*{$\mathcal R_2$}
	\end{subfigure}%
	\begin{subfigure}{0.33\textwidth}
		\centering
		\begin{overpic}[scale=1]{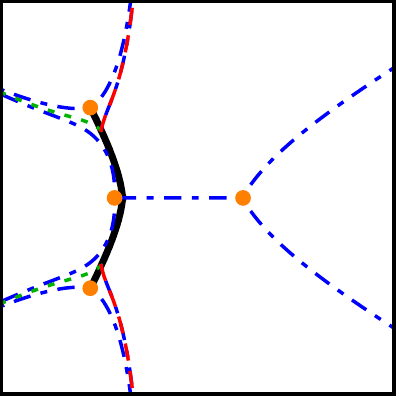}
		\end{overpic} 
		\caption*{$\mathcal R_3$}
	\end{subfigure}
	\caption{Numerical evaluation of the critical graph for $\tau=1/5$.}\label{numerics_critical_graph_2}
\end{figure}


\begin{figure}
	\begin{subfigure}{0.33\textwidth}
		\centering
		\begin{overpic}[scale=1]{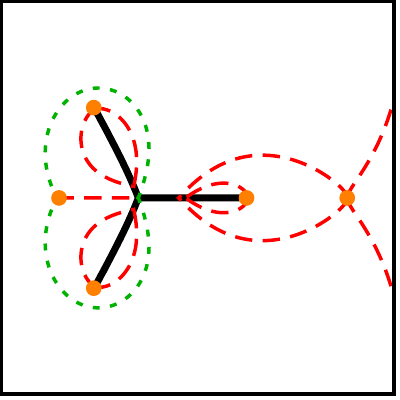}
		\end{overpic}
		\caption*{$\mathcal R_1$}
	\end{subfigure}%
	\begin{subfigure}{0.33\textwidth}
		\centering
		\begin{overpic}[scale=1]{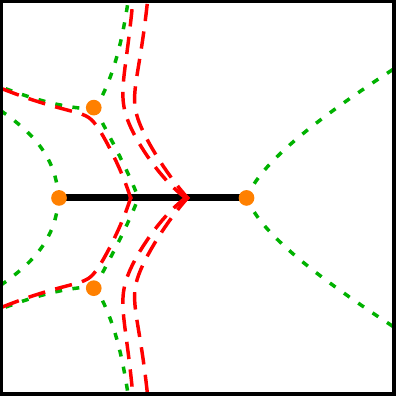}
		\end{overpic}
		\caption*{$\mathcal R_2$}
	\end{subfigure}%
	\begin{subfigure}{0.33\textwidth}
		\centering
		\begin{overpic}[scale=1]{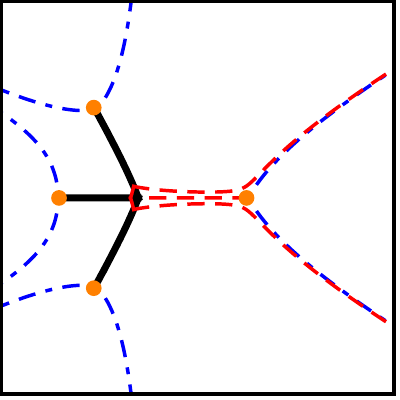}
		\end{overpic} 
		\caption*{$\mathcal R_3$}
	\end{subfigure}
	\caption{Numerical evaluation of the critical graph for $\tau=1/5$.}\label{numerics_critical_graph_3}
\end{figure}


The numerical procedure used for these pictures is as follows. If $\gamma(t), t\in J\subset \R,$ is a parameterization of a trajectory of $\varpi=Q^2(z)dz^2$, then
$$
Q^2(\gamma(t))\gamma'(t)=if(t),\quad t\in J,
$$
where $f$ is some real valued function. A different choice of $f$ correspond just to a reparameterization of $\gamma$. In a natural (arc-length) parametrization of $\gamma$ the 
equation 
above takes the form
\begin{equation}\label{ODE}
\gamma'(t)=i\frac{\overline{Q^2(\gamma(t))}}{|Q^2 (\gamma(t))|},\quad t\in J.
\end{equation}
This is an ordinary differential equation for $\gamma$ that can be solved by standard methods (e.g.~the family of Runge-Kutta algorithms). 

We should bewared of two aspects when implementing this method. The first one concerns the initial values, that usually are at branch points of \eqref{spectral_curve}, and thus, at 
zeros of $Q$, from where more than one trajectory emanates. This can be resolved by perturbing the initial point along the prescribed direction, which can be obtained from a local 
analysis of $Q$ at the singular point.

The second aspect concerns the choice of branches of $Q^2$ in \eqref{ODE}, where we take advantage of the fact that all the branch points of \eqref{spectral_curve} are quadratic. 
The value of $Q^2$ is found numerically by solving \eqref{spectral_curve}, which gives us the three possible values $\xi_k$. Near a branch point two of these values are close 
(corresponding to the two solutions coinciding at this branch point), allowing us to distinguish the third solution. This, in turn, allows to recognize on the the possible branches 
of $Q^2$, but not the other two.

For instance, at $z=b_2$ we will have $\xi_1 \approx \xi_3$ (see \eqref{xi_at_branchpoints4}), which singles out the branch  $\xi_2$ and makes the branch $Q^2=(\xi_1-\xi_3)^2$ 
easily distinguishable. In order to identify the remaining two branches (i.e., $(\xi_2-\xi_1)^2$ and  $(\xi_2-\xi_3)^2$) we must use further results about the local and global 
structure of the critical graph of $\varpi$, established in Section~\ref{section:dynamics}.


\section*{Acknowledgements}
Part of this work was carried out during our visit to the Department of Mathematics of the Vanderbilt University, USA,
as well as during a stay of the first author as a Visiting Chair Professor at the Department of Mathematics of the Shanghai Jiao Tong University (SJTU), China. We gratefully acknowledge the hospitality and the work environment of the host institutions.

We also thank 
L.~Zhang for explaining his article \cite{filipuk_vanassche_zhang}, which was our starting motivation, E.~A.~Rakhmanov for his interest, and A.B.J.~Kuijlaars 
for a careful reading of the first version of this manuscript and for his useful suggestions. In particular, we give him credit for the Example~\ref{example_arno}.

The first author was partially supported by the Spanish Government together with the European Regional Development Fund (ERDF) under grants
MTM2011-28952-C02-01 (from MICINN) and MTM2014-53963-P (from MINECO), by Junta de Andaluc\'{\i}a (the Excellence Grant P11-FQM-7276 and the research group FQM-229), and by Campus 
de Excelencia Internacional del Mar (CEIMAR) of the University of Almer\'{\i}a. The second author was supported by FWO Flanders projects G.0641.11 and G.0934.13.

\newpage

\bibliographystyle{amsplain}

\end{document}